\documentclass[10pt]{amsart}
\usepackage{amssymb}
\usepackage{bm}
\usepackage{graphicx}
\usepackage[all]{xy}

\newtheorem{thm}{Theorem}[section]
\newtheorem{cor}[thm]{Corollary}
\newtheorem{lem}[thm]{Lemma}
\newtheorem{prop}[thm]{Proposition}
\theoremstyle{definition}

\theoremstyle{remark}

\numberwithin{equation}{section}

\begin{document}

\title[Integrality of hypergeometric series with quadratic parameters]
{Criterion for the integrality of hypergeometric series with parameters from quadratic fields}%
\author[S.F. Hong]{Shaofang Hong}
\address{Mathematical College, Sichuan University, Chengdu 610064, P.R. China}
\email{sfhong@scu.edu.cn; s-f.hong@tom.com; hongsf02@yahoo.com}
\author[C.L. Wang]{Chunlin Wang}
\address{Center for Combinatorics, Nankai University, Tianjin 300071, P.R. China}
\email{wdychl@126.com}
\thanks{The research was supported partially by National Science
Foundation of China Grant \#11371260.}
\address{}%
\email{}%

\thanks{}%

\begin{abstract}
For the hypergeometric series with parameters from the rational fields,
there is an effective criterion due to Christol to decide whether the
hypergeometric series is N-integral or not. Christol criterion is a basic
and vital tool in the recent striking work of Delaygue, Rivoal and Roques
on the N-integrality of the hypergeometric mirror maps with rational
parameters. In this paper, we develop a systematic theory on the
N-integrality of the hypergeometric series with parameters from
quadratic fields. We first present a detailed $p$-adic analysis
to set up a criterion of the $p$-adic integrality of the 
hypergeometric series with parameters from rational fields. 
Consequently, we present two equivalent statements for the 
hypergeometric series with parameters from algebraic number 
fields to be N-integral. Finally, by using these results, 
introducing a new function that extends the Christol's function 
and developing a further $p$-adic analysis, we establish a criterion 
for the N-integrality of the hypergeometric series with parameters
from the quadratic fields. In the process, there are two important
ingredients. One is the uniform distribution result of roots of a
quadratic congruence which is due to Duke, Friedlander and Iwaniec
together with Toth. Another one is an upper bound on the number
of solutions of polynomial congruences given by Stewart in 1991. 
\end{abstract}
\maketitle
\section{Introduction}
Let $r$ and $s$ be positive integers. As usual, let $\mathbb{Q}, \mathbb{R}$
and $\mathbb{C}$ denote the fields of rational numbers, real numbers and
complex numbers, respectively. Let $\mathbb{Z}$ be the ring of integers.
For a prime number $p$, let $\mathbb{Q}_p$ and $\mathbb{Z}_p$ denote the
field of $p$-adic numbers and the ring of $p$-adic integers, respectively.
For given $\bm{\alpha}=(\alpha_1,...,\alpha_r)$ and $\bm{\beta}=(\beta_1,...,\beta_s)$
with $\alpha_i,\beta_j\in\mathbb{Q}\setminus\mathbb{Z}_{\le0}$ and
$\alpha_i\ne \beta_j$ for all integers $i$ and $j$ with $1\le i\le r$
and $1\le j\le s$, the {\it hypergeometric series associated with
$(\bm{\alpha},\bm{\beta})$}, denoted by $F_{\bm{\alpha},\bm{\beta}}(z)$,
is defined to be the formal power series
\begin{equation} \label{Eq1.1}
F_{\bm{\alpha},\bm{\beta}}(z):=\sum_{n=0}^{\infty}\frac{(\alpha_1)_n\cdots(\alpha_r)_n}
{(\beta_1)_n\cdots(\beta_s)_n}z^n,
\end{equation}
where $(x)_n$ is the {\it Pochhammer symbol} with $(x)_0:=1$
and for all integers $n$ with $n\ge 1$,
$$(x)_n:=x(x+1)\cdots(x+n-1).$$

If $\beta_s=1$, then $F_{\bm{\alpha,\beta}}(z)$ satisfies
\begin{equation} \label{Eq1.2}
L_{\bm{\alpha,\beta}}(F_{\bm{\alpha,\beta}}(z))=0,
\end{equation}
where $L_{\bm{\alpha,\beta}}$ is a linear differential operator defined by
$$L_{\bm{\alpha,\beta}}:=\prod_{j=1}^{s}\Big(z\frac{d}{dz}+\beta_j-1\Big)
-z\prod_{i=1}^{r}\Big(z\frac{d}{dz}+\alpha_i\Big). $$
Let
$$G_{\bm{\alpha,\beta}}(z):=\sum_{n=0}^{\infty}\frac{(\alpha_1)_n\cdots(\alpha_r)_n}
{(\beta_1)_n\cdots(\beta_s)_n}\Big(\sum_{i=1}^{r}H_{\alpha_i}(n)
-\sum_{j=1}^{s}H_{\beta_j}(n)\Big)z^n,$$
where for any $x\in\mathbb{C}\setminus\mathbb{Z}_{\le 0}$, one has
$$H_x(n):=\sum_{k=0}^{n-1}\frac{1}{x+k}.$$
Furthermore, if $\beta_{s-1}=1$, then
$G_{\bm{\alpha,\beta}}(z)+\log z F_{\bm{\alpha,\beta}}(z)$
is another solution of (\ref{Eq1.2}). Let
\begin{equation} \label{Eq1.3}
q_{\bm{\alpha,\beta}}(z):=\exp\Big(\frac{G_{\bm{\alpha,\beta}}(z)(z)
+\log z F_{\bm{\alpha,\beta}}(z)(z)}{F_{\bm{\alpha,\beta}}(z)}\Big)
=z\exp\Big(\frac{G_{\bm{\alpha,\beta}}(z)}
{F_{\bm{\alpha,\beta}}(z)}\Big).
\end{equation}
The power series $q_{\bm{\alpha,\beta}}(z)$ appears in the study
of mirror maps associated with the hypergeometric differential
equations (\ref{Eq1.2}). The arithmetic properties of $F_{\bm{\alpha,\beta}}(z)$ and
$q_{\bm{\alpha,\beta}}(z)$ have been extensively studied.
For example, Beukers and Heckman \cite{[BH]} investigated all algebraic
hypergeometric functions and Christol \cite{[Ch],[Ch2]} conjectured that
all globally bounded hypergeometric functions are diagonal.
Landau \cite{[La]} and Bober \cite{[Bo]} discussed the relationship between
the integrality of Taylor coefficients of certain hypergeometric series
and the integrality of values of certain step functions. We say that
a power series $F(z)\in 1+z\mathbb{Q}[[z]]$ is {\it N-integral} if there
exists a nonzero rational number $c$ such that $F(cz)\in\mathbb{Z}[[z]]$.
One would easily find that $F_{\bm{\alpha},\bm{\beta}}(z)$ is N-integral
if and only if for all but finitely many prime numbers $p$, one has
$F_{\bm{\alpha},\bm{\beta}}(z)\in\mathbb{Z}_p[[z]]$
(cf. Proposition 22 of \cite{[DRR]}). By using this fact, in 1986,
Christol \cite{[Ch]} was able to find a criterion
to decide for which pairs of $(\bm{\alpha},\bm{\beta})$ such that
$F_{\bm{\alpha},\bm{\beta}}(z)$ is $N$-integral.
Before stating Christol's ceriterion, we introduce some notation.

For all $x\in\mathbb{R}$, let $\langle x\rangle$ be the unique
number in $(0,1]$ such that $x-\langle x\rangle\in\mathbb{Z}$.
By this definition, one is easy to find that $x-\langle x\rangle$
is the largest integer satisfying $x-\langle x\rangle<x$.
Christol \cite{[Ch]} defined "$\preccurlyeq$" to be a total
order on $\mathbb{R}$ by $x\preccurlyeq y$
if and only if either $\langle x\rangle<\langle y\rangle$ or
$\langle x\rangle=\langle y\rangle$ and $x\ge y$.
For $\bm{\alpha}=(\alpha_1,...,\alpha_r)$ and $\bm{\beta}=(\beta_1,...,\beta_s)$,
let $d_{\bm{\alpha},\bm{\beta}}$ be the least common multiple of the
exact denominators of components of $\bm{\alpha}$ and $\bm{\beta}$.
For all positive integers $a$ coprime to $d_{\bm{\alpha},\bm{\beta}}$
and all $x\in\mathbb{R}$, one defines
\begin{equation} \label{Eq1.4}
\delta_{\bm{\alpha},\bm{\beta}}(x,a):=\#\{1\le i\le r: a\alpha_i\preccurlyeq x\}
-\#\{1\le j\le s: a\beta_j\preccurlyeq x\}.
\end{equation}
As in \cite{[DRR]}, we call the function
$\delta_{\bm{\alpha},\bm{\beta}}(x,a)$ {\it Christol
function}. Then Christol's criterion goes as follows.

\begin{thm} \cite{[Ch]} \label{thm1.1}
{\it Let $r$ and $s$ be positive integers.
For any $\bm{\alpha}=(\alpha_1,...,\alpha_r)$ and
$\bm{\beta}=(\beta_1,...,\beta_s)$
with $\alpha_i,\beta_j\in\mathbb{Q}\setminus\mathbb{Z}_{\le0}$,
$F_{\bm{\alpha},\bm{\beta}}(z)$ is $N$-integral if and only if
$\delta_{\bm{\alpha},\bm{\beta}}(x, a)\ge 0$ for all $x\in\mathbb{R}$
and all $a\in\{1,...,d_{\bm{\alpha},\bm{\beta}}\}$ coprime to
$d_{\bm{\alpha},\bm{\beta}}$.}
\end{thm}
Christol's criterion (i.e., Theorem \ref{thm1.1}) plays an important
and key role on the study of N-integrality of the hypergeometric mirror maps
$q_{\bm{\alpha,\beta}}(z)$ in \cite{[DRR]}. It generalizes Landau's
criterion which is the basic tool of \cite{[De]} on the N-integrality of
$q_{\bm{\alpha, \beta}}(z)$.
On the other hand, if $F_{\bm{\alpha},\bm{\beta}}(z)$ is $N$-integral,
then it is not hard to see that the set of all $c\in\mathbb{Q}$ satisfying
$F_{\bm{\alpha},\bm{\beta}}(cz)\in\mathbb{Z}[[z]]$ can be written as
$C_{\bm{\alpha},\bm{\beta}}\mathbb{Z}$ for a unique positive
rational number $C_{\bm{\alpha},\bm{\beta}}$.
We call $C_{\bm{\alpha},\bm{\beta}}$ the {\it Eisenstein constant}
of $F_{\bm{\alpha},\bm{\beta}}(z)$. Delaygue, Rivoal and Roques
\cite{[DRR]} gave a formula of $C_{\bm{\alpha},\bm{\beta}}$.
By using Christol's function, very recently, Delaygue, Rivoal and Roques
obtained an explicit formula for the Eisenstein constant of
any those N-integral hypergeometric series with rational parameters.

The arithmetic property of $q_{\bm{\alpha,\beta}}(z)$ was first studied by Dwork.
In \cite{[Dw1]} and \cite{[Dw2]}, $q_{\bm{\alpha,\beta}}(z)$ was treated as
a $p$-adic analytic function.
To investigate the $p$-adic analytic continuation of $q_{\bm{\alpha,\beta}}(z)$,
Dwork developed a $p$-adic formal congruence theorem.
This theorem turned out to be a crucial tool to study the N-integrality of
$q_{\bm{\alpha,\beta}}(z)$.
Encouraged by numerical results, Lian and Yau started to consider
the N-integrality of $q_{\bm{\alpha,\beta}}(z)$.
By using Dwork's formal congruence theorem, Lian and Yau \cite{[LY1]}
were able to establish for the first time the N-integrality of
$q_{\bm{\alpha,\beta}}(z)$ for infinitely many cases.
After that, partial results in this direction were found by Lian and Yau \cite{[LY2]},
by Zudilin \cite{[Zu]}, by Krattenthaler and Rivoal \cite{[KR1],[KR2]},
by Delaygue \cite{[De]} and by Roques \cite{[Rq1], [Rq2]}.
So far the most general result, which contains all the interested cases,
was obtained by Delaygue, Rivoal and Roques \cite{[DRR]}.
They established a criterion to decide whether $q_{\bm{\alpha,\beta}}(z)$ is
N-integral or not for any disjoint $\bm{\alpha}$ and $\bm{\beta}$ under
the assumption that $F_{\bm{\alpha,\beta}}(z)$ is N-integral,
where $\bm{\alpha}$ and $\bm{\beta}$ are disjoint if and only if
$\alpha_i-\beta_j\notin\mathbb{Z}$ for all $(i,j)\in\{1,...,r\}\times\{1,...,s\}$.
The arithmetic properties of multivariable hypergeometric series and
hypergeometric mirror maps with rational parameters are also extensively considered.
For example, Adolphson and Sperber \cite{[AS1]} studied the $p$-integrality of
multivariable hypergeometric series with rational parameters.
Furthermore, Adolphson and Sperber \cite{[AS2]}, Beukers \cite{[B]},
Delaygue \cite{[De2]} and Krattenthaler and Rivoal \cite{[KR3]} investigated
the N-integrality of multivariable hypergeometric mirror maps.

In this paper, we initiate the investigation of the N-integrality of hypergeometric
series with parameters from algebraic number fields, and in particular,
we will develop a systematic theory on the N-integrality of
the hypergeometric series with parameters from quadratic fields.
Let's start with the $p$-adic integrality of $F_{\bm{\alpha,\beta}}(z)$,
where the parameters $\bm{\alpha}$ and $\bm{\beta}$ are given
as in Theorem \ref{thm1.1}. Let
\begin{equation} \label{Eq1.5}
M_{\bm{\alpha,\beta}}:=d_{\bm{\alpha,\beta}}\big(2+2\max_{1\le i\le r \atop 1\le j\le s}
\{|\alpha_i|,|\beta_j|\}\big)
\end{equation}
and
\begin{equation} \label{Eq1.6}
m_{\bm{\alpha,\beta}}:=\min_{1\le i\le r \atop 1\le j\le s}
\{\alpha_i-\langle\alpha_i\rangle, \beta_j-\langle\beta_j\rangle\}.
\end{equation}
We have the following criterion on the $p$-adic integrality of
$F_{\bm{\alpha,\beta}}(z)$ with the parameters $\alpha$ and
$\beta$ being rational numbers.

\begin{thm} \label{thm1.2}
Let $r$ and $s$ be positive integers. Let $\bm{\alpha}=(\alpha_1,...,
\alpha_r)$ and $\bm{\beta}=(\beta_1,...,\beta_s)$
with $\alpha_i,\beta_j\in\mathbb{Q}\setminus\mathbb{Z}_{\le0}$.
Let $p$ be a prime with $p>\max\{M_{\bm{\alpha,\beta}}, 3sd_{\bm{\alpha,\beta}}\}$.
Let $a$ be the integer such that $1\le a\le d_{\bm{\alpha,\beta}}$ and
$ap \equiv 1 \mod d_{\bm{\alpha,\beta}}$. Then each of the following is true.

{\rm (i).} If $r<s$, then $F_{\bm{\alpha,\beta}}(z)\not\in\mathbb{Z}_p[[z]]$.

{\rm (ii).} If $r>s$, then $F_{\bm{\alpha,\beta}}(z)\in\mathbb{Z}_p[[z]]$ if
and only if $\delta_{\bm{\alpha,\beta}}(x;a)\ge 0$ for any $x\in\mathbb{R}$.

{\rm (iii).} If $r=s$, then $F_{\bm{\alpha,\beta}}(z)\in\mathbb{Z}_p[[z]]$ if
and only if
\begin{equation} \label{Eq1.7}
\sum_{l=1}^{h}\delta_{\bm{\alpha,\beta}}(a^l\beta_k;a^l)\ge 0
\end{equation}
for all $h\in\{1,...,{\rm ord} (a)\}$ with ${\rm ord}(a)$ being
the order of $a$ modulo $d_{\bm{\alpha,\beta}}$ and all $k\in\{1,...,s\}$, and
\begin{equation} \label{Eq1.8}
\delta_{\bm{\alpha,\beta}}\Big(a^l\Big(\frac{e}{d_{\bm{\alpha,\beta}}}
+m_{\bm{\alpha,\beta}}\Big), a^l\Big)\ge 0
\end{equation}
for all $l\in\{1,...,{\rm ord}(a)\}$ and $e\in\{1,...,d_{\bm{\alpha,\beta}}\}$.
%
\end{thm}

Now let $K$ be an algebraic number field and $\mathcal{O}_K$ be the ring of integers in $K$.
For a power series $F(z)\in 1+zK[[z]]$, we say that $F(z)$ is {\it N-integral}
if there is a nonzero element $c\in K$ such that $F(cz)\in\mathcal{O}_K[[z]]$.
For any $\bm{\alpha}=(\alpha_1,...,\alpha_r)$ and $\bm{\beta}=(\beta_1,...,\beta_s)$
with $\alpha_i,\beta_j\in K\setminus\mathbb{Z}_{\le0}$,
the equations (\ref{Eq1.1}) and (\ref{Eq1.3}) are obviously still valid.
We define $F_{\bm{\alpha,\beta}}(z)$ given by (\ref{Eq1.1}) to be the hypergeometric series
associated with $\bm{\alpha}$ and $\bm{\beta}$ and $q_{\bm{\alpha,\beta}}(z)$
given by (\ref{Eq1.3}) to be the associated mirror map.
We are interested in the N-integrality of $F_{\bm{\alpha,\beta}}(z)$ for
algebraic number parameters $\bm{\alpha}$ and $\bm{\beta}$.
The following question arises naturally: Does there exist an effective criterion
determining effectively the N-integrality of $F_{\bm{\alpha,\beta}}(z)$
for algebraic number parameters $\bm{\alpha}$ and $\bm{\beta}$?

For a prime ideal $\mathfrak{p}$ of $\mathcal{O}_K$,
let $K_{\mathfrak{p}}$ be the $\mathfrak{p}$-adic
completion of $K$, and $\mathcal{O}_{K,\mathfrak{p}}$
be the valuation ring of $K_\mathfrak{p}$.
For a prime number $p$, let $\mathbb{C}_p$ denote the completion
of the algebraic closure of $\mathbb{Q}_p$,
and $\mathcal{O}_p$ be the valuation ring of $\mathbb{C}_p$.
For any monomorphism $\sigma:K\mapsto\mathbb{C}_p$, we define
$$\sigma\big(F_{\bm{\alpha},\bm{\beta}}(z)\big):=\sum_{n=0}^{\infty}
\sigma\Big(\frac{(\alpha_1)_n\cdots(\alpha_r)_n}
{(\beta_1)_n\cdots(\beta_s)_n}\Big)z^n\in\mathbb{C}_p[[z]].$$
In this direction, we have the following result.

\begin{thm} \label{thm1.3}
Let $K$ be an algebraic number field and $\mathcal{O}_K$
be the ring of algebraic integers in $K$. Let
$\bm{\alpha}=(\alpha_1,...,\alpha_r)$ and $\bm{\beta}=(\beta_1,...,\beta_s)$
with $\alpha_i,\beta_j\in K\setminus\mathbb{Z}_{\le0}$.
Then the following statements are equivalent.

{\rm (i).} $F_{\bm{\alpha,\beta}}(z)$ is N-integral in $K$.

{\rm (ii).} $F_{\bm{\alpha,\beta}}(z)\in\mathcal{O}_{K,\mathfrak{p}}[[z]]$
for almost all prime ideals $\mathfrak{p}$ of $\mathcal{O}_K$.

{\rm (iii).} For almost all prime numbers $p$ and
for all monomorphisms $\sigma_p: K\rightarrow\mathbb{C}_p$,
the power series $\sigma_p(F_{\bm{\alpha,\beta}}(z))$
is contained in $\mathcal{O}_p[[z]]$.
\end{thm}
Clearly, if letting $K=\mathbb{Q}$, then Theorem \ref{thm1.3}
becomes Proposition 22 of \cite{[DRR]}. On the other hand,
Theorem \ref{thm1.3} gives two equivalent descriptions
of the N-integrality of hypergeometric series with algebraic number
parameters and so can be viewed as a rough answer to the above answer.
However, these equivalent descriptions are not efficient to
determine the N-integrality of $F_{\bm{\alpha,\beta}}(z)$.
For the quadratic field case, we success in giving
an effective criterion and hence we answer
the above question when $K$ is a quadratic field.

In the rest of this section, let $K=\mathbb{Q}(\sqrt{D})$ for a square-free
integer $D$ other than 1. Every element $\gamma\in K$ can be written as
$\gamma=\gamma_1+\gamma_2\sqrt{D}$ with $\gamma_1,\gamma_2\in\mathbb{Q}$.
And $\gamma \in \mathbb{Q}$ if and only if $\gamma_2=0$.
Then one can write
$$\bm{\alpha}=(\alpha_{11}+\alpha_{21}\sqrt{D}, ...,
\alpha_{1r}+\alpha_{2r}\sqrt{D})$$
and
$$\bm{\beta}=(\beta_{11}+\beta_{21}\sqrt{D},...,\beta_{1s}+\beta_{2s}\sqrt{D})$$
with $\alpha_{1i},\alpha_{2i},\beta_{1j}, \beta_{2j}$
being elements in $\mathbb{Q}$ for $1\le i\le r$ and $1\le j\le s$. Let
$$u:=\#\{0\le i\le r: \alpha_i\in\mathbb{Q}\}$$
and
$$v:=\#\{0\le j\le s: \beta_j\in\mathbb{Q}\}.$$
Note that we may have $u=0$ or $u=r$ as well as $v=0$ or $v=s$.
Without loss of any generality, suppose that $\alpha_1,...,\alpha_u$
and $\beta_1,...,\beta_v$ are rational while $\alpha_{u+1},...,\alpha_{r}$
and $\beta_{v+1},...,\beta_s$ are irrational.
Then $\alpha_{21}=\cdots\alpha_{2u}=\beta_{21}=\cdots=\beta_{2v}=0$,
and $\alpha_{2i}$ and $\beta_{2j}$ are nonzero for all integers $i$
and $j$ with $u+1\le i\le r$ and $v+1\le j\le s$. Set
$$\bm{\alpha}_1:=(\alpha_{11},..., \alpha_{1r}),\ \
\bm{\beta}_1:=(\beta_{11},..., \beta_{1s}),$$
$$\bm{\alpha}_2:=(\alpha_{21},..., \alpha_{2r}),\ \
\bm{\beta}_2:=(\beta_{21},..., \beta_{2s}),$$
and
$$\bm{\mu}:=(\alpha_1,...,\alpha_u),\ \
\bm{\nu}:=(\beta_1,...,\beta_v).$$
Let $d_{\bm{\alpha_1,\beta_1}}$ (resp. $d_{\bm{\alpha_2,\beta_2}}$, $d_{\bm{\bm{\mu,\nu}}}$) be
the least common multiple of the exact denominators of elements of $\bm{\alpha_1}$ and
$\bm{\beta_1}$ (resp. $\bm{\alpha_2}$ and $\bm{\beta_2}$, $\bm{\mu}$ and $\nu$). Define
$$E:={\rm lcm}(4D, d_{\bm{\alpha_1,\beta_1}},d_{\bm{\alpha_2,\beta_2}})
 \ \ {\rm and}\ \ G:=(\mathbb{Z}/E\mathbb{Z})^{\times}.$$
We correspond $a\in G$ with a positive integer $a$ with $a\in\{1, ..., E\}$ and $(a,E)=1$.
We mention without proof here that if two primes $p,q$ satisfies $p\equiv q \mod 4D$,
then $\big(\frac{D}{p}\big)=\big(\frac{D}{q}\big)$. Hence one can define
$$H:=\Big\{a\in(\mathbb{Z}/E\mathbb{Z})^{\times} : \Big(\frac{D}{p}\Big)=1 \
{\rm holds \ for\ all\ primes}\ p\ {\rm with}\ ap\equiv 1 \mod E \Big\},$$
and $I:=G\setminus H$.

For $a\in G$, $\epsilon\in (0,1)$ and $x\in\mathbb{R}$, we define
\begin{align*}
\Delta_{\bm{\alpha}, \bm{\beta}}(x,a,\epsilon)
:=&\#\{1\le i\le r: a\alpha_{1i}+\tilde{\alpha}_{2i}\epsilon \preccurlyeq x\}
-\#\{1\le j\le s: a\beta_{1j}+\tilde{\beta}_{2j}\epsilon \preccurlyeq x\},
\end{align*}
where $\tilde{\alpha}_{2i}:=\alpha_{2i}d_{\bm{\alpha_2,\beta_2}}$ and
$\tilde{\beta}_{2j}:=\beta_{2j}d_{\bm{\alpha_2,\beta_2}}$.
One may notice that if all components of $\bm{\alpha}$ and
$\bm{\beta}$ are rational, then for any fixed $x$ and $a$,
$$\Delta_{\bm{\alpha}, \bm{\beta}}(x,a,\epsilon)=\delta_{\bm{\alpha}, \bm{\beta}}(x,a)$$
holds for all $\epsilon\in(0,1)$. Hence our function
$\Delta_{\bm{\alpha}, \bm{\beta}}(x,a,\epsilon)$ generalizes the
Christol function $\delta_{\bm{\alpha}, \bm{\beta}}(x,a)$. Let
\begin{equation} \label{Eq1.9}
S_{\bm{\alpha,\beta}}:=\Big(\bigcup_{i=1}^5 S_i\Big)\cap (0,1),
\end{equation}
where
$$S_1:=\Big\{\frac{a\alpha_{1i}-a\alpha_{1i'}+n}
{\tilde{\alpha}_{2i'}-\tilde{\alpha}_{2i}}:
\ a\in H, n\in \mathbb{Z},  1\le i,i'\le r \ {\rm and}\
\tilde{\alpha}_{2i}\neq\tilde{\alpha}_{2i'} \Big\}, $$
$$S_2:=\Big\{\frac{a\beta_{1j}-a\beta_{1j'}+n}
{\tilde{\beta}_{2j'}-\tilde{\beta}_{2j}}:
\ a\in H, n\in \mathbb{Z}, 1\le j, j'\le s \ {\rm and}
\ \tilde{\beta}_{2j}\neq\tilde{\beta}_{2j'} \Big\},$$
$$S_3:=\Big\{\frac{a\alpha_{1i}-a\beta_{1j}+n}
{\tilde{\beta}_{2j}-\tilde{\alpha}_{2i}}:
\ a\in H, n\in \mathbb{Z},  1\le i\le r, 1\le j\le s \ {\rm and}
\ \tilde{\alpha}_{2i}\neq\tilde{\beta}_{2j} \Big\},$$
$$S_4:=\Big\{\frac{n-a\alpha_{1i}}{\tilde\alpha_{2i}}: \ a\in H,
n\in \mathbb{Z},  1\le i\le r, \tilde\alpha_{2i}\ne 0\Big\}$$
and
$$S_5:=\Big\{\frac{n-a\beta_{1j}}{\tilde\beta_{2i}}:\ a\in H,
n\in \mathbb{Z},  1\le j\le s, \tilde\beta_{2j}\ne 0\Big\}.$$
Evidently, $S_{\bm{\alpha,\beta}}$ is a finite set.
For quadratic parameters $\bm{\alpha}$ and $\bm{\beta}$
and using the notation above, we present the following
four statements:

{\bf Statement I}: $\delta_{\bm{\bm{\mu,\nu}}}(x,a)\ge 0$ for all $a\in I$ and $x\in\mathbb{R}$.

{\bf Statement II}:
$\delta_{\bm{\bm{\mu,\nu}}}(a^l(\frac{e}{d_{\bm{\bm{\mu,\nu}}}}+m_{\bm{\bm{\mu,\nu}}}),a^l)\ge 0$
for all $a\in I$, $e\in\{1,...,d_{\bm{\bm{\mu,\nu}}}\}$ and $l\in\{1,...,{\rm ord}(a)\}$,
where ${\rm ord}(a)$ is the order of $a$ modulo $E$.

{\bf Statement III}: $\sum_{l=1}^{h}\delta_{\bm{\bm{\mu,\nu}}}(a^l\beta_k,a^l)\ge 0$
for all $a\in I$, $k\in\{1,...,s\}$ and $h\in\{1,..., {\rm ord}(a)\}$.

{\bf Statement IV}: $\Delta_{\bm{\alpha}, \bm{\beta}}(x,a,\epsilon)\ge 0$ for all $a\in H$,
$\epsilon\in(0,1)\setminus S_{\bm{\alpha,\beta}}$ and $x\in\mathbb{R}$.\\

We can now state the following criterion on the N-integrality of
$F_{\bm{\alpha},\bm{\beta}}(z)$.

\begin{thm} \label{thm1.4}
Let $D$ be a square-free integer other than 1.
Let $\bm{\alpha}=(\alpha_1,...,\alpha_r)$ and $\bm{\beta}=(\beta_1,...,\beta_s)$
with $\alpha_i,\beta_j\in \mathbb{Q}(\sqrt{D})\setminus\mathbb{Z}_{\le 0}$.
Let $u$ and $v$ be the number of rational components of $\bm{\alpha}$ and
$\bm{\beta}$, respectively. Then each of the following is true:

{\rm (i).} If $u<v$, then the hypergeometric series $F_{\bm{\alpha},\bm{\beta}}(z)$
is not N-integral.

{\rm (ii).} If $u>v$, then the hypergeometric series $F_{\bm{\alpha},\bm{\beta}}(z)$
is N-integral if and only if the statements I and IV are true.

{\rm (iii).} If $u=v$, then the hypergeometric series $F_{\bm{\alpha},\bm{\beta}}(z)$
is N-integral if and only if the statements II, III and IV are true.
\end{thm}


This paper is organized as follows. In Section 2, we first explore
properties of Dwork map and Christol operator, and then use them to
study the $p$-adic integrality of hypergeometric series with rational
parameters and finally give the proof of Theorem \ref{thm1.2}.
Then in Section 3, we exploit the N-integrality for
the hypergeometric series with parameters from
algebraic number fields. The criterion for such
hypergeometric series to be N-integral is given in
Theorem \ref{thm1.3} above. Its proof needs one important
result of Stewart on the upper bound of the number
of solutions of polynomial congruences obtained
in 1991 \cite{[St]} which is one of important ingredients
of this paper. Consequently, we study in Section 4 roots
of quadratic congruences modulo various prime numbers.
We will make use of the uniform distribution theorem of Duke,
Friedlander, Iwaniec \cite{[DFI]} and Toth \cite{[To]} to
study the roots of quadratic congruences modulo various
prime numbers. Especially, we will show that
\begin{equation} \label{Eq1.10}
\#\Big\{\frac{v}{p}: p\in\mathcal{S},0\le v<p<x,
f(v)\equiv 0 \mod p^2\Big\}=o(\pi (x))
\end{equation}
if $f$ is a primitive irreducible polynomial of degree two
and integer coefficients and $x$ goes to infinity.
This result is one key step in the proof of Theorem
\ref{thm1.4} and another important ingredient of this paper.
Section 5 is devoted to the study of N-integrality of
the hypergeometric series with parameters from
quadratic fields. We will develop some properties on the
function $\langle \cdot \rangle$, the operator $T_{p, l}$
and the function $\Delta_{\bm{\alpha}, \bm{\beta}}$.
Then we use them together with (\ref{Eq1.10}) to give
the proof of Theorem \ref{thm1.4}. Finally, we provide
one example to demonstrate Theorem \ref{thm1.4} as
the conclusion of this paper.

In the forthcoming works, we will use the main results of
the current paper to study the Eisenstein constant of the
hypergeometric series $F_{\bm{\alpha},\bm{\beta}}(z)$ with
the parameters $\alpha $ and $\beta$ coming from quadratic
fields. Moreover, we will investigate the N-integrality
of the hypergeometric series $q_{\bm{\alpha,\beta}}(z)$
for the quadratic parameters $\alpha $ and $\beta$.

Throughout this paper, for any given rational number $\alpha$,
we denote by $d(\alpha)$
the denominator of $\alpha$. As usual, $\lfloor \cdot \rfloor$,
$\lceil \cdot \rceil$ and $\{ \cdot \}$ will stand for the
floor function, the ceil function and the fractional part function,
respectively. For any integers $a$ and $b$, let $(a,b)$
denote the greatest common divisor of $a$ and $b$.

\section{Dwork map, Christol operator, $p$-adic integrality
of the hypergeometric series with rational parameters and
proof of Theorem \ref{thm1.2}}

In this section, our purpose is to find an equivalent
condition for $F_{\bm{\alpha,\beta}}(z)\in\mathbb{Z}_p[[z]]$
for a given prime number $p$. We will first investigate properties
of Dwork map and Christol operator, and then use them to
study the $p$-adic integrality of hypergeometric series
with rational parameters and finally arrive at the proof
of Theorem \ref{thm1.2}.

Let $F_{\bm{\alpha,\beta}}(z)$ be a hypergeometric function
with rational parameters $\bm{\alpha}$ and $\bm{\beta}$.
We begin with the Dwork map. For any $\alpha\in\mathbb{Z}_p$,
there is an unique element in $\mathbb{Z}_p$, denoted by
$\mathfrak{D}_p(\alpha)$, satisfying that
$$p\mathfrak{D}_p(\alpha)-\alpha \in \{0,1,...,p-1 \}.$$
This defines a map
\begin{align*}
\mathfrak{D}_p: & \mathbb{Z}_p\rightarrow \mathbb{Z}_p,\\
                & \alpha\mapsto\mathfrak{D}_p(\alpha),
\end{align*}
which is called {\it Dwork map} as introduced first
by Dwork in \cite{[Dw1]}. For any positive integer $l$,
let $\mathfrak{D}_p^l$ denote the $l$-th iteration of
$\mathfrak{D}_p$. Then $\mathfrak{D}_p^l(\alpha)$ is
the unique element in $\mathbb{Z}_p$ such that
$$p^l\mathfrak{D}_p^l(\alpha)-\alpha\in \{0,1,...,p^l-1\}.$$
So one has
\begin{equation}\label{Eq2.0}
0\le\mathfrak{D}_p^l(\alpha)-\frac{\alpha}{p^l}\le 1-\frac{1}{p^l}.
\end{equation}
Christol \cite{[Ch]} introduced the operator $T_{p, l}$ which is defined by
\begin{equation}\label{Eq2.1}
T_{p,l}(\alpha):=p^l\mathfrak{D}^l_p(\alpha)-\alpha.
\end{equation}
We introduce a new operator $R_{p, l}$ that is defined by
\begin{equation} \label{Eq2.2}
R_{p,l}(\alpha):=p^l-T_{p,l}(\alpha).
\end{equation}
When $l=1$, $T_{p,l}(\alpha)$ and $R_{p,l}(\alpha)$ is denoted by
$T_{p}(\alpha)$ and $R_{p}(\alpha)$ for short.
Then $T_{p,l}(\alpha)$ is the unique element in $\{0, 1, ..., p^l-1\}$ such that
$$T_{p,l}(\alpha)+\alpha \equiv 0 \mod p^l,$$
and $R_{p,l}(\alpha)$ is the unique element in $\{1,...,p^l\}$ such that
$$R_{p,l}(\alpha)-\alpha \equiv 0 \mod p^l.$$
We also point out a basic fact which states that for any
positive integer $l$ and for any $\alpha \in \mathbb{Z}_p$,
one has
\begin{align} \label{Eq2.2'}
T_{p, l}(\alpha)\le T_{p, l+1}(\alpha).
\end{align}
Actually, since $T_{p,l}(\alpha)+\alpha \equiv 0 \mod p^l$ and
$T_{p,l+1}(\alpha)+\alpha \equiv 0 \mod p^{l+1}$, we deduce that
$T_{p,l}(\alpha)\equiv T_{p, l+1}(\alpha) \mod p^l$. One may write
$T_{p,l+1}(\alpha)=T_{p, l}(\alpha)+t p^l$ for some integer $t$.
But $T_{p, l}(\alpha)\in [0, p^l-1]$ and $T_{p, l+1}(\alpha)\ge 0$.
Then we must have $t\ge 0$. So (2.4) is true.

Let $v_p$ denote the $p$-adic valuation. For basic properties
of $v_p$, see, for example, \cite{[Ko]} and \cite{[Ro]}.
We have the following result on the $p$-adic valuation
of Pochhammer symbol.

\begin{lem} \label{lem2.1}
Let $\alpha\in\mathbb{Z}_p\setminus\mathbb{Z}_{\le 0}$.
Then for any nonnegative integer $n$, we have
\begin{align} \label{Eq2.3}
v_p\big((\alpha)_n\big)=&\sum_{l=1}^{\infty}\Big\lfloor\frac{n-1+R_{p,l}(\alpha)}{p^l}\Big\rfloor \\
\label{Eq2.3'} =&\sum_{l=1}^{\infty}\Big\lceil\frac{n-T_{p,l}(\alpha)}{p^l}\Big\rceil \\
\label{Eq2.3''}=&\sum_{l=1}^{\infty}\Big\lceil\frac{n+\alpha}{p^l}-\mathfrak{D}_p^{l}(\alpha)\Big\rceil.
\end{align}
\end{lem}

\begin{proof}
Christol proved in \cite{[Ch]} that for a $p$-adic integer
$\alpha$ and any nonnegative integer $n$, one has
$$v_p\big((\alpha)_n\big)=\sum_{l=1}^{\infty}\Big\lfloor
\frac{n-1+R_{p,l}(\alpha)}{p^l}\Big\rfloor.$$
Delaygue, Roques and Rivoal proved in \cite{[DRR]}
that for all positive integers $l$, one has
\begin{equation} \label{Eq2.4}
\Big\lfloor\frac{n-1+R_{p,l}(\alpha)}{p^l}\Big\rfloor=
\Big\lceil\frac{n-T_{p,l}(\alpha)}{p^l}\Big\rceil.
\end{equation}
But by (\ref{Eq2.1}), one has
\begin{equation} \label{Eq2.5}
\Big\lceil\frac{n-T_{p,l}(\alpha)}{p^l}\Big\rceil=\Big\lceil
\frac{n+\alpha}{p^l}-\mathfrak{D}_p^{l}(\alpha)\Big\rceil.
\end{equation}
Then (\ref{Eq2.3''}) follows immediately form the above three equations.
Thus Lemma \ref{lem2.1} is proved.
%
\end{proof}

Let now $\bm{\alpha}=(\alpha_1,...,\alpha_r)$ and
$\bm{\beta}=(\beta_1,...,\beta_s)$
with $\alpha_i,\beta_j\in\mathbb{Z}_p\setminus\mathbb{Z}_{\le 0}$.
Let $F_{\bm{\alpha,\beta}}(z)\in\mathbb{Q}_p[[z]]$
be the formal power series defined by (1.1).
As an application of Lemma \ref{lem2.1},
we supply the following result.

\begin{lem} \label{lem2.2}
If $r<s$, then $F_{\bm{\alpha,\beta}}(z)\not\in\mathbb{Z}_p[[z]]$.
\end{lem}

\begin{proof}
Deduced from (\ref{Eq2.3}), one has
$$
v_p\Big(\frac{(\alpha_1)_n\cdots(\alpha_r)_n}{(\beta_1)_n\cdots(\beta_s)_n}\Big)=
\sum_{i=1}^{r}\sum_{l=1}^{\infty}\Big\lfloor\frac{n-1+R_{p,l}(\alpha_i)}{p^l}\Big\rfloor
-\sum_{j=1}^{s}\sum_{l=1}^{\infty}\Big\lfloor\frac{n-1+R_{p,l}(\beta_j)}{p^l}\Big\rfloor.
$$
Since $1\le R_{p, l}(\alpha )\le p^l$ implying that
\begin{equation} \label{Eq2.6}
\Big\lfloor\frac{n}{p^l}\Big\rfloor\le
\Big\lfloor\frac{n-1+R_{p,l}(\alpha)}{p^l}\Big\rfloor
\le\Big\lfloor\frac{n-1}{p^l}\Big\rfloor+1,
\end{equation}
it then follows that
$$(r-s)\sum_{l=1}^{\infty}\Big\lfloor\frac{n}{p^l}\Big\rfloor-\frac{s\log n}{\log p}
\le v_p\Big(\frac{(\alpha_1)_n\cdots(\alpha_r)_n}{(\beta_1)_n\cdots(\beta_s)_n}\Big)
\le (r-s)\sum_{l=1}^{\infty}\Big\lfloor\frac{n}{p^l}\Big\rfloor +\frac{r\log n}{\log p}.$$
Since $r<s$, we have
$$(r-s)\sum_{l=1}^{\infty}\Big\lfloor\frac{n}{p^l}\Big\rfloor+\frac{r\log n}{\log p}
<(r-s)\Big(\frac{n}{p}-1\Big)+\frac{r\log n}{\log p}.$$

Let $g(x):=(r-s)(x/p-1)+\log x/\log p$. Then one can compute the derivative
and get that $g'(x)=(r-s)/p+1/x$. Evidently, we have $g'(x)<0$ if $x>p/(s-r)$.
So there exists a positive integer $N$ such that for all integers $n$ with
$n>N$, one has $g(n)<0$, that is
$$(r-s)\sum_{l=1}^{\infty}\Big\lfloor\frac{n}{p^l}\Big\rfloor+\frac{r\log n}{\log p}<0.$$
Hence $F_{\bm{\alpha,\beta}}(z)\not\in\mathbb{Z}_p[[z]]$.
This ends the proof of Lemma \ref{lem2.2}.
\end{proof}

Let $\alpha\in\mathbb{Q}$, and $d(\alpha)$ denote
the exact denominator of $\alpha$.
For a prime $p$ such that $\alpha\in\mathbb{Q}\cap\mathbb{Z}_p$,
there is an explicit formula of $\mathfrak{D}_{p}^{l}(\alpha)$.

\begin{lem} \cite{[DRR]}\label{lem2.3}
Let $\alpha\in\mathbb{Q}\setminus\mathbb{Z}_{\leq0}$. Then for any prime $p$
such that $\alpha\in\mathbb{Z}_p$ and all positive integers $l$ such that
$p^l\ge d(\alpha)(|\lfloor 1-\alpha \rfloor|+\langle \alpha \rangle)$,
we have
$$\mathfrak{D}_p^l(\alpha)=\mathfrak{D}_p^l(\langle \alpha \rangle)
=\langle \omega\alpha \rangle,$$
where $\omega$ denotes any integer satisfying that
$\omega p^l\equiv 1 \mod d(\alpha)$.
\end{lem}

\begin{lem} \label{lem2.4}
Let $\gamma_1$ and $\gamma_2$ be two rational numbers
in $\mathbb{Q}\setminus\mathbb{Z}_{\leq0}$.
Then each of the following is true:

{\rm (i).} If $\langle b\gamma_1\rangle=\langle b\gamma_2\rangle$
for a positive integer $b$ coprime to
${\rm lcm}(d(\gamma_1),d(\gamma_2))$, then
$\langle a\gamma_1\rangle=\langle a\gamma_2\rangle$
for any positive integer $a$.

{\rm (ii).} Let $p$ be a prime such that
$$p\ge \max(d(\gamma_1)(|\lfloor 1-\gamma_1 \rfloor|+\langle
\gamma_1\rangle), d(\gamma_2)(|\lfloor 1-\gamma_2 \rfloor|+\langle\gamma_2\rangle)).$$
Then $\mathfrak{D}_p^L(\gamma_1)=\mathfrak{D}_p^L(\gamma_2)$
holds for a given positive integer $L$
if and only if $\mathfrak{D}_p^l(\gamma_1)=\mathfrak{D}_p^l(\gamma_2)$
holds for all positive integers $l$.
\end{lem}
\begin{proof} Set $d:={\rm lcm}(d(\gamma_1),d(\gamma_2))$.
We begin with the proof of part (i).

(i). Write $\gamma_1=\frac{e_1}{d}$ and $\gamma_2=\frac{e_2}{d}$,
which are not necessarily reduced. Let
$$be_1=dq_1+r_1\ \ {\rm and}\ \ be_2=dq_2+r_2,$$
where $q_1,q_2,r_1,r_2$ are integers with $0\le r_1,r_2<d$.
Then by the definition of the function $\langle~\rangle$, we have
$\langle b\gamma_1\rangle=\frac{r_1}{d}$ and
$\langle b\gamma_2\rangle=\frac{r_2}{d}.$
Since $\langle b\gamma_1\rangle=\langle b\gamma_2\rangle$ and $0\le r_1,r_2<d$,
it follows that $r_1=r_2$, which implies $be_1\equiv be_2 \mod d$.
Moreover, since $(b,d)=1$, one has $ae_1\equiv ae_2 \mod d$ for any positive integer $a$.
Thus $a\gamma_1-a\gamma_2$ is an integer, which follows that
$\langle a\gamma_1\rangle=\langle a\gamma_2\rangle$.

(ii). If $\mathfrak{D}_p^l(\gamma_1)=\mathfrak{D}_p^l(\gamma_2)$
holds for all positive integers $l$,
then $\mathfrak{D}_p^L(\gamma_1)=\mathfrak{D}_p^L(\gamma_2)$
holds for a given positive integer $L$. In what follows,
we assume that $\mathfrak{D}_p^L(\gamma_1)=\mathfrak{D}_p^L(\gamma_2)$
holds for a given positive integer $L$. Then by Lemma \ref{lem2.3},
for a positive integer $b$ with $bp^L\equiv 1\mod d$,
one has $\langle b\gamma_1\rangle=\langle b\gamma_2\rangle$.
Then one can deduce from (i) that $\langle a\gamma_1\rangle=\langle a\gamma_2\rangle$
for any positive integer $a$. In particular, for a positive integer $l$, if let $a_l$
be a positive integer such that $a_lp^l\equiv 1 \mod d$, then one has
$\langle a_l\gamma_1\rangle=\langle a_l\gamma_2\rangle$, So by Lemma \ref{lem2.3},
$\mathfrak{D}_p^l(\gamma_1)=\mathfrak{D}_p^l(\gamma_2)$.
This finishes the proof of Lemma \ref{lem2.4}.
\end{proof}

Before proceeding, we present a lemma which should
be useful later in this section.

\begin{lem} \label{lem2.5}
Let $A$ be a set on which a total order $\le $ is defined.
Let $x_1,...,x_r$ and $y_1,...,y_s$ be the elements of $A$.
Then
\begin{equation} \label{Eq2.7}
\#\{i: x_i\le a\}-\#\{j: y_j\le a\}\ge 0 \ \forall a\in A
\end{equation}
if and only if
\begin{equation} \label{Eq2.8}
\#\{i: x_i\le y_k\}-\#\{j: y_j\le y_k\}\ge 0 \ \forall k\in\{1,...,s\}.
\end{equation}
\end{lem}

\begin{proof}
First of all, it is clear that (\ref{Eq2.8}) holds
if (\ref{Eq2.7}) is true. Now let us show the truth
of the converse. Assume that (\ref{Eq2.8}) is true.
In what follows we show that (\ref{Eq2.7}) is true.
With no loss of any generality, one may suppose that
$y_1\le\cdots\le y_s$. For any given $a\in A$,
we consider the following three cases:

{\sc Case 1.} $a\le y_1, a\neq y_1$. Then
$\#\{j: y_j\le a\}=0$. So one has
$$\#\{i: x_i\le a\}-\#\{j: y_j\le a\}\ge 0$$
as (\ref{Eq2.7}) claimed.

{\sc Case 2.} $y_k\le a\le y_{k+1}, a\neq y_{k+1}$
for some integer $k_0$ with $1\le k_0< s$.
Then
$$\#\{j: y_j\le a\}=\#\{j: y_j\le y_{k_0}\}=k_0$$
and
$$\#\{i: x_i\le a\}\ge \#\{i:x_i\le y_{k_0}\}.$$
On the other hand, (\ref{Eq2.8}) implies that
$$\#\{i: x_i\le y_{k_0}\}-\#\{j: y_j\le y_{k_0}\}\ge 0.$$
It then follows that
$$\#\{i:x_i\le a\}-\#\{j: y_j\le a\}
\ge \#\{i: x_i\le y_{k_0}\}-\#\{j: y_j\le y_{k_0}\}\ge 0.$$
So (\ref{Eq2.7}) holds in this case.

{\sc Case 3.} $y_s\le a$. Then we have
$\#\{j: y_j\le a\}=\#\{j: y_j\le y_s\}=s$ and
$\#\{i: x_i\le a\}\ge \#\{i: x_i\le y_s\}$.
But (\ref{Eq2.8}) tells us that
$$\#\{i: x_i\le y_s\}-\#\{j: y_j\le y_s\}\ge 0.$$
Thus one derives that
$$\#\{i: x_i\le a\}-\#\{j: y_j\le a\}
\ge\#\{i: x_i\le y_s\}-\#\{j: y_j\le y_s\}\ge 0.$$
Hence (\ref{Eq2.7}) is true in this case.
This ends the proof of Lemma \ref{lem2.4}.
\end{proof}

For $\bm{\alpha}=(\alpha_1,...,\alpha_{r})$ and
$\bm{\beta}=(\beta_1,...\beta_{s})$ with
$\alpha_i,\beta_j\in\mathbb{Q}\setminus\mathbb{Z}_{\leq0}$,
we recall that $d_{\bm{\alpha,\beta}}$ is the least common
multiple of the exact denominators of elements of
$\bm{\alpha}$ and $\bm{\beta}$ and $M_{\bm{\alpha,\beta}}$
is given by (\ref{Eq1.5}). Throughout this paper, we always let
\begin{align} \label{eq2.0.4}
\Psi:=\{\alpha_1,...,\alpha_r,\beta_1,...,\beta_s\}.
\end{align}
For $\gamma\in\Psi$, one has
$$1-\gamma=1-(\gamma-\langle\gamma\rangle+\langle\gamma\rangle)
=-\gamma+\langle\gamma\rangle+1-\langle\gamma\rangle.$$
By definition of $\langle\gamma\rangle$, $-\gamma+\langle\gamma\rangle$ is an integer and
$1-\langle\gamma\rangle\in[0,1)$. Thus $\lfloor 1-\gamma \rfloor=-\gamma+\langle\gamma\rangle$.
If $\gamma\le0$, then $|\lfloor1-\gamma\rfloor|+\langle\gamma\rangle=-\gamma+2\langle\gamma\rangle$.
Evidently, $\gamma-\langle\gamma\rangle$ is the largest integer $<\gamma$. If $\gamma>0$, then
$\gamma-\langle\gamma\rangle\ge 0$, so in this case, one has
$|\lfloor1-\gamma\rfloor|+\langle\gamma\rangle=\gamma$.
Therefore for any $\gamma\in\Psi $,
we have
\begin{align} \label{eq2.0.5}
|\lfloor1-\gamma\rfloor|+\langle\gamma\rangle\le|\gamma|+2,
\end{align}
and so
\begin{equation} \label{Eq2.9}
M_{\bm{\alpha,\beta}}>d(\gamma)(|\lfloor1-\gamma \rfloor|+\langle\gamma\rangle).
\end{equation}
\begin{lem} \label{lem2.6}
Let $p$ be a prime such that $p>M_{\bm{\alpha,\beta}}$ and $l$
be a positive integer. Let $a\in \{1,...,d_{\bm{\alpha,\beta}}\}$
satisfy that $ap\equiv 1 \mod d_{\bm{\alpha,\beta}}$.
Then the following statements are equivalent:

{\rm (i).} For all $x\in\mathbb{R}$, one has $\delta_{\bm{\alpha,\beta}}(x, a^l)\ge 0$.

{\rm (ii).} For all $k\in\{1,...,s\}$, one has
$\delta_{\bm{\alpha,\beta}}(a^l\beta_k, a^l)\ge 0$.

{\rm (iii).} For all $k\in\{1,...,s\}$, one has
$$\#\Big\{i: \mathfrak{D}_{p}^{l}(\alpha_i)-\frac{\alpha_i}{p^l}
\le\mathfrak{D}_{p}^{l}(\beta_k)-\frac{\beta_k}{p^l} \Big\}
\ge\#\Big\{j: \mathfrak{D}_{p}^{l}(\beta_j)-\frac{\beta_j}{p^l}
\le\mathfrak{D}_{p}^{l}(\beta_k)-\frac{\beta_k}{p^l} \Big\}. $$

{\rm (iv).} For all $n\in\{1,...,p^l\}$, one has
$$\#\Big\{1\le i\le r: \mathfrak{D}_{p}^{l}(\alpha_i)-\frac{\alpha_i}{p^l}\le\frac{n}{p^l} \Big\}
\ge\#\Big\{1\le j\le s: \mathfrak{D}_{p}^{l}(\beta_j)-\frac{\beta_j}{p^l}\le\frac{n}{p^l} \Big\}. $$

{\rm (v).} For all $n\in\{1,...,p^l\}$, one has
$$\sum_{i=1}^{r}\Big\lceil\frac{n+\alpha_i}{p^l}-\mathfrak{D}_{p}^{l}(\alpha_i)
\Big\rceil\ge \sum_{j=1}^{s}\Big\lceil\frac{n+\beta_j}{p^l}-\mathfrak{D}_{p}^{l}
(\beta_j)\Big\rceil.$$
\end{lem}
\noindent {\it Remark.} The equivalent statements in
Lemma \ref{lem2.6} first appears in Christol's proof
of Theorem I. The equivalence of these statements is
the key point of Christol's proof. Note that the proof
given here is different from that of Christol
presented in \cite{[Ch]}.

\begin{proof}
${\rm (i)} \Leftrightarrow {\rm (ii)}$. One has
$$\delta_{\bm{\alpha,\beta}}(x, a^l)=\#\{i: a^l\alpha_i\preccurlyeq x\}
-\#\{j: a^l\beta_j\preccurlyeq x\}$$
and
$$\delta_{\bm{\alpha,\beta}}(a^l\beta_k, a^l)=\#\{i: a^l\alpha_i\preccurlyeq a^l\beta_k\}
-\#\{j: a^l\beta_j\preccurlyeq a^l\beta_k\}.$$
Since $\mathbb{R}$ is a totally ordered set with
respect to the order ``$\preccurlyeq$", then Lemma \ref{lem2.5}
applied to the sets $\mathbb{R}$ and $\{a^l\beta_1, ..., a^l\beta_s\}$
with total order $\preccurlyeq$, the equivalence of parts (i) and (ii)
follows immediately from Lemma \ref{lem2.5}.

${\rm (ii)} \Leftrightarrow {\rm (iii)}$.
Let $\gamma_1$ and $\gamma_2$ be in the set $\Psi$.
We claim that for any given positive integer $b$ with $(b,d_{\bm{\alpha,\beta}})=1$
and $bp^l \equiv 1 \mod d_{\bm{\alpha,\beta}}$, one has that
\begin{equation} \label{Eq2.10}
\mathfrak{D}_p^l(\gamma_1)-\frac{\gamma_1}{p^l}\le\mathfrak{D}_p^l(\gamma_2)-\frac{\gamma_2}{p^l}
\end{equation}
holds if and only if exactly one of the following two is true:

(a). $\mathfrak{D}_p^l(\gamma_1)<\mathfrak{D}_p^l(\gamma_2)$.

(b). $\mathfrak{D}_p^l(\gamma_1)=\mathfrak{D}_p^l(\gamma_2)$ and $\gamma_1\ge\gamma_2$.  \\
In fact, if (b) is true, then obviously (2.10) holds. If (a) is true, then $\mathfrak{D}_p^l(\gamma_2)
-\mathfrak{D}_p^l(\gamma_1)>0$. But Lemma \ref{lem2.3} tells us that $\mathfrak{D}_p^l(\gamma_1)$ and
$\mathfrak{D}_p^l(\gamma_2)$ are contained in the set $\{\frac{1}{d_{\bm{\alpha,\beta}}},
\frac{2}{d_{\bm{\alpha,\beta}}},...,1\}$. Hence
$$\mathfrak{D}_p^l(\gamma_2)-\mathfrak{D}_p^l(\gamma_1)\ge\frac{1}{d_{\bm{\alpha,\beta}}}.$$
Since $p^l\ge p>M_{\bm{\alpha,\beta}}>{d_{\bm{\alpha,\beta}}}(|\gamma_1|+|\gamma_2|)$, we can deduce that
\begin{align} \label{Eq2.11}
\mathfrak{D}_p^l(\gamma_2)-\frac{\gamma_2}{p^l}-\mathfrak{D}_p^l(\gamma_1)+\frac{\gamma_1}{p^l}
& \ge\frac{1}{d_{\bm{\alpha,\beta}}}+\frac{\gamma_1-\gamma_2}{p^l}\\
& \ge\frac{1}{d_{\bm{\alpha,\beta}}}-\frac{|\gamma_1|+|\gamma_2|}{p^l}>0 \nonumber.
\end{align}
Thus (\ref{Eq2.10}) holds if one of (a) and (b) is true. Conversely, assume that (\ref{Eq2.10}) holds. If
$\mathfrak{D}_p^l(\gamma_1)>\mathfrak{D}_p^l(\gamma_2)$, then in the same argument as in proving (\ref{Eq2.11}),
one can show that
$$\mathfrak{D}_p^l(\gamma_2)-\frac{\gamma_2}{p^l}-\mathfrak{D}_p^l(\gamma_1)+\frac{\gamma_1}{p^l}
\le-\frac{1}{d_{\bm{\alpha,\beta}}}+\frac{\gamma_1-\gamma_2}{p^l}<0,$$
which contradicts to (\ref{Eq2.10}). Therefore $\mathfrak{D}_p^l(\gamma_1)\le\mathfrak{D}_p^l(\gamma_2)$.
Furthermore, if $\mathfrak{D}_p^l(\gamma_1)=\mathfrak{D}_p^l(\gamma_2)$, then one derives from (\ref{Eq2.10})
that $\gamma_1\ge\gamma_2$. Hence one of (a) and (b) must be true. This proves the truth of the claim.

Now by the definition of $\preccurlyeq$, one knows that $b\gamma_1\preccurlyeq b\gamma_2$ holds
if and only if either $\langle b\gamma_1\rangle=\langle b\gamma_2\rangle$ and $b\gamma_1\ge b\gamma_2$,
or $\langle b\gamma_1\rangle<\langle b\gamma_2\rangle$, by Lemma \ref{lem2.3}, which is equivalent to
either $\mathfrak{D}_p^l(\gamma_1)=\mathfrak{D}_p^l(\gamma_2)$ and $b\gamma_1\ge b\gamma_2$,
or $\mathfrak{D}_p^l(\gamma_1)<\mathfrak{D}_p^l(\gamma_2)$. Since $b>0$, one then deduces that
$b\gamma_1\preccurlyeq b\gamma_2$ holds if and only if exactly one of (a) and (b) is true,
hence if and only if (\ref{Eq2.10}) is true by the claim. That is, one has
\begin{equation} \label{Eq2.12}
b\gamma_1\preccurlyeq b\gamma_2\Longleftrightarrow \mathfrak{D}_{p}^{l}(\gamma_1)
-\frac{\gamma_1}{p^l}\le\mathfrak{D}_{p}^{l}(\gamma_2)-\frac{\gamma_2}{p^l}.
\end{equation}
So by (\ref{Eq2.12}), for $k\in\{1,...,s\}$, we have
$$\#\{i:b\alpha_i\preccurlyeq b\beta_k\}=\#\Big\{i: \mathfrak{D}_{p}^{l}(\alpha_i)
-\frac{\alpha_i}{p^l}\le\mathfrak{D}_{p}^{l}(\beta_k)-\frac{\beta_k}{p^l} \Big\}$$
and
$$\#\{i:b\beta_j\preccurlyeq b\beta_k\}=\#\Big\{i: \mathfrak{D}_{p}^{l}(\beta_j)
-\frac{\beta_j}{p^l}\le\mathfrak{D}_{p}^{l}(\beta_k)-\frac{\beta_k}{p^l} \Big\}.$$
Thereby, one has
\begin{align*}
\delta_{\bm{\alpha,\beta}}(b\beta_k,b)
=&\#\Big\{i: \mathfrak{D}_{p}^{l}(\alpha_i)-\frac{\alpha_i}{p^l}
\le\mathfrak{D}_{p}^{l}(\beta_k)-\frac{\beta_k}{p^l} \Big\} \\
&-\#\Big\{j: \mathfrak{D}_{p}^{l}(\beta_j)-\frac{\beta_j}{p^l}
\le\mathfrak{D}_{p}^{l}(\beta_k)-\frac{\beta_k}{p^l} \Big\}.
\end{align*}
Then taking $b=a^l$ gives the equivalence of parts (ii) and (iii).

${\rm (iii)} \Leftrightarrow {\rm (iv)}$.
Let $A=\{\frac{n}{p^l}: 0\le n\le p^l-1\}$.
For all integers $i$ and $j$ with $1\le i\le r$ and $1\le j\le s$, we set
$$x_i:=\mathfrak{D}_{p}^{l}(\alpha_i)-\frac{\alpha_i}{p^l}
=\frac{T_{p,l}(\alpha_i)}{p^l}$$
and
$$y_j:=\mathfrak{D}_{p}^{l}(\beta_j)-\frac{\beta_j}{p^l}=\frac{T_{p,l}(\beta_j)}{p^l}.$$
Since $0\le T_{p,l}(\alpha_i),T_{p,l}(\beta_j)\le p^l-1$, $x_i$
and $y_j$ are elements of $A$ for all integers $i$ and $j$ with
$1\le i\le r$ and $1\le j\le s$. Hence the equivalence of
parts (iii) and (iv) follows from Lemma \ref{lem2.5}.

${\rm (iv)} \Leftrightarrow {\rm (v)}$.
For $\gamma\in\Psi$ and $n\in\{1,...,p^l\}$,
by (\ref{Eq2.4}) to (\ref{Eq2.6}), one has
$$0\le\Big\lfloor\frac{n}{p^l}\Big\rfloor\le\Big\lceil\frac{n+\gamma}{p^l}
-\mathfrak{D}_p^l(\gamma)\Big\rceil\le\Big\lfloor\frac{n-1}{p^l}\Big\rfloor+1=1.$$
So $\big\lceil\frac{n+\gamma}{p^l}-\mathfrak{D}_p^l(\gamma)\big\rceil=1$
if $\mathfrak{D}_p^l(\gamma)-\frac{\gamma}{p^l}<\frac{n}{p^l}$,
and $\big\lceil\frac{n+\gamma}{p^l}-\mathfrak{D}_p^l(\gamma)\big\rceil=0$ otherwise. Then
\begin{align*}
&\sum_{i=1}^{r}\Big\lceil\frac{n+\alpha_i}{p^l}-\mathfrak{D}_p^l(\alpha_i)
\Big\rceil-\sum_{j=1}^{s}\Big\lceil\frac{n+\beta_j}{p^l}-\mathfrak{D}_p^l(\beta_j)\Big\rceil \\
=&\#\Big\{i: \mathfrak{D}_p^l(\alpha_i)-\frac{\alpha_i}{p^l}\le\frac{n-1}{p^l}\Big\}
-\#\Big\{j: \mathfrak{D}_p^l(\beta_j)-\frac{\beta_j}{p^l}\le\frac{n-1}{p^l}\Big\}.
\end{align*}
Thus the equivalence ${\rm (iv)} \Leftrightarrow {\rm (v)}$ is proved.

This completes the proof of Lemma \ref{lem2.6}.
\end{proof}

\begin{lem} \label{lem2.10}
Let $l\ge 2$ be an integer and $p$ be a prime such that $p>M_{\bm{\alpha,\beta}}$.
Let $n$ and $e$ be integers such that $1\le n\le p$ and
$1\le e\le d_{\bm{\alpha,\beta}}$. Then for any pair $(n,e)$ with
either $n=p$ and $e=d_{\bm{\alpha,\beta}}$ or $\frac{e}{d_{\bm{\alpha,\beta}}}<\frac{n}{p}
<\frac{e+1}{d_{\bm{\alpha,\beta}}}$, one has
\begin{align} \label{Eq2.27}
&\sum_{i=1}^{r}\Big\lceil\frac{n}{p}+\frac{\alpha_i}{p^l}
-\mathfrak{D}_p^l(\alpha_i)\Big\rceil
-\sum_{j=1}^{s}\Big\lceil\frac{n}{p}+\frac{\beta_j}{p^l}
-\mathfrak{D}_p^l(\beta_j)\Big\rceil \\
=&\#\Big\{1\le i\le r: \mathfrak{D}_p^l(\alpha_i)\le\frac{e}{d_{\bm{\alpha,\beta}}}\Big\}
-\#\Big\{1\le j\le s: \mathfrak{D}_p^l(\beta_j)\le\frac{e}{d_{\bm{\alpha,\beta}}}\Big\}. \notag
\end{align}
\end{lem}
\begin{proof}
Since $\frac{1}{d_{\bm{\alpha,\beta}}}\le\mathfrak{D}_p^l(\alpha_i),
\mathfrak{D}_p^l(\beta_j)\le 1$, it follows that
$$\#\Big\{i: \mathfrak{D}_p^l(\alpha_i)\le 1\Big\}
-\#\Big\{j: \mathfrak{D}_p^l(\beta_j)\le 1\Big\}=r-s.$$
But $-1<\frac{\gamma}{p^l}-\mathfrak{D}_p^l(\gamma)\le 0$
for any $\gamma\in\Psi $. Then one has
$$\sum_{i=1}^{r}\Big\lceil 1+\frac{\alpha_i}{p^l}
-\mathfrak{D}_p^l(\alpha_i)\Big\rceil
-\sum_{j=1}^{s}\Big\lceil 1+\frac{\beta_j}{p^l}
-\mathfrak{D}_p^l(\beta_j)\Big\rceil=r-s.$$
So (\ref{Eq2.27}) is true for the case that $n=p$ and $e=d_{\bm{\alpha,\beta}}$.

In what follows, we let $1\le n<p$ and $0\le e<d_{\bm{\alpha,\beta}}$ satisfying that
$$\frac{e}{d_{\bm{\alpha,\beta}}}<\frac{n}{p}<\frac{e+1}{d_{\bm{\alpha,\beta}}}.$$
Since $l\ge 2$, it follows from $p>M_{\bm{\alpha,\beta}}>|\gamma|d_{\bm{\alpha,\beta}}$
that for any $\gamma\in\Psi $, we have
$$\frac{n}{p}-\frac{e}{d_{\bm{\alpha,\beta}}}\ge\frac{1}
{pd_{\bm{\alpha,\beta}}}>\frac{|\gamma|}{p^l}$$
and
$$\frac{e+1}{d_{\bm{\alpha,\beta}}}-\frac{n}{p}\ge\frac{1}
{pd_{\bm{\alpha,\beta}}}>\frac{|\gamma|}{p^l}.$$
Noticing that $-1<\frac{\gamma}{p^l}-\mathfrak{D}_p^l(\gamma)\le 0$, one concludes that
$$-1<\frac{n}{p}-1<\frac{n}{p}+\frac{\gamma}{p^l}-\mathfrak{D}_p^l(\gamma)\le \frac{n}{p}<1.$$

If $\mathfrak{D}_p^l(\gamma)\le \frac{e}{d_{\bm{\alpha,\beta}}}$, then one has
$$\frac{n}{p}+\frac{\gamma}{p^l}-\mathfrak{D}_p^l(\gamma)\ge \frac{n}{p}+\frac{\gamma}{p^l}
-\frac{e}{d_{\bm{\alpha,\beta}}}>\frac{\gamma+|\gamma|}{p^l}\ge 0.$$
If $\mathfrak{D}_p^l(\gamma)>\frac{e}{d_{\bm{\alpha,\beta}}}$, then
$\mathfrak{D}_p^l(\gamma)\ge\frac{e+1}{d_{\bm{\alpha,\beta}}}$ since
$$\mathfrak{D}_p^l(\gamma)\in\Big\{\frac{1}{d_{\bm{\alpha,\beta}}},\frac{2}
{d_{\bm{\alpha,\beta}}},...,\frac{d_{\bm{\alpha,\beta}}}{d_{\bm{\alpha,\beta}}}\Big\}.$$
It then follows that
$$\frac{n}{p}+\frac{\gamma}{p^l}-\mathfrak{D}_p^l(\gamma)\le \frac{\gamma}{p^l}
-\Big(\frac{e+1}{d_{\bm{\alpha,\beta}}}-\frac{n}{p}\Big)<\frac{\gamma-|\gamma|}{p^l}\le 0.$$
Therefore
\begin{align*}
\Big\lceil\frac{n}{p}+\frac{\gamma}{p^l}-\mathfrak{D}_p^l(\gamma)\Big\rceil
=\bigg\{\begin{array}{cl}
       1, & {\rm if}\  \mathfrak{D}_p^l(\gamma)\le \frac{e}{d_{\bm{\alpha,\beta}}}, \\
       0, & {\rm otherwise,}
     \end{array}
\end{align*}
which implies that
$$\sum_{i=1}^{r}\Big\lceil\frac{n}{p}+\frac{\alpha_i}{p^l}-\mathfrak{D}_p^l(\alpha_i)\Big\rceil
=\#\Big\{1\le i\le r: \mathfrak{D}_p^l(\alpha_i)\le\frac{e}{d_{\bm{\alpha,\beta}}}\Big\}$$
and
$$\sum_{j=1}^{s}\Big\lceil\frac{n}{p}+\frac{\beta_j}{p^l}-\mathfrak{D}_p^l(\beta_j)\Big\rceil
=\#\Big\{1\le j\le s: \mathfrak{D}_p^l(\beta_j)\le\frac{e}{d_{\bm{\alpha,\beta}}}\Big\}.$$
So the desired result (\ref{Eq2.27}) follows immediately.
The proof of Lemma \ref{lem2.10} is complete.
\end{proof}

The following lemma is an analogue to the equivalence of parts
(ii) and (v) presented in Lemma \ref{lem2.6}.

\begin{lem} \label{lem2.7}
Let $l\ge 2$ be an integer and $p$ be a prime such that $p>M_{\bm{\alpha,\beta}}$.
Let $a\in \{1,...,d_{\bm{\alpha,\beta}}\}$ satisfy that $ap\equiv 1 \mod d_{\bm{\alpha,\beta}}$.
Then the following three statements are equivalent:

{\rm (i).} For all integers $n$ with $1\le n\le p$, one has
\begin{equation} \label{Eq2.13}
\sum_{i=1}^{r}\Big\lceil\frac{n}{p}+\frac{\alpha_i}{p^l}-\mathfrak{D}_p^l(\alpha_i)
\Big\rceil-\sum_{j=1}^{s}\Big\lceil\frac{n}{p}+\frac{\beta_j}{p^l}
-\mathfrak{D}_p^l(\beta_j)\Big\rceil\ge 0.
\end{equation}

{\rm (ii).} For all $e\in\{1,...,d_{\bm{\alpha,\beta}}\}$, one has
\begin{equation} \label{Eq2.14}
\#\Big\{1\le i\le r: \mathfrak{D}_p^l(\alpha_i)\le\frac{e}{d_{\bm{\alpha,\beta}}}\Big\}
-\#\Big\{1\le j\le s: \mathfrak{D}_p^l(\beta_j)\le\frac{e}{d_{\bm{\alpha,\beta}}}\Big\}\ge 0.
\end{equation}

{\rm (iii).} For all $e\in\{1,...,d_{\bm{\alpha,\beta}}\}$, (\ref{Eq1.8}) holds.
\end{lem}

\begin{proof}
At first, we show that parts (i) and (ii) are equivalent.
First, we suppose that part (i) is true. In the following
we prove that part (ii) is true.
Obviously (\ref{Eq2.14}) is true for $e=d_{\bm{\alpha,\beta}}$ by Lemma \ref{lem2.10}.
Since for any $e\in\{1,...,d_{\bm{\alpha,\beta}-1}\}$, one can find an $n\in\{1,...,p-1\}$
such that
\begin{align} \label{Eq2.14'}
\frac{e}{d_{\bm{\alpha,\beta}}}<\frac{n}{p}<\frac{e+1}{d_{\bm{\alpha,\beta}}},
\end{align}
it follows from Lemma \ref{lem2.10} that (\ref{Eq2.14}) holds for
$1\le e\le d_{\bm{\alpha,\beta}}-1$.
Hence part (ii) is true. Conversely, we assume that part (ii)
holds and show that part (i) is true.
Since $(\ref{Eq2.14})$ holds when $e=d_{\bm{\alpha,\beta}}$,
we have that (2.13) holds for $n=p$ by Lemma \ref{lem2.10}.
For $1\le n<p$, there is an $e\in\{0,...,d_{\bm{\alpha,\beta}}-1\}$
such that (\ref{Eq2.14'}) holds. Thus by Lemma \ref{lem2.10},
$(\ref{Eq2.13})$ is true. In other words, part (i) is true.
The proof of the equivalence of parts (i) and (ii) is complete.

Consequently we show the equivalence of parts (ii) and (iii).
Let $\gamma\in\Psi $
and $e\in\{1,...,d_{\bm{\alpha,\beta}}\}$. If
$$\langle a^l\gamma\rangle=\Big\langle\frac{a^le}
{d_{\bm{\alpha,\beta}}}+a^lm_{\bm{\alpha,\beta}}\Big\rangle,$$
then by Lemma \ref{lem2.4}, we have $\langle\gamma \rangle=\frac{e}{d_{\bm{\alpha,\beta}}}$.
But the definition (\ref{Eq1.6}) gives us that
$m_{\bm{\alpha,\beta}}\le\gamma-\langle\gamma\rangle$.
It then follows that
$$a^l\Big(\frac{e}{d_{\bm{\alpha,\beta}}}+m_{\bm{\alpha,\beta}}\Big)\le a^l\gamma.$$
Hence
$$a^l\gamma\preccurlyeq a^l\Big(\frac{e}{d_{\bm{\alpha,\beta}}}+m_{\bm{\alpha,\beta}}\Big)$$
holds for all $\gamma\in\Psi $, if and only if
one has either
$$\langle a^l\gamma\rangle <\Big\langle a^l\Big(\frac{e}{d_{\bm{\alpha,\beta}}}+m_{\bm{\alpha,\beta}}\Big)\Big\rangle$$
or
$$\langle a^l\gamma\rangle =\Big\langle a^l\Big(\frac{e}{d_{\bm{\alpha,\beta}}}+m_{\bm{\alpha,\beta}}\Big)\Big\rangle
\ {\rm and} \ a^l\Big(\frac{e}{d_{\bm{\alpha,\beta}}}+m_{\bm{\alpha,\beta}}\Big)\le a^l\gamma,$$
if and only if
$$\langle a^l\gamma\rangle \le \Big\langle a^l\Big(\frac{e}{d_{\bm{\alpha,\beta}}}+m_{\bm{\alpha,\beta}}\Big)\Big\rangle
=\Big\langle \frac{a^le}{d_{\bm{\alpha,\beta}}}\Big\rangle.$$
Thus one can derive that
\begin{align*}
&\delta_{\bm{\alpha,\beta}}\Big(a^l\big(\frac{e}{d_{\bm{\alpha,\beta}}}
+m_{\bm{\alpha,\beta}}\big), a^l\Big) \\
=&\#\Big\{i: a^l\alpha_i\preccurlyeq a^l\Big(\frac{e}{d_{\bm{\alpha,\beta}}}
+m_{\bm{\alpha,\beta}}\Big)\Big\}
-\#\Big\{j: a^l\beta_j\preccurlyeq a^l\Big(\frac{e}{d_{\bm{\alpha,\beta}}}
+m_{\bm{\alpha,\beta}}\Big)\Big\} \\
=&\#\Big\{i: \langle a^l\alpha_i\rangle\le\Big\langle\frac{a^le}{d_{\bm{\alpha,\beta}}}
\Big\rangle\Big\}-\#\Big\{j: \langle a^l\beta_j\rangle
\le\Big\langle\frac{a^le}{d_{\bm{\alpha,\beta}}}\Big\rangle\Big\} \\
=&\#\Big\{i: \mathfrak{D}_p^l(\alpha_i)\le\Big\langle\frac{a^le}
{d_{\bm{\alpha,\beta}}}\Big\rangle\Big\}
-\#\Big\{j: \mathfrak{D}_p^l(\beta_j)\le\Big\langle\frac{a^le}
{d_{\bm{\alpha,\beta}}}\Big\rangle\Big\},
\end{align*}
the last equality follows from Lemma \ref{lem2.3}.
Therefore part (iii) is true if and only if
$$\#\Big\{i: \mathfrak{D}_p^l(\alpha_i)
\le\Big\langle\frac{a^le}{d_{\bm{\alpha,\beta}}}\Big\rangle\Big\}
-\#\Big\{j: \mathfrak{D}_p^l(\beta_j)
\le\Big\langle\frac{a^le}{d_{\bm{\alpha,\beta}}}\Big\rangle\Big\}\ge0$$
for all $e\in\{1,...,d_{\bm{\alpha,\beta}}\}$.

Since $(a,d_{\bm{\alpha,\beta}})=1$, then for every
$e\in\{1,...,d_{\bm{\alpha,\beta}}\}$, one can find
a unique $e'\in\{1,...,d_{\bm{\alpha,\beta}}\}$ such that
\begin{equation} \label{Eq2.28}
a^le\equiv e' \mod d_{\bm{\alpha,\beta}}.
\end{equation}
Conversely, for each $e'\in\{1,...,d_{\bm{\alpha,\beta}}\}$, one can find
a unique $e\in\{1,...,d_{\bm{\alpha,\beta}}\}$ such that (\ref{Eq2.28}) is true.
Now for any pair $(e, e')$ with $e, e'\in\{1,...,d_{\bm{\alpha,\beta}}\}$
such that (\ref{Eq2.28}) holds, one has that $\frac{a^le}{d_{\bm{\alpha,\beta}}}
-\frac{e'}{d_{\bm{\alpha,\beta}}}$ is an integer, which implies that
$$\frac{e'}{d_{\bm{\alpha,\beta}}}=\Big\langle\frac{a^le}{d_{\bm{\alpha,\beta}}}\Big\rangle.$$
Therefore
\begin{align*}
&\#\Big\{i: \mathfrak{D}_p^l(\alpha_i)\le\Big\langle\frac{a^le}
{d_{\bm{\alpha,\beta}}}\Big\rangle\Big\}
-\#\Big\{j: \mathfrak{D}_p^l(\beta_j)\le\Big\langle\frac{a^le}
{d_{\bm{\alpha,\beta}}}\Big\rangle\Big\} \\
=&\#\Big\{i: \mathfrak{D}_p^l(\alpha_i)\le\frac{e'}{d_{\bm{\alpha,\beta}}}\Big\}
-\#\Big\{j: \mathfrak{D}_p^l(\beta_j)\le\frac{e'}{d_{\bm{\alpha,\beta}}}\Big\}.
\end{align*}
Hence
$$\#\Big\{i: \mathfrak{D}_p^l(\alpha_i)
\le\Big\langle\frac{a^le}{d_{\bm{\alpha,\beta}}}\Big\rangle\Big\}
-\#\Big\{j: \mathfrak{D}_p^l(\beta_j)
\le\Big\langle\frac{a^le}{d_{\bm{\alpha,\beta}}}\Big\rangle\Big\}\ge0$$
for all $e\in\{1,...,d_{\bm{\alpha,\beta}}\}$ if and only if
$$\#\Big\{i: \mathfrak{D}_p^l(\alpha_i)\le\frac{e'}{d_{\bm{\alpha,\beta}}}\Big\}
-\#\Big\{j: \mathfrak{D}_p^l(\beta_j)\le\frac{e'}{d_{\bm{\alpha,\beta}}}\Big\}\ge0$$
for all $e'\in\{1,...,d_{\bm{\alpha,\beta}}\}$.
This proves the equivalence of parts (ii) and (iii).
So Lemma \ref{lem2.7} is proved.
\end{proof}

For positive integers $n$ and $l$, we define two subsets of
$\Psi=\{\alpha_1,...,\alpha_r,\beta_1,...,\beta_s\}$ as follows:
\begin{equation} \label{Eq2.33}
A_l(n):=\Big\{\gamma\in\Psi:
\mathfrak{D}_p^l(\gamma)-\frac{\gamma}{p^l}<\Big\langle\frac{n}{p^l}\Big\rangle\Big\},
\end{equation}
\begin{equation} \label{Eq2.34}
B_l(n):=\Big\{\gamma\in\Psi:
\mathfrak{D}_p^l(\gamma)-\frac{\gamma}{p^l}\ge\Big\langle\frac{n}{p^l}\Big\rangle\Big\}.
\end{equation}
Then $A_l(n)\cap B_l(n)=\emptyset$ and $A_l(n)\cup B_L(n)=\Psi$.
It is possible that $A_l(n)$ or $B_l(n)$ is empty.
If $A_l(n)$ and $B_l(n)$ are nonempty, then let $\xi_l(n)\in A_l(n)$ and $\eta _l(n)\in B_l(n)$
satisfy
\begin{equation} \label{Eq2.35}
\mathfrak{D}_p^l(\xi_l(n))-\frac{\xi_l(n)}{p^l}
=\max_{\gamma\in A_l(n)}\Big\{\mathfrak{D}_p^l(\gamma)-\frac{\gamma}{p^l}\Big\}
\end{equation}
and
\begin{equation} \label{Eq2.36}
\mathfrak{D}_p^l(\eta_l(n))-\frac{\eta_l(n)}{p^l}
=\min_{\gamma\in B_l(n)}\Big\{\mathfrak{D}_p^l(\gamma)-\frac{\gamma}{p^l}\Big\},
\end{equation}
respectively. We have the following result.
\begin{lem} \label{lem2.8}
Let $n$ and $L$ be positive integers such that both of $A_L(n)$ and $B_L(n)$
are nonempty and $\mathfrak{D}_p^L(\xi_L(n))=\mathfrak{D}_p^L(\eta_L(n)).$
Then for all positive integers $l$ with $l\le L$, $\xi_L(n)\in A_l(n)$ and $\eta_L(n)\in B_l(n)$.
Furthermore, we have
$$\mathfrak{D}_p^l(\xi_L(n))-\frac{\xi_L(n)}{p^l}=\max_{\gamma\in A_l(n)}
\Big\{\mathfrak{D}_p^l(\gamma)-\frac{\gamma}{p^l}\Big\}$$
and
$$\mathfrak{D}_p^l(\eta_L(n))-\frac{\eta_L(n)}{p^l}=\min_{\gamma\in B_l(n)}
\Big\{\mathfrak{D}_p^l(\gamma)-\frac{\gamma}{p^l}\Big\}.$$
\end{lem}
\begin{proof}
For brevity, we write $\xi_L:=\xi_L(n)$ and $\eta_L:=\eta_L(n)$. Obviously, we have
\begin{equation} \label{Eq2.15}
p^L\mathfrak{D}_p^L(\xi_L)-\xi_L=\sum_{l=0}^{L-1}\big(p\mathfrak{D}_p^{l+1}(\xi_L)
-\mathfrak{D}_p^{l}(\xi_L)\big)p^{l},
\end{equation}
where $\mathfrak{D}_p^0(\xi_L):=\xi_L$. Since
$$p\mathfrak{D}_p^{l+1}(\xi_L)-\mathfrak{D}_p^{l}(\xi_L)=p\mathfrak{D}_p(\mathfrak{D}_p^l
(\xi_L))-\mathfrak{D}_p^{l}(\xi_L)\in\{0, 1, ..., p-1\},$$
it follows that (\ref{Eq2.15}) is the $p$-adic expansion of $p^L\mathfrak{D}_p^L(\xi_L)-\xi_L$.
Likewise,
\begin{equation} \label{Eq2.16}
p^L\mathfrak{D}_p^L(\eta_L)-\eta_L=\sum_{l=0}^{L-1}\big(p\mathfrak{D}_p^{l+1}(\eta_L)
-\mathfrak{D}_p^l(\eta_L)\big)p^{l}
\end{equation}
is the $p$-adic expansion of $p^L\mathfrak{D}_p^L(\eta_L)-\eta_L$.
Furthermore, let $n=\sum_{l=0}^{\infty}n_lp^l$ be the $p$-adic expansion of $n$.
By definition of $\xi_L$ and $\eta_L$, we have
$$\mathfrak{D}_p^L(\xi_L)-\frac{\xi_L}{p^L}<\Big\langle\frac{n}{p^L}\Big\rangle
\le\mathfrak{D}_p^L(\eta_L)-\frac{\eta_L}{p^L}.$$
Since
$$\Big\langle\frac{n}{p^L}\Big\rangle=p^{-L}(n_0+\cdots+n_{L-1}p^{L-1}),$$
one can derive by (\ref{Eq2.15}) and (\ref{Eq2.16}) that
\begin{equation} \label{Eq2.17}
\sum_{l=0}^{L-1}\big(p\mathfrak{D}_p^{l+1}(\xi_L)-\mathfrak{D}_p^l(\xi_L)\big)p^l
<\sum_{l=0}^{L-1}n_lp^l\le\sum_{l=0}^{L-1}\big(p\mathfrak{D}_p^{l+1}(\eta_L)
-\mathfrak{D}_p^l(\eta_L)\big)p^{l}.
\end{equation}
From the hypothesis $\mathfrak{D}_p^L(\xi_L)=\mathfrak{D}_p^L(\eta_L)$
and Lemma \ref{lem2.4}, it follows that
$\mathfrak{D}_p^l(\xi_L)=\mathfrak{D}_p^l(\eta_L)$
for all positive integers $l$ with $l\le L$.
On then deduces from (\ref{Eq2.17}) that
$$p\mathfrak{D}_p(\xi_L)-\xi_L<n_0\le p\mathfrak{D}_p(\eta_L)-\eta_L,$$
and for all integers $l$ with $1\le l\le L-1$,
$$p\mathfrak{D}_p^{l+1}(\xi_L)-\mathfrak{D}_p^l(\xi_L)=n_l=p\mathfrak{D}_p^{l+1}
(\eta_L)-\mathfrak{D}_p^l(\eta_L).$$
It follows that for all positive integers $l$ with $l\le L$, we have
\begin{align} \label{Eq2.29}
\mathfrak{D}_p^l(\xi_L)-\frac{\xi_L}{p^l}&=\frac{1}{p^l}\sum_{i=0}^{l-1}\big(p\mathfrak{D}_p^{i+1}(\xi_L)
-\mathfrak{D}_p^i(\xi_L)\big)p^i\\
&<\frac{1}{p^l}\sum_{i=0}^{l-1}n_ip^i =\Big\langle\frac{n}{p^l}\Big\rangle \notag \\
&\le\frac{1}{p^l}\sum_{i=0}^{l-1}\big(p\mathfrak{D}_p^{i+1}(\eta_L)
-\mathfrak{D}_p^i(\eta_L)\big)p^i=\mathfrak{D}_p^l(\eta_L)-\frac{\eta_L}{p^l}. \notag
\end{align}
Therefore if $l\le L$, then $\xi_L\in A_l(n)$ and $\eta_L\in B_l(n)$.

Now to finish the proof of Lemma \ref{lem2.8}, it remains to show that
if $\gamma\in A_l(n)$, then we have
\begin{equation} \label{Eq2.18}
\mathfrak{D}_p^l(\gamma)-\frac{\gamma}{p^l}\le\mathfrak{D}_p^l(\xi_L)-\frac{\xi_L}{p^l},
\end{equation}
and if $\gamma\in B_l(n)$, then we have
\begin{equation} \label{Eq2.19}
\mathfrak{D}_p^l(\gamma)-\frac{\gamma}{p^l}\ge\mathfrak{D}_p^l(\eta_L)-\frac{\eta_L}{p^l}.
\end{equation}
But $A_L(n)\cup B_L(n)=\Psi$ and $A_L(n)\cap B_L(n)=\emptyset$.
So it is sufficient to show that for any $\gamma\in\Psi$,
either (\ref{Eq2.18}) holds, or (\ref{Eq2.19}) holds,
which will be done in what follows.

Pick $\gamma\in\Psi$. By the proof of the equivalence
of parts (ii) and (iii) of Lemma \ref{lem2.6}, we know that
$\mathfrak{D}_p^l(\gamma)-\frac{\gamma}{p^l}<\mathfrak{D}_p^l(\xi_L)-\frac{\xi_L}{p^l}$
if $\mathfrak{D}_p^l(\gamma)<\mathfrak{D}_p^l(\xi_L)$, and
$\mathfrak{D}_p^l(\gamma)-\frac{\gamma}{p^l}>\mathfrak{D}_p^l(\eta_L)
-\frac{\eta_L}{p^l}$ if $\mathfrak{D}_p^l(\gamma)>\mathfrak{D}_p^l(\eta_L)$.
Thus either (\ref{Eq2.18}) or (\ref{Eq2.19}) is true for all $\gamma$ satisfying
$\mathfrak{D}_p^l(\gamma)\neq\mathfrak{D}_p^l(\xi_L)$.

Now let $\mathfrak{D}_p^l(\gamma)=\mathfrak{D}_p^l(\xi_L)
(=\mathfrak{D}_p^l(\eta_L))$. Then by Lemma \ref{lem2.4}, we have
\begin{equation} \label{Eq2.30}
\mathfrak{D}_p^L(\gamma)=\mathfrak{D}_p^L(\xi_L)=\mathfrak{D}_p^L(\eta_L).
\end{equation}
Since $A_L(n)\cup B_L(n)=\Psi$ and $A_L(n)\cap B_L(n)=\emptyset$,
by the definitions of $\xi_L$ and $\eta_L$, we have either
$$\mathfrak{D}_p^L(\gamma)-\frac{\gamma}{p^L}\le\mathfrak{D}_p^L(\xi_L)-\frac{\xi_L}{p^L}$$
or
$$\mathfrak{D}_p^L(\gamma)-\frac{\gamma}{p^L}\ge\mathfrak{D}_p^L(\eta_L)-\frac{\eta_L}{p^L}.$$
Then by (\ref{Eq2.30}), one can derive that either $\gamma\ge\xi_L$ or $\gamma\le\eta_L$.
Hence for $\gamma$ with $\mathfrak{D}_p^l(\gamma)=\mathfrak{D}_p^l(\xi_L)$, we have that
either (\ref{Eq2.18}) or (\ref{Eq2.19}) is true.
Therefore $\xi_l$ and $\eta_l$ exist with $\xi_l=\xi_L$ and $\eta_l=\eta_L$.
This finishes the proof of Lemma \ref{lem2.8}.
\end{proof}

\begin{lem} \label{lem2.9}
Let $a\in\{1,...,d_{\bm{\alpha,\beta}}\}$ be coprime to $d_{\bm{\alpha,\beta}}$
and ${\rm ord}(a)$ be the order of $a$ modulo $d_{\bm{\alpha,\beta}}$.
Then each of the following is true.

{\rm (i).} Let $l_1$ and $l_2$ be two positive integers such that $l_1\equiv l_2\ {\rm mod\ ord}(a)$.
Then for any $\gamma\in\Psi $, one has
$\delta_{\bm{\alpha,\beta}}(a^{l_1}\gamma, a^{l_1})
=\delta_{\bm{\alpha,\beta}}(a^{l_2}\gamma, a^{l_2}).$

{\rm (ii).} Let $l$ be a positive integer and $\alpha\in\{\alpha_1,...,\alpha_r\}$.
If the set $J_{\alpha}:=\{\beta_j: 1\le j\le s, \langle\beta_j\rangle=\langle\alpha\rangle \ and \ \beta_j\ge \alpha\}$
is nonempty, then
$\delta_{\bm{\alpha,\beta}}(a^l\alpha, a^l)\ge\delta_{\bm{\alpha,\beta}}(a^l\beta, a^l)$
with $\beta:=\min(J_{\alpha})$. If the set $J_{\alpha}$ is empty, then
$$\delta_{\bm{\alpha,\beta}}(a^l\alpha, a^l)\ge\#\{1\le i\le r: \mathfrak{D}_p^l(\alpha_i)
<\mathfrak{D}_p^l(\alpha)\}-\#\{1\le j\le s: \mathfrak{D}_p^l(\beta_j)<\mathfrak{D}_p^l(\alpha)\}.$$
\end{lem}

\begin{proof}
(i). Since $l_1\equiv l_2\ {\rm mod\ ord}(a)$, one has
$a^{{\rm ord}(a)}\equiv 1\mod d_{\bm{\alpha,\beta}}$, and so
$a^{l_1}\equiv a^{l_2}\mod d_{\bm{\alpha,\beta}}$.
Let $\gamma'$ be any element of the set $\Psi$.
Then the exact denominator of $\gamma'$ divides $d_{\bm{\alpha,\beta}}$.
It follows that $a^{l_1}\gamma'-a^{l_2}\gamma'$ is an integer.
Hence $\langle a^{l_1}\gamma'\rangle=\langle a^{l_2}\gamma'\rangle$.
So one has
\begin{align*}
a^{l_1}\gamma'\preccurlyeq a^{l_1}\gamma&\Longleftrightarrow\langle a^{l_1}\gamma'\rangle
<\langle a^{l_1}\gamma\rangle\ \ {\rm or}\ \ (\langle a^{l_1}\gamma'\rangle
=\langle a^{l_1}\gamma\rangle\ {\rm and}\ a^{l_1}\gamma'\ge a^{l_1}\gamma) \\
&\Longleftrightarrow\langle a^{l_2}\gamma'\rangle<\langle a^{l_2}\gamma\rangle\ \ {\rm or}\ \
(\langle a^{l_2}\gamma'\rangle=\langle a^{l_2}\gamma\rangle\ {\rm and}\
a^{l_2}\gamma'\ge a^{l_2}\gamma) \\
&\Longleftrightarrow a^{l_2}\gamma'\preccurlyeq a^{l_2}\gamma.
\end{align*}
Thus we have
\begin{align*}
\delta_{\bm{\alpha,\beta}}(a^{l_1}\gamma, a^{l_1})
&=\#\{1\le i\le r: a^{l_1}\alpha_i\preccurlyeq a^{l_1}\gamma\}
-\#\{1\le j\le s: a^{l_1}\beta_j\preccurlyeq a^{l_1}\gamma\} \\
&=\#\{1\le i\le r: a^{l_2}\alpha_i\preccurlyeq a^{l_2}\gamma\}
-\#\{1\le j\le s: a^{l_2}\beta_j\preccurlyeq a^{l_2}\gamma\} \\
&=\delta_{\bm{\alpha,\beta}}(a^{l_2}\gamma, a^{l_2})
\end{align*}
as desired. Part (i) is proved.

(ii). First let $J_{\alpha}$ be nonempty. Since $\beta=\min(J_{\alpha})$,
one has
$\langle\beta\rangle=\langle\alpha\rangle$ and $\beta \ge \alpha$.
So by Lemma \ref{lem2.4}, one can derive that $\langle a^l\beta\rangle
=\langle a^l\alpha\rangle$ and $a^l\beta \ge a^l\alpha$, that is,
$a^l\beta\preccurlyeq a^l\alpha$. It follows that
$$\#\{1\le i\le r: a^l\alpha_i\preccurlyeq a^l\alpha\}\ge \#\{1\le i\le r: a^l\alpha_i\preccurlyeq a^l\beta\}.$$

Now we claim that for any $\beta'\in\{\beta_1,...,\beta_s\}$, $a^l\beta'\preccurlyeq a^l\alpha$ holds
if and only if $a^l\beta'\preccurlyeq a^l\beta$ holds. First we let $a^l\beta'\preccurlyeq a^l\beta$.
Then it is clear that $a^l\beta'\preccurlyeq a^l\alpha$ since $a^l\beta\preccurlyeq a^l\alpha$.
Contrarily, let $a^l\beta'\preccurlyeq a^l\alpha$. Then either $\langle a^l\beta'\rangle<\langle a^l\alpha\rangle$,
or $\langle a^l\beta'\rangle=\langle a^l\alpha\rangle$ and $a^l\beta '\ge a^l\alpha$.
For the former case, one can deduce that
$\langle a^l\beta'\rangle<\langle a^l\beta\rangle$ since $\langle a^l\beta\rangle=\langle a^l\alpha\rangle$.
This means that $a^l\beta'\preccurlyeq a^l\beta$ as required. For the latter case, we have
$\langle a^l\beta'\rangle=\langle a^l\alpha\rangle$ and $\beta '\ge \alpha$.
Then by Lemma \ref{lem2.4} (i) and noting that $a^l$ is coprime $d_{\bm{\alpha, \beta}}$, one has
$\langle \beta'\rangle=\langle \alpha\rangle$. Thus $\beta '\in J_{\alpha}$. It then follows from
$\beta=\min(J_{\alpha})$ that $\beta'\ge \beta$ and $\langle \beta'\rangle=\langle \beta\rangle$.
Since $a^l$ is a positive integer, one then derives that $\langle a^l\beta'\rangle=\langle a^l\beta\rangle$
and $a^l\beta'\ge a^l\beta$. In other words, one has $a^l\beta'\preccurlyeq a^l\beta$ as desired.
Therefore
$$\#\{1\le j\le s: a^l\beta_j\preccurlyeq a^l\alpha\}=\{1\le j\le s: a^l\beta_j\preccurlyeq a^l\beta\}.$$
So we can conclude that
\begin{align*}
\delta_{\bm{\alpha,\beta}}(a^l\alpha, a^l)
&=\#\{1\le i\le r: a^l\alpha_i\preccurlyeq a^l\alpha\}
-\#\{1\le j\le s: a^l\beta_j\preccurlyeq a^l\alpha\} \\
&\ge\#\{1\le i\le r: a^l\alpha_i\preccurlyeq a^l\beta\}
-\#\{1\le j\le s: a^l\beta_j\preccurlyeq a^l\beta\} \\
&=\delta_{\bm{\alpha,\beta}}(a^l\beta, a^l)
\end{align*}
as required.

If $J_{\alpha}$ is empty, then $\langle\beta_j\rangle\neq\langle\alpha\rangle$
for all integers $j$ with $1\le j\le s$. By Lemma {\ref{lem2.4}}, we deduce that
$\langle a^l\beta_j\rangle\neq\langle a^l\alpha\rangle$ for all
integers $j$ with $1\le j\le s$. Thus $a^l\beta_j\preccurlyeq a^l\alpha$
if and only if
$\mathfrak{D}_p^l(\beta_j)=\langle a^l\beta_j\rangle
<\langle a^l\alpha\rangle=\mathfrak{D}_p^l(\alpha)$.
Note that
$$\{1\le i\le r: a^l\alpha_i\preccurlyeq a^l\alpha\}\supseteq
\{1\le i\le r: \mathfrak{D}_p^l(\alpha_i)<\mathfrak{D}_p^l(\alpha)\}.$$
Hence
\begin{align*}
\delta_{\bm{\alpha,\beta}}(a^l\alpha, a^l)
&=\#\{1\le i\le r: a^l\alpha_i\preccurlyeq a^l\alpha\}
-\#\{j: a^l\beta_j\preccurlyeq a^l\alpha\} \\
&\ge\#\{1\le i\le r: \mathfrak{D}_p^l(\alpha_i)<\mathfrak{D}_p^l(\alpha)\}
-\#\{j: \mathfrak{D}_p^l(\beta_j)<\mathfrak{D}_p^l(\alpha)\}.
\end{align*}
Part (ii) is proved.

This completes the proof of Lemma \ref{lem2.9}.
\end{proof}

As the conclusion of this section, we give the proof
of Theorem \ref{thm1.2}.\\

{\it Proof of Theorem \ref{thm1.2}.} (i). Let $r<s$.
Then by Lemma \ref{lem2.2}, we know that part (i) is true.

(ii). Let $r>s$. Write
$$m:=\max\{|\alpha_i|, |\beta_j|:1\le i\le r, 1\le j\le s\}.$$
Let $\gamma$ be an element in the set $\Psi=\{\alpha_1,...,\alpha_r, \beta_1,...,\beta_s\}$.
For all integers $n$ and $l$ with $1\le n< p^l/d_{\bm{\alpha,\beta}}-m$, since
$$\mathfrak{D}_{p}^{l}(\gamma)\in\Big\{\frac{e}{d_{\bm{\alpha,\beta}}}: 1\le e\le d_{\bm{\alpha,\beta}}\Big\},$$
one has
\begin{equation} \label{Eq2.31}
\frac{n+\gamma}{p^l}<\frac{1}{d_{\bm{\alpha,\beta}}}\le\mathfrak{D}_{p}^{l}(\gamma).
\end{equation}
It then follows from $-1<\frac{\gamma}{p^l}-\mathfrak{D}_{p}^{l}(\gamma)\le 0$ that
\begin{equation} \label{Eq2.20}
\Big\lceil \frac{n+\gamma}{p^l}-\mathfrak{D}_{p}^{l}(\gamma)\Big\rceil=0.
\end{equation}
Then for $1\le n< p^2/d_{\bm{\alpha,\beta}}-m$,
one deduces from (\ref{Eq2.3}) and (\ref{Eq2.20}) that
\begin{equation} \label{Eq2.21}
v_p\Big(\frac{(\alpha_1)_n\cdots(\alpha_r)_n}{(\beta_1)_n\cdots(\beta_s)_n}\Big)
=\sum_{i=1}^{r}\Big\lceil\frac{n+\alpha_i}{p}-\mathfrak{D}_p(\alpha_i)\Big\rceil
-\sum_{j=1}^{s}\Big\lceil\frac{n+\beta_j}{p}-\mathfrak{D}_p(\beta_j)\Big\rceil.
\end{equation}

If $F_{\bm{\alpha,\beta}}(z)\in\mathbb{Z}_p[[z]]$, then the value
of the right-hand side of (\ref{Eq2.21}) is no less than zero for all
integers $n$ with $1\le n\le p$. By Lemma \ref{lem2.6}, one has
$\delta_{\bm{\alpha,\beta}}(x;a)\ge 0$ for any $x\in\mathbb{R}$.

Conversely, suppose that $\delta_{\bm{\alpha,\beta}}(x;a)\ge 0$
for any $x\in\mathbb{R}$. We show that
$F_{\bm{\alpha,\beta}}(z)\in\mathbb{Z}_p[[z]]$.
Equivalently, we show that
\begin{equation} \label{Eq2.32}
v_p\Big(\frac{(\alpha_1)_n\cdots(\alpha_r)_n}
{(\beta_1)_n\cdots(\beta_s)_n}\Big)\ge 0
\end{equation}
for all positive integers $n$. We define the arithmetic function $t$
for a positive integer $n$ by $t(n):=\lceil\log n/\log p\rceil$.
Then it is easy to see that $n<p^l/d_{\bm{\alpha,\beta}}-m$
for all integers $l$ with $l>t(n)$. By using (\ref{Eq2.3}),
(\ref{Eq2.6}) and (\ref{Eq2.20}), we obtain that
\begin{align} \label{Eq2.22}
&v_p\Big(\frac{(\alpha_1)_n\cdots(\alpha_r)_n}{(\beta_1)_n\cdots(\beta_s)_n}\Big)\\
=&\sum_{i=1}^{r}\sum_{l=1}^{t(n)}\Big\lceil\frac{n+\alpha_i}{p^l}
-\mathfrak{D}_{p}^{l}(\alpha_i)\Big\rceil
-\sum_{j=1}^{s}\sum_{l=1}^{t(n)}\Big\lceil\frac{n+\beta_j}{p^l}
-\mathfrak{D}_{p}^{l}(\beta_j)\Big\rceil \notag \\
\ge& (r-s)\sum_{l=1}^{t(n)}\Big\lfloor \frac{n}{p^l}\Big\rfloor-st(n). \notag
\end{align}

If $1\le n< p^2/d_{\bm{\alpha,\beta}}-m$, then by (\ref{Eq2.21})
and Lemma \ref{lem2.6}, one can derive that (\ref{Eq2.32}) is true.

If $p^{2}/d_{\bm{\alpha,\beta}}-m\le n\le p^2$,
since $p>M_{\bm{\alpha,\beta}}$, then one derives that
\begin{align*}
n>& pM_{\bm{\alpha,\beta}}/d_{\bm{\alpha,\beta}}-m\\
=&(2+2m)p-m\\
=&2p+(2p-1)m>2p>p
\end{align*}
and
\begin{align*}
\frac{1}{d_{\bm{\alpha,\beta}}}-\frac{m}{p}
>& \frac{1}{d_{\bm{\alpha,\beta}}}-\frac{m}{M_{\bm{\alpha,\beta}}}\\
=& \frac{1}{d_{\bm{\alpha,\beta}}}-\frac{m}{d_{\bm{\alpha,\beta}}(2+2m)}\\
=& \frac{2+m}{d_{\bm{\alpha,\beta}}(2+2m)}>0.
\end{align*}

Since $n\le p^2$, the former one implies that $t(n)=2$.
The latter one together with the following inequalities
$$\frac{1}{d_{\bm{\alpha,\beta}}}-\frac{m}{p}<\frac{1}{d_{\bm{\alpha,\beta}}}<1$$
implies that
$$\Big\lfloor\frac{1}{d_{\bm{\alpha,\beta}}}-\frac{m}{p}\Big\rfloor=0.$$
But the fact $p\not\equiv 0 \mod d_{\bm{\alpha,\beta}}$ infers that
$$\frac{p}{d_{\bm{\alpha,\beta}}}\ge\Big\lfloor\frac{p}{d_{\bm{\alpha,\beta}}}
\Big\rfloor+\frac{1}{d_{\bm{\alpha,\beta}}}.$$
It then follows that

\begin{align*}
\Big\lfloor\frac{n}{p}\Big\rfloor
\ge & \Big\lfloor\frac{p}{d_{\bm{\alpha,\beta}}}-\frac{m}{p}\Big\rfloor \\
\ge & \Big\lfloor \Big\lfloor\frac{p}{d_{\bm{\alpha,\beta}}}\Big\rfloor
+\frac{1}{d_{\bm{\alpha,\beta}}}-\frac{m}{p}\Big\rfloor\\
=&\Big\lfloor\frac{p}{d_{\bm{\alpha,\beta}}}\Big\rfloor
+\Big\lfloor\frac{1}{d_{\bm{\alpha,\beta}}}-\frac{m}{p}\Big\rfloor
=\Big\lfloor\frac{p}{d_{\bm{\alpha,\beta}}}\Big\rfloor.
\end{align*}
So we can deduce from (\ref{Eq2.22}) and $p>3sd_{\bm{\alpha,\beta}}$ that
\begin{align*}
v_p\Big(\frac{(\alpha_1)_n\cdots(\alpha_r)_n}{(\beta_1)_n\cdots(\beta_s)_n}\Big)
\ge &(r-s)\Big\lfloor\frac{n}{p}\Big\rfloor-2s \\
\ge &(r-s)\Big\lfloor \frac{p}{d_{\bm{\alpha,\beta}}}\Big\rfloor-2s\\
\ge &3s(r-s)-2s\ge s>0.
\end{align*}

If $n>p^2$, then $t(n)>2$. Since $t(n)=\lceil\log n/\log p\rceil$,
one has $t(n)<1+\log n/\log p$ and so $\lfloor\frac{n}{p}\rfloor\ge p^{t(n)-2}$.
Note that $p>3s$. It then follows from (\ref{Eq2.22}) and $t(n)>2$ that
\begin{align*}
v_p\Big(\frac{(\alpha_1)_n\cdots(\alpha_r)_n}{(\beta_1)_n\cdots (\beta_s)_n}\Big)
\ge &(r-s)p^{t(n)-2}-st(n)\\
> &(3s)^{t(n)-2}-st(n)\\
\ge & (3^{t(n)-2}-t(n))s\\
\ge & (1+2(t(n)-2)-t(n))s \\
= & (t(n)-3)s\ge 0.
\end{align*}
One concludes that (\ref{Eq2.32}) holds for all positive integers $n$.
Namely, $F_{\bm{\alpha,\beta}}(z)\in\mathbb{Z}_p[[z]]$. This proves part (ii).

(iii). Let $r=s$. First, we show the necessity part.
Let $F_{\bm{\alpha,\beta}}(z)\in\mathbb{Z}_p[[z]]$.
For any given $k\in\{1,...,s\}$ and $h\in\{1,...,{\rm ord}(a)\}$,
we let
$$n_k:=T_{p,h}(\beta_k)+1=p^h\mathfrak{D}_p^h(\beta_k)-\beta_k+1.$$
Then $1\le n_k\le p^h<p^l/d_{\bm{\alpha,\beta}}-m$
if $l>h$ since $p>d_{\bm{\alpha,\beta}}(2+2m)$.
So for all integers integers $l$ with $l>h$ and $\gamma\in\Psi $,
we deduce from (\ref{Eq2.20}) that
$$\Big\lceil \frac{n_k+\gamma}{p^l}-\mathfrak{D}_{p}^{l}(\gamma)\Big\rceil=0.$$
Thus for all integers $l$ with $l>h$, we have
\begin{equation} \label{Eq2.23}
\sum_{i=1}^{r}\Big\lceil\frac{n_k+\alpha_i}{p^l}-\mathfrak{D}_p^l(\alpha)\Big\rceil
-\sum_{j=1}^{s}\Big\lceil\frac{n_k+\beta_j}{p^l}-\mathfrak{D}_p^l(\beta_j)\Big\rceil=0.
\end{equation}

On the other hand, for any positive integer $l$ with $l\le h$,
by the definition of $T_{p,l}$, we derive that
$$n_k=T_{p,h}(\beta_k)+1\equiv T_{p,l}(\beta_k)+1
=p^l\mathfrak{D}_p^l(\beta_k)-\beta_k+1\mod p^l.$$
Then one has
$$\Big\lceil\frac{n_k+\gamma}{p^l}-\mathfrak{D}_p^l(\gamma)\Big\rceil
=\frac{n_k-p^l\mathfrak{D}_p^l(\beta_k)+\beta_k-1}{p^l}
+\Big\lceil\mathfrak{D}_p^l(\beta_k)-\mathfrak{D}_{p}^{l}(\gamma)
+\frac{\gamma-\beta_k+1}{p^l}\Big\rceil.$$
But (\ref{Eq2.0}) tells us that
$$
-1+\frac{1}{p^l}\le \mathfrak{D}_p^l(\beta_k)-\mathfrak{D}_{p}^{l}(\gamma)
+\frac{\gamma-\beta_k}{p^l}\le 1-\frac{1}{p^l}.
$$
Hence from (\ref{Eq2.12}) one deduces that
\begin{align*}
\Big\lceil\mathfrak{D}_p^l(\beta_k)-\mathfrak{D}_{p}^{l}(\gamma)
+\frac{\gamma-\beta_k+1}{p^l}\Big\rceil
=\bigg\{\begin{array}{cl}
       1, & {\rm if}\  a^l\gamma\preccurlyeq a^l\beta_k, \\
       0, & {\rm otherwise.}
     \end{array}
\end{align*}
Since $r=s$, it then follows immediately that
\begin{align} \label{Eq2.24}
&\sum_{i=1}^{r}\Big\lceil\frac{n_k+\alpha_i}{p^l}-\mathfrak{D}_p^l(\alpha_i)\Big\rceil
-\sum_{j=1}^{s}\Big\lceil\frac{n_k+\beta_j}{p^l}-\mathfrak{D}_p^l(\beta_j)\Big\rceil \\
=&\sum_{i=1}^{r}\Big(\Big\lceil\mathfrak{D}_p^l(\beta_k)-\mathfrak{D}_{p}^{l}(\alpha_i)+
+\frac{\alpha_i-\beta_k+1}{p^l}\Big\rceil+\frac{n_k-p^l\mathfrak{D}_p^l(\beta_k)+\beta_k-1}{p^l}\Big)\notag\\
&-\sum_{j=1}^{s}\Big(\Big\lceil\mathfrak{D}_p^l(\beta_k)-\mathfrak{D}_{p}^{l}(\beta_j)
+\frac{\beta_j-\beta_k+1}{p^l}\Big\rceil +\frac{n_k-p^l\mathfrak{D}_p^l(\beta_k)+\beta_k-1}{p^l}\Big) \notag\\
=& \sum_{i=1}^{r}\Big\lceil\mathfrak{D}_p^l(\beta_k)-\mathfrak{D}_{p}^{l}(\alpha_i)
+\frac{\alpha_i-\beta_k+1}{p^l}\Big\rceil- \sum_{j=1}^{s}\Big\lceil\mathfrak{D}_p^l(\beta_k)
-\mathfrak{D}_{p}^{l}(\beta_j)+\frac{\beta_j-\beta_k+1}{p^l}\Big\rceil \notag \\
=&\delta_{\bm{\alpha,\beta}}(a^l\beta_k;a^l). \notag
\end{align}
Then (\ref{Eq2.3''}) together with (\ref{Eq2.23}) and (\ref{Eq2.24}) gives us that
$$v_p\Big(\frac{(\alpha_1)_{n_k}\cdots(\alpha_r)_{n_k}}{(\beta_1)_{n_k}\cdots(\beta_s)_{n_k}}\Big)
=\sum_{l=1}^{h}\delta_{\bm{\alpha,\beta}}(a^l\beta_k;a^l).$$
Hence one can deduce from $F_{\bm{\alpha,\beta}}(z)\in\mathbb{Z}_p[[z]]$ that (\ref{Eq1.7})
holds for all $h\in\{1,...,{\rm ord}\ a\}$ and $k\in\{1,...,s\}$ as required.

In the following, we show that (\ref{Eq1.8}) holds for all
$l\in\{1,...,{\rm ord}(a)\}$ and $e\in\{1,...,d_{\bm{\alpha,\beta}}\}$.
Since $F_{\bm{\alpha,\beta}}\in\mathbb{Z}_p[[z]]$, one has
$$v_p\Big(\frac{(\alpha_1)_n\cdots(\alpha_r)_n}{(\beta_1)_n\cdots(\beta_s)_n}\Big)\ge 0$$
for each positive integer $n$. Then by (\ref{Eq2.21}) and Lemma \ref{lem2.6},
we can conclude that
$\delta_{\bm{\alpha,\beta}}(x, a)\ge 0$ for all $x\in\mathbb{R}$, thus
$$\delta_{\bm{\alpha,\beta}}\Big(a\Big(\frac{e}{d_{\bm{\alpha,\beta}}}
+m_{\bm{\alpha,\beta}}\Big), a\Big)\ge 0.$$
That is, (1.8) is true when $l=1$.

If $l\ge 2$, then by Lemma \ref{lem2.7}, it is equivalent to show that (\ref{Eq2.13})
is true for all $n\in\{1,...,p\}$. Let $h$ be any given positive integer. If $h>l$,
then $np^{l-1}<p^h/d_{\bm{\alpha,\beta}}-m$. It follows from (\ref{Eq2.20}) that
$$\Big\lceil \frac{np^{l-1}+\gamma}{p^h}-\mathfrak{D}_p^h(\gamma)\Big\rceil=0.$$
If $h<l$, then by (\ref{Eq2.0}), one derives that
$$\Big\lceil\frac{np^{l-1}+\gamma}{p^h}-\mathfrak{D}_p^h(\gamma)\Big\rceil=np^{l-h-1}.$$
Now we can conclude that if $h\neq l$, then noting that $r=s$, we have
$$\sum_{i=1}^{r}\Big\lceil\frac{np^{l-1}+\alpha_i}{p^h}-\mathfrak{D}_p^h(\alpha_i)\Big\rceil
-\sum_{j=1}^{s}\Big\lceil\frac{np^{l-1}+\beta_j}{p^h}-\mathfrak{D}_p^h(\beta_j)\Big\rceil=0.$$
So for all $n\in\{1,...,p\}$, it follows from (\ref{Eq2.3''}) that
\begin{align*}
&v_p\Big(\frac{(\alpha_1)_{np^{l-1}}\cdots(\alpha_r)_{np^{l-1}}}
{(\beta_1)_{np^{l-1}}\cdots(\beta_s)_{np^{l-1}}}\Big) \\
=&\sum_{h=1}^{\infty}\Big(\sum_{i=1}^{r}\Big\lceil\frac{np^{l-1}+\alpha_i}{p^h}
-\mathfrak{D}_p^h(\alpha_i)\Big\rceil-\sum_{j=1}^{s}\Big\lceil\frac{np^{l-1}
+\beta_j}{p^h}-\mathfrak{D}_p^h(\beta_j)\Big\rceil\Big) \\
=&\sum_{i=1}^{r}\Big\lceil\frac{n}{p}+\frac{\alpha_i}{p^l}-\mathfrak{D}_p^l(\alpha_i)\Big\rceil
-\sum_{j=1}^{s}\Big\lceil\frac{n}{p}+\frac{\beta_j}{p^l}-\mathfrak{D}_p^l(\beta_j)\Big\rceil\ge 0.
\end{align*}
Therefore by Lemma \ref{lem2.7}, we know that (\ref{Eq1.8}) holds for all
$l\in\{1,...,{\rm ord}(a)\}$ and $e\in\{1,...,d_{\bm{\alpha,\beta}}\}$ as one desires.
This finishes the proof of the necessity part of part (iii).

Now we turn our attention to the sufficiency part. Suppose that (\ref{Eq1.7}) holds
for all $h\in\{1,...,{\rm ord}\ a\}$ and all $k\in\{1,...,s\}$, and (\ref{Eq1.8}) holds
for all $l\in\{1,...,{\rm ord}(a)\}$ and $e\in\{1,...,d_{\bm{\alpha,\beta}}\}$.
In the remaining part of the proof, we show that $F_{\bm{\alpha,\beta}}(z)
\in\mathbb{Z}_p[[z]]$. In what follows, let $n$ be a fixed positive integer.
For any integer $l>0$, let $A_l (=A_l(n))$ be defined as in (\ref{Eq2.33}).
By (2.1), one has
$$-1<\Big\langle\frac{n}{p^l}\Big\rangle+\frac{\gamma}{p^l}-\mathfrak{D}_p^l(\gamma)\le 1$$
if $\gamma\in\Psi $.
Note that
$$\Big\langle\frac{n}{p^l}\Big\rangle+\frac{\gamma}{p^l}-\mathfrak{D}_p^l(\gamma)>0$$
if $\gamma\in A_l$. Thus for any
$\gamma\in\Psi $, we have
\begin{align*}
\Big\lceil\Big\langle\frac{n}{p^l}\Big\rangle+\frac{\gamma}{p^l}-\mathfrak{D}_p^l(\gamma)\Big\rceil
=\bigg\{\begin{array}{cl}
       1, & {\rm if}\  \gamma\in A_l, \\
       0, & {\rm otherwise.}
     \end{array}
\end{align*}
Since $\frac{n}{p^l}-\langle\frac{n}{p^l}\rangle$ is an integer
and $r=s$, it then follows that
\begin{align*}
&\sum_{i=1}^{r}\Big\lceil\frac{n+\alpha_i}{p^l}-\mathfrak{D}_p^l(\alpha_i)\Big\rceil
-\sum_{j=1}^{s}\Big\lceil\frac{n+\beta_j}{p^l}-\mathfrak{D}_p^l(\beta_j)\Big\rceil \\
=&\sum_{i=1}^{r}\Big\lceil\Big\langle\frac{n}{p^l}\Big\rangle+\frac{\alpha_i}{p^l}
-\mathfrak{D}_p^l(\alpha_i)\Big\rceil
-\sum_{j=1}^{s}\Big\lceil\Big\langle\frac{n}{p^l}\Big\rangle+\frac{\beta_j}{p^l}
-\mathfrak{D}_p^l(\beta_j)\Big\rceil. \\
=&\#\{i: \alpha_i\in A_l\}-\#\{j: \beta_j\in A_l\}.
\end{align*}
Therefore if $A_l$ is empty, then
\begin{equation} \label{Eq2.37}
\sum_{i=1}^{r}\Big\lceil\frac{n+\alpha_i}{p^l}-\mathfrak{D}_p^l(\alpha_i)\Big\rceil
-\sum_{j=1}^{s}\Big\lceil\frac{n+\beta_j}{p^l}-\mathfrak{D}_p^l(\beta_j)\Big\rceil=0.
\end{equation}

If $A_l$ is nonempty, let $\xi_l (=\xi_l(n))\in A_l$ satisfy
(\ref{Eq2.35}). Hence by (\ref{Eq2.12}), for any $\gamma\in\Psi $,
we have that
\begin{equation} \label{Eq2.45}
\gamma\in A_l \Longleftrightarrow a^l\gamma\preccurlyeq a^l\xi_l.
\end{equation}
It follows that if $A_l\ne \emptyset$, then
\begin{align} \label{Eq2.38}
&\sum_{i=1}^{r}\Big\lceil\frac{n+\alpha_i}{p^l}-\mathfrak{D}_p^l(\alpha_i)\Big\rceil
-\sum_{j=1}^{s}\Big\lceil\frac{n+\beta_j}{p^l}-\mathfrak{D}_p^l(\beta_j)\Big\rceil \\
=&\#\{i: \alpha_i\in A_l\}-\#\{j: \beta_j\in A_l\} \notag\\
=&\#\{i: a^l\alpha_i\preccurlyeq a^l\xi_l\}-\#\{j: a^l\beta_j\preccurlyeq a^l\xi_l\} \notag\\
=&\delta_{\bm{\alpha,\beta}}(a^l\xi_l, a^l). \notag
\end{align}
Then by (\ref{Eq2.3''}), (\ref{Eq2.37}) and (\ref{Eq2.38}),
one can conclude that
\begin{align*}
v_p\Big(\frac{(\alpha_1)_n\cdots(\alpha_r)_n}{(\beta_1)_n\cdots(\beta_s)_n}\Big)
=& \sum_{i=1}^rv_p((\alpha_i)_n)-\sum_{j=1}^sv_p((\beta _j)_n)\\
=&\sum_{l=1\atop A_l=\emptyset}^{\infty}
\Big(\sum_{i=1}^{r}\Big\lceil\frac{n+\alpha_i}{p^l}-\mathfrak{D}_p^l(\alpha_i)\Big\rceil
-\sum_{j=1}^{s}\Big\lceil\frac{n+\beta_j}{p^l}-\mathfrak{D}_p^l(\beta_j)\Big\rceil\Big)\\
 &-\sum_{l=1\atop A_l\ne \emptyset}^{\infty}
 \Big(\sum_{i=1}^{r}\Big\lceil\frac{n+\alpha_i}{p^l}-\mathfrak{D}_p^l(\alpha_i)\Big\rceil
-\sum_{j=1}^{s}\Big\lceil\frac{n+\beta_j}{p^l}-\mathfrak{D}_p^l(\beta_j)\Big\rceil\Big)\\
=&\sum_{l=1 \atop A_l\neq\emptyset}^{\infty }\delta_{\bm{\alpha,\beta}}(a^l\xi_l, a^l).
\end{align*}
Thus to show that $F_{\bm{\alpha,\beta}}(z)\in\mathbb{Z}_p[[z]]$,
it is sufficient to show that
\begin{equation} \label{Eq2.25}
\sum_{l=1, \atop A_l\neq\emptyset}^{\infty}\delta_{\bm{\alpha,\beta}}(a^l\xi_l, a^l)\ge 0,
\end{equation}
which will be done in the following.

For every positive integer $l$, let $B_l(=B_l(n))$ be defined as in (\ref{Eq2.34}).
If $B_l$ is nonempty, let $\eta_l(=\eta_l(n))\in B_l$ satisfy (\ref{Eq2.36}).
Then $\gamma\in B_l$ if and only if $a^l\eta_l\preccurlyeq a^l\gamma$. Set
$$L:=\max\{l>0: A_l\ne \emptyset, B_l\neq\emptyset\ {\rm and}\ \mathfrak{D}_p^l(\xi_l)
=\mathfrak{D}_p^l(\eta_l)\}$$
if the set $\{l>0: A_l\ne \emptyset, B_l\neq\emptyset \ {\rm and}\ \mathfrak{D}_p^l(\xi_l)=\mathfrak{D}_p^l(\eta_l)\}$
is nonempty, and $L:=0$ otherwise.
Since for all $\gamma\in\Psi $ and all integers $l$ with
$l>\big\lceil\frac{\log n}{\log p}\big\rceil$, by (\ref{Eq2.31}) one has
$$\Big\langle\frac{n}{p^l}\Big\rangle=\frac{n}{p^l}
<\mathfrak{D}_p^l(\gamma)-\frac{\gamma}{p^l}.$$
So $A_l$ is empty if $l>\big\lceil\frac{\log n}{\log p}\big\rceil$,
which implies that
$L\le \big\lceil\frac{\log n}{\log p}\big\rceil$.
Denote
$$\sum_{l=1, \atop A_l\neq\emptyset}^{\infty}\delta_{\bm{\alpha,\beta}}(a^l\xi_l, a^l):=\Sigma_1+\Sigma_2,$$
where
$$\Sigma_1:=\sum_{l=1 \atop A_l\neq\emptyset}^L\delta_{\bm{\alpha,\beta}}(a^l\xi_l, a^l) \ \ {\rm and}\ \
\Sigma_2:=\sum_{l=L+1 \atop A_l\neq\emptyset}^\infty\delta_{\bm{\alpha,\beta}}(a^l\xi_l, a^l).$$

First of all, we treat $\Sigma_1$. By Lemma \ref{lem2.8}, for all positive integers
$l\le L$, $A_l$ and $B_l$ are both nonempty and $\xi_l=\xi_L\ \ {\rm and}\ \ \eta_l=\eta_L.$
Hence
\begin{equation} \label{Eq2.39}
\Sigma_1=\sum_{l=1}^{L}\delta_{\bm{\alpha,\beta}}(a^l\xi_L, a^l).
\end{equation}
In the following we show that $\Sigma_1\ge 0$. For this purpose,
we consider the following two cases:

{\sc Case 1.} $\xi_L\in\{\beta_1,...,\beta_s\}$.
Write $L:=q {\rm ord}(a)+L_0$ for integers $q$ and $L_0$ with $0\le L_0<{\rm ord}(a)$.
For any $\beta\in\{\beta_1,...,\beta_s\}$, by Lemma \ref{lem2.9}, we have
\begin{equation} \label{Eq2.40}
\sum_{l=1}^{L}\delta_{\bm{\alpha,\beta}}(\beta, a^l)
=q\sum_{l=1}^{{\rm ord}(a)}\delta_{\bm{\alpha,\beta}}(a^l\beta, a^l)
+\sum_{l=1}^{L_0}\delta_{\bm{\alpha,\beta}}(a^l\beta, a^l).
\end{equation}
Since (\ref{Eq1.7}) holds for all
$h\in\{1,...,{\rm ord}(a)\}$ and all $k\in\{1,...,s\}$, one has
$$\sum_{l=1}^{{\rm ord}(a)}\delta_{\bm{\alpha,\beta}}(a^l\beta, a^l)\ge0
\ \ {\rm and}\ \ \sum_{l=1}^{L_0}\delta_{\bm{\alpha,\beta}}(a^l\beta, a^l)\ge 0.$$
Hence by (\ref{Eq2.40}), we have
\begin{equation} \label{Eq2.41}
\sum_{l=1}^{L}\delta_{\bm{\alpha,\beta}}(a^l\beta, a^l)\ge0.
\end{equation}
Then (\ref{Eq2.39}) together with (\ref{Eq2.41})
applied to $\xi_L$ tells us that $\Sigma_1\ge 0$ as required.

{\sc Case 2.} $\xi_L\in\{\alpha_1,...,\alpha_r\}$.
If the set $\{\beta_j: \langle\beta_j\rangle=\langle\xi_L\rangle\}$
is nonempty, then by Lemma \ref{lem2.1}0, one has
$$\delta_{\bm{\alpha,\beta}}\big(a^l\xi_L, a^l\big)
\ge\delta_{\bm{\alpha,\beta}}\big(a^l\xi_L', a^l\big),$$
where
$\xi_L'=\min\{\beta_j: \langle\beta_j\rangle=\langle\xi_L\rangle, \beta_j\ge \xi_L\}$.
So by (\ref{Eq2.41}) applied to $\xi_L'$ gives us that
$$\sum_{l=1}^{L}\delta_{\bm{\alpha,\beta}}(a^l\xi_L, a^l)
\ge\sum_{l=1}^{L}\delta_{\bm{\alpha,\beta}}(a^l\xi_L', a^l)\ge 0.$$
Hence by (\ref{Eq2.40}), $\Sigma_1\ge 0$ if
$\{\beta_j: \langle\beta_j\rangle=\langle\xi_L\rangle\}$ is nonempty.

If the set $\{\beta_j: \langle\beta_j\rangle=\langle\xi_L\rangle, \beta_j\ge \xi_L\}$
is empty, then by Lemma \ref{lem2.9}, one has
\begin{equation} \label{Eq2.42}
\delta_{\bm{\alpha,\beta}}(a^l\xi_L, a^l)\ge\#\big\{i: \mathfrak{D}_p^l(\alpha_i)
<\mathfrak{D}_p^l(\xi_L)\big\}-\#\big\{j: \mathfrak{D}_p^l(\beta_j)
<\mathfrak{D}_p^l(\xi_L)\big\}.
\end{equation}
Since by Lemma \ref{lem2.3}, one has for any $\gamma\in\Psi $ that
$$\mathfrak{D}_p^l(\gamma)\in\Big\{\frac{e}{d_{\bm{\alpha,\beta}}}:
1\le e\le d_{\bm{\alpha,\beta}}\Big\},$$
it then follows that
\begin{align} \label{Eq2.43}
&\#\big\{i: \mathfrak{D}_p^l(\alpha_i)<\mathfrak{D}_p^l(\xi_L)\big\}
-\#\big\{j: \mathfrak{D}_p^l(\beta_j)<\mathfrak{D}_p^l(\xi_L)\big\} \\
=&\#\Big\{i: \mathfrak{D}_p^l(\alpha_i)\le \mathfrak{D}_p^l(\xi_L)-\frac{1}{d_{\bm{\alpha,\beta}}}\Big\}
-\#\Big\{j: \mathfrak{D}_p^l(\beta_j)\le \mathfrak{D}_p^l(\xi_L)-\frac{1}{d_{\bm{\alpha,\beta}}}\Big\}. \notag
\end{align}
Since (1.8) holds for all $l\in\{1,...,{\rm ord}(a)\}$ and
$e\in\{1,...,d_{\bm{\alpha,\beta}}\}$, one can deduce from  Lemma \ref{lem2.7} that
\begin{equation} \label{Eq2.44}
\#\Big\{i: \mathfrak{D}_p^l(\alpha_i)\le \mathfrak{D}_p^l(\xi_L)-\frac{1}{d_{\bm{\alpha,\beta}}}\Big\}
-\#\Big\{j: \mathfrak{D}_p^l(\beta_j)\le \mathfrak{D}_p^l(\xi_L)-\frac{1}{d_{\bm{\alpha,\beta}}}\Big\}\ge 0
\end{equation}
if $\mathfrak{D}_p^l(\xi_L)-\frac{1}{d_{\bm{\alpha,\beta}}}>0$.
Also it is easy to check that (\ref{Eq2.44}) is true if $\mathfrak{D}_p^l(\xi_L)-\frac{1}{d_{\bm{\alpha,\beta}}}=0$.
So it follows from (\ref{Eq2.42}) to (\ref{Eq2.44}) that $\delta_{\bm{\alpha,\beta}}(a^l\xi_L, a^l)\ge 0$
for all integers $l$ with $1\le l\le L$. Hence by (\ref{Eq2.39}), $\Sigma_1\ge 0$ as desired.

Finally, we deal with $\Sigma_2$. Let $l$ be an integer with $l>L$ and $A_l\neq\emptyset$.
Recall that $A_l\cup B_l=\Psi $ and $A_l\cap B_l=\emptyset$.
If $B_l$ is empty, then $A_l=\Psi$. Thus by (\ref{Eq2.45}), we have
$$\delta_{\bm{\alpha,\beta}}(a^l\xi_l, a^l)=\#\{i: \alpha_i\in A_l\}
-\#\{j: \beta_j\in A_l\}=r-s=0.$$
If $B_l$ is nonempty, then
$$\mathfrak{D}_p^l(\xi_l)-\frac{\xi_l}{p^l}<\Big\langle\frac{n}{p^l}\Big\rangle
\le\mathfrak{D}_p^l(\eta_l)-\frac{\eta_l}{p^l}.$$
Then (\ref{Eq2.10}) applied to $\xi_l$ and $\eta_l$ gives us that
$\mathfrak{D}_p^l(\xi_l)\le\mathfrak{D}_p^l(\eta_l)$.
Since $l>L$, we have $\mathfrak{D}_p^l(\xi_l)\ne \mathfrak{D}_p^l(\eta_l)$
by the definition of $L$. Thus $\mathfrak{D}_p^l(\xi_l)<\mathfrak{D}_p^l(\eta_l)$.

We claim that for any $\gamma\in\Psi $,
one has
$$a^l\gamma\preccurlyeq a^l\xi_l \Longleftrightarrow
\mathfrak{D}_p^l(\gamma)\le\mathfrak{D}_p^l(\xi_l).$$
In fact, if $a^l\gamma\preccurlyeq a^l\xi_l$, then
$\langle a^l\gamma\rangle\le \langle a^l\xi_l\rangle$, which implies that
$\mathfrak{D}_p^l(\gamma)\le\mathfrak{D}_p^l(\xi_l)$ as claimed.
Conversely, if $\mathfrak{D}_p^l(\gamma)\le \mathfrak{D}_p^l(\xi_l)$, then
$\mathfrak{D}_p^l(\gamma)<\mathfrak{D}_p^l(\eta_l)$
since $\mathfrak{D}_p^l(\xi_l)<\mathfrak{D}_p^l(\eta_l)$.
Hence in the same way as in the proof of (\ref{Eq2.11}), one can show that
$$\mathfrak{D}_p^l(\eta_l)-\frac{\eta_l}{p^l}
>\mathfrak{D}_p^l(\gamma)-\frac{\gamma}{p^l}.$$
This infers that $\gamma\not\in B_l$. Since $A_l\cup B_l=\Psi$, one has $\gamma\in A_l$.
Then by (\ref{Eq2.45}), $a^l\gamma\preccurlyeq a^l\xi_l$.
The claim is proved.

Now by the claim we can deduce that
\begin{align*}
\delta_{\bm{\alpha,\beta}}(a^l\xi_l, a^l)
=&\#\{i: a^l\alpha_i\preccurlyeq a^l\xi_l\}
-\#\{j: a^l\beta_j\preccurlyeq a^l\xi_l\} \\
=&\#\{i: \mathfrak{D}_p^l(\alpha_i)\le\mathfrak{D}_p^l(\xi_l)\}
-\#\{j: \mathfrak{D}_p^l(\beta_j)\le \mathfrak{D}_p^l(\xi_l)\}.
\end{align*}
Since (\ref{Eq1.8}) holds for all $e\in\{1,...,d_{\bm{\alpha,\beta}}\}$,
it follows from Lemma \ref{lem2.7} that (\ref{Eq2.14}) is true for all
$e\in\{1,...,d_{\bm{\alpha,\beta}}\}$. Since
$\mathfrak{D}_p^l(\xi_l)=\frac{e}{d_{\bm{\alpha,\beta}}}$ for some integer
$e$ with $1\le e\le d_{\bm{\alpha,\beta}}$, we then deduce that
$$\#\{i: \mathfrak{D}_p^l(\alpha_i)\le\mathfrak{D}_p^l(\xi_l)\}
-\#\{j: \mathfrak{D}_p^l(\beta_j)\le \mathfrak{D}_p^l(\xi_l)\}\ge 0.$$
Thus $\delta_{\bm{\alpha,\beta}}(a^l\xi_l, a^l)\ge 0$
for all integers $l$ with $l>L$ and $A_l\neq\emptyset$. Therefore
$\Sigma_2\ge 0$ that completes the proof of (\ref{Eq2.25}).
So the sufficiency part of part (iii) is proved.

The proof of Theorem \ref{thm1.2} is complete. \hfill$\Box$

\section{N-integrality of the hypergeometric series with
algebraic parameters and proof of Theorem \ref{thm1.3}}

In the present section, we study the N-integrality for
the hypergeometric series with parameters from
algebraic number fields. The criterion for such
hypergeometric series to be N-integral is stated
in the first section, namely, Theorem \ref{thm1.3}.
Notice that such result extends Proposition 22 of
\cite{[DRR]}.

Let $K$ be an algebraic number field, and $\mathcal{O}_{K}$
be the ring of the algebraic integers in $K$. Let us recall
some basic concepts and facts. For a prime number $p$, let
$$p\mathcal{O}_{K}=\prod_{\mathfrak{p}\mid p}\mathfrak{p}^{e(\mathfrak{p}|p)}$$
be the unique factorization of $p$ in $\mathcal{O}_{K}$.
Let $\mathfrak{p}$ be a prime ideal dividing $p$
and let $x\in\mathcal{O}_{K}\setminus\{0\}$.
Then there is a nonnegative integer $n$ such that
$x\in\mathfrak{p}^n\setminus\mathfrak{p^{n+1}}$.
We define for every $x\in\mathcal{O}_{K}\setminus\{0\}$ that
$$v_{\mathfrak{p}}(x):=\frac{n}{e({\mathfrak{p}|p})} \ \ {\rm if}\
x\in\mathfrak{p}^n\setminus\mathfrak{p^{n+1}},$$
and let $v_{\mathfrak{p}}(0):=\infty$.
Then $v_{\mathfrak{p}}$ is the normalized $\mathfrak{p}$-adic valuation
of $\mathcal{O}_K$ and can be extended uniquely to a valuation of $K$.
Let $v_p$ denote the $p$-adic valuation on $\mathbb{Q}$. Then
$$v_{\mathfrak{p}}(x)=v_p(x)\ \ \forall\ x\in\mathbb{Q}.$$
Let $K_{\mathfrak{p}}$ and $\mathcal{O}_{K,\mathfrak{p}}$ denote the
$\mathfrak{p}$-adic completion of $K$ and $\mathcal{O}_K$, respectively.
Then $\mathcal{O}_{K,\mathfrak{p}}$ is the ring of integers in $K_{\mathfrak{p}}$.
Let $f(\mathfrak{p}|p):=[\mathcal{O}_K/\mathfrak{p}:\mathbb{Z}/p\mathbb{Z}]$.
One has
\begin{equation*}
[K_{\mathfrak{p}}:\mathbb{Q}_p]=f(\mathfrak{p}|p)e({\mathfrak{p}|p}).
\end{equation*}
Let $\mathbb{C}_p$ be the completion of the algebraic closure of $\mathbb{Q}_p$,
and $v_p$ be the corresponding extension of valuation from $\mathbb{Q}_p$
to $\mathbb{C}_p$. For a monomorphism $\sigma_p: K \mapsto\mathbb{C}_p$,
we define a map $v_{\sigma}$ from $K$ to the set of nonnegative real numbers
$\mathbb{R}_{\ge 0}$ as follows:
$$v_{\sigma}(x):=v_p(\sigma_p(x)).$$
So $v_{\sigma}$ defines a valuation on $K$ and there exists a prime
ideal $\mathfrak{p}$ such that $v_{\sigma}=v_{\mathfrak{p}}$ and
there is a unique monomorphism $\widehat{\sigma}_p: K_{\mathfrak{p}}\mapsto\mathbb{C}_p$
which satisfies $\widehat{\sigma}_p|_K=\sigma_p$.
That is, for every monomorphism $\sigma_p$ from $K$ to $\mathbb{C}_p$,
there is a prime ideal $\mathfrak{p}$ of $K$ and a monomorphism
$\widehat\sigma_p$ from $K_{\mathfrak{p}}$ to $\mathbb{C}_p$
such that the following diagram is commutative:
$$\xymatrix{
  K \ar[d]_{\iota} \ar[r]^{\sigma_p}
                & \mathbb{C}_p \ar[d]^{v_p}  \\
  K_{\mathfrak{p}} \ar@{.>}[ur]|-{\widehat\sigma_p} \ar[r]_{v_\mathfrak{p}}
                & \mathbb{R}_{\ge 0},             }$$
where $\iota$ is the canonical monomorphism from $K$ to $K_{\mathfrak{p}}$.
Let $\widehat{\sigma_p(K)}$ denote the completion of $\sigma_p(K)$,
then $\widehat{\sigma}_p$ gives an isomorphism:
$\widehat{\sigma_p(K)}\cong K_{\mathfrak{p}}$.

First of all, we prove a bounded result on the $p$-adic valuation of the product
of the polynomial $f$ with integer coefficients and no zero of nonnegative
integers evaluated at the first $n$ nonnegative integers. For this purpose,
we need the following important result that is due to Stewart
(Corollary 2 of \cite{[St]}). It is one of the important ingredients of this paper.

\begin{lem} \label{lem3.1} \cite{[St]}
Let $f(z)\in\mathbb{Z}[z]$ be primitive of degree $d$ $(\ge 2)$
and nonzero discriminant $D$. Let $p$ be a prime number.
Then for any positive integer $l$, the
number of solutions modulo $p^l$ of $f(z) \equiv 0\mod p^l$
is at most $2p^{\lfloor v_p(D)/2\rfloor}+d-2$.
\end{lem}

\begin{lem} \label{lem3.2}
Let $f(z)\in\mathbb{Z}[z]$ be a polynomial such that $f(k)\neq 0$
for all nonnegative integers $k$. Denote $d:=\deg f(z)$.
Then for all prime p and positive integers $n$,
there is a positive constant $C$ that depends on $f$ and $p$ such that
$$v_p\Big(\prod_{k=0}^{n-1}f(k)\Big)\le Cn.$$
\end{lem}

\begin{proof}
First of all, let $f(z)\in\mathbb{Z}[[z]]$ be primitive and
irreducible. Then the discriminant $D_f$ of $f(z)$ is nonzero.
The following identity is clear true (see, for example, (3.1) of
\cite{[HQ1]} or (4.1) of \cite{[HQ2]}):
For any $n$ nonzero integers $A_1, ..., A_n$, one has
\begin{equation} \label{eq3.1.1}
v_p\Big(\prod_{k=1}^nA_k\Big)=\sum_{l=1}^\infty \#\{1\le k\le n: A_k\equiv 0 \mod {p^l}\}.
\end{equation}
With $A_k$ being applied to $f(k)$ in (\ref{eq3.1.1}), we have
\begin{equation} \label{eq3.1.2}
v_p\Big(\prod_{k=0}^{n-1} f(k)\Big)
=\sum_{l=0}^{\infty}\#\{0\le k< n: f(k)\equiv 0\mod p^l\}.
\end{equation}

On the one hand, by Lemma \ref{lem3.1}, for any positive integers $l$,
the congruence equation $f(z)\equiv 0 \mod p^l$ has at most
$$S_p:=2p^{\lfloor v_p(D_f)/2\rfloor}+d-2$$
roots modulo $p^l$. Thus one has
\begin{align} \label{Eq3.3}
& \#\{0\le k< n: f(k)\equiv 0\mod p^l\} \\
\le &\#\{0\le k<p^l\Big\lceil\frac{n}{p^l}\Big\rceil: f(k)\equiv 0\mod p^l\} \notag \\
\le & \Big\lceil\frac{n}{p^l}\Big\rceil\cdot\#\{0\le k<p^l: f(k)\equiv 0\mod p^l\} \notag \\
\le & \Big\lceil\frac{n}{p^l}\Big\rceil S_p \notag \\
\le & \Big(\Big\lfloor\frac{n}{p^l}\Big\rfloor+1\Big)S_p. \notag
\end{align}

On the other hand, since $\deg f(x)=d$, one has
$$\lim_{n\rightarrow \infty}\frac{f(n)}{n^{d+1}}=0,$$
which implies that there is a positive integer $N$
such that $|f(n)|<n^{d+1}$ for all integers $n$ with $n>N$.
Set
$$Q_p:=\max_{0\le k\le N}\big\{v_p\big(f(k)\big)\big\}.$$
Obviously, one has
$$\#\{0\le k\le N: f(k)\equiv 0\mod p^l\}=0$$
if $l>Q_p$. For a given positive integer $n$, let
$$\mathcal{M}(n):=\max\Big\{Q_p, (d+1)\Big\lceil\frac{\log n}{\log p}\Big\rceil\Big\}.$$
Then $\mathcal{M}(n)\ge Q_p$ and $p^{\mathcal{M}(n)}\ge n^{d+1}$.

If $n\le N$ and $l>\mathcal{M}(n)$, then
$$\#\{0\le k< n: f(k)\equiv 0\mod p^l\}=0.$$

If $n>N$ and $l>\mathcal{M}(n)$, then for all integers $k$
with $N<k<n$, we have
$$0<|f(k)|<k^{d+1}<n^{d+1}\le p^{\mathcal{M}(n)}<p^l,$$
which implies that $f(k)\not\equiv 0\mod p^l$.
Thus
$$\#\{N< k< n: f(k)\equiv 0\mod p^l\}=0.$$
Therefore for all positive integers $n$ and $l$
with $l>\mathcal{M}(n)$, one has
\begin{equation} \label{Eq3.4}
\#\{0\le k< n: f(k)\equiv 0\mod p^l\}=0.
\end{equation}

By using (\ref{eq3.1.2}) to (\ref{Eq3.4}),
we derive immediately that
\begin{align} \label{Eq3.5}
v_p\Big(\prod_{k=0}^{n-1} f(k)\Big)=&\sum_{l=1}^{\mathcal{M}(n)}\#\{0\le k< n: f(k)\equiv 0\mod p^l\}\\
&+\sum_{l=\mathcal{M}(n)+1}^{\infty }\#\{0\le k< n: f(k)\equiv 0\mod p^l\} \notag \\
=&\sum_{l=1}^{\mathcal{M}(n)}\#\{0\le k< n: f(k)\equiv 0\mod p^l\} \notag \\
\le &\sum_{l=1}^{\mathcal{M}(n)}\Big(\Big\lfloor\frac{n}{p^l}\Big\rfloor+1\Big)S_p \notag \\
\le & S_p\sum_{l=1}^{\infty}\frac{n}{p^l}+\mathcal{M}(n)S_p \notag \\
= & \frac{nS_p}{p-1}+\mathcal{M}(n)S_p \notag \\
=& n\Big(\frac{1}{p-1}+\frac{\mathcal{M}(n)}{n}\Big)S_p. \notag
\end{align}
But
$$\frac{\mathcal{M}(n)}{n}=\max\Big(\frac{Q_p}{n},
\frac{d+1}{n}\Big\lceil\frac{\log n}{\log p}\Big\rceil\Big)\le\max\Big(Q_p,
\frac{d+1}{n}\Big\lceil\frac{\log n}{\log p}\Big\rceil\Big).$$
Let $n\in (p^t, p^{t+1}]$ with $t\ge 0$ being an integer. Then
$$\frac{1}{n}\Big\lceil\frac{\log n}{\log p}\Big\rceil<\frac{t+1}{p^t}\le \frac{t+1}{2^t}\le 1.$$
It infers that
$$\frac{\mathcal{M}(n)}{n}\le \max(Q_p, d+1).$$
This together with (\ref{Eq3.5}) gives us that
$$
v_p\Big(\prod_{k=0}^{n-1} f(k)\Big)\le \Big(\frac{1}{p-1}+\max(Q_p, d+1)\Big)S_pn:=Cn
$$
if $f(z)$ is primitive and irreducible. Hence Lemma 3.2 is true
when $f(z)$ is primitive and irreducible.

In what follows, we deal with the general case.
Let $f(z)$ be any polynomial in $\mathbb{Z}[z]$ such that $f(k)\neq 0$ for all
nonnegative integers $k$. Then there are nonnegative integer $c$ and primitive
irreducible polynomials $f_1(z),...,f_m(z)\in\mathbb{Z}[z]$ such that
$$f(z)=c\prod_{i=1}^{m} f_i(z).$$
The arguments above tell us that there are positive constant $C_1,...,C_m$ such that
$$v_p\Big(\prod_{k=0}^{n-1} f_i(k)\Big)\le C_in$$
for all integers $i$ with $1\le i\le m$. It then follows that
\begin{align*}
v_p\Big(\prod_{k=0}^{n-1} f(k)\Big)= & v_p\Big(\prod_{k=0}^{n-1}c\prod_{i=1}^m f_i(k)\Big)\\
= & v_p\Big(c^n\prod_{i=1}^{m}\prod_{k=0}^{n-1} f_i(k)\Big)\\
= & n v_p(c)+\sum_{i=1}^m v_p\Big(\prod_{k=0}^{n-1} f_i(k)\Big)\\
\le & \Big(v_p(c)+\sum_{i=1}^m C_i\Big)n:=Cn
\end{align*}
as desired. This proves Lemma \ref{lem3.2} for the general case.
So Lemma \ref{lem3.2} is proved.
\end{proof}

\begin{lem} \label{lem3.3}
Let $\gamma\in K\setminus\mathbb{Z}_{\le0}$ and $n$ be a positive integer.
Let $\mathfrak{p}$ be a prime ideals of $K$. Then each of the following is true:

{\rm (i)}. If $v_{\mathfrak{p}}(\gamma)<0$, then
$$v_{\mathfrak{p}}\Big(\prod_{k=0}^{n-1}(k+\gamma)\Big)=v_{\mathfrak{p}}(\gamma)n.$$

{\rm (ii)}. If $v_{\mathfrak{p}}(\gamma)\ge 0$, then there exists
a positive constant $C$ such that
$$v_{\mathfrak{p}}\Big(\prod_{k=0}^{n-1}(k+\gamma)\Big)\le Cn.$$
\end{lem}

\begin{proof}
(i). Let $v_{\mathfrak{p}}(\gamma)<0$. Then for any nonnegative integer $k$, one has
$v_{\mathfrak{p}}(k+\gamma)=v_{\mathfrak{p}}(\gamma)$. It follows that
$$v_{\mathfrak{p}}\Big(\prod_{k=0}^{n-1}(k+\gamma)\Big)=v_{\mathfrak{p}}(\gamma)n.$$

(ii). Let $v_{\mathfrak{p}}(\gamma)\ge 0$.
Let $L$ be the Galois closure of $K$ over $\mathbb{Q}$ and
$\mathcal{O}_L$ be the ring of integers of $L$.
Denote $G:={\rm Gal}(L/\mathbb{Q}).$
Let $\mathfrak{q}$ be a prime ideal of $\mathcal{O}_L$
with $\mathfrak{q}\mid \mathfrak{p}$.
Since $v_{\mathfrak{p}}$ and $v_{\mathfrak{p}}$ are normalized, one has
$v_{\mathfrak{q}}|_K=v_{\mathfrak{p}}$, and so
$$v_{\mathfrak{p}}\Big(\prod_{k=0}^{n-1}(k+\gamma)\Big)
=v_{\mathfrak{q}}\Big(\prod_{k=0}^{n-1}(k+\gamma)\Big).$$
We treat the latter $v_{\mathfrak{q}}\big(\prod_{k=0}^{n-1}(k+\gamma)\big)$
in the following.

Clearly, we have
\begin{equation} \label{Eq3.6}
v_{\mathfrak{q}}\Big(\prod_{k=0}^{n-1}(k+\gamma)\Big)\le \sum_{\tau\in G}
v_{\mathfrak{q}}\Big(\prod_{k=0}^{n-1}(k+\tau(\gamma))\Big)
-\sum_{\tau\in G \atop v_{\mathfrak{q}}(\tau(\gamma))<0}
v_{\mathfrak{q}}\Big(\prod_{k=0}^{n-1}(k+\tau(\gamma))\Big).
\end{equation}
Write
$$f(z):=\prod_{\tau\in G}(z+\tau(\gamma)).$$
Then $f(z)\in\mathbb{Q}[z]$. There is exactly
one positive integer $c$ such that
$cf(z)\in\mathbb{Z}[z]$ is primitive. Let $p$ be a
prime number such that $\mathfrak{q}\mid p$.
By Lemma \ref{lem3.2} applied to $cf(z)$, one knows
that there is a positive constant $C_1$ such that
\begin{align} \label{Eq3.7}
\sum_{\tau\in G}v_{\mathfrak{q}}\Big(\prod_{k=0}^{n-1}(k+\tau(\gamma))\Big)
=&v_{\mathfrak{q}}\Big(\prod_{k=0}^{n-1}cf(k)\Big)-v_{\mathfrak{q}}(c)n \\
=&v_{p}\Big(\prod_{k=0}^{n-1}cf(k)\Big)-v_{\mathfrak{q}}(c)n \notag\\
\le& C_1n-v_{\mathfrak{q}}(c)n. \notag
\end{align}
Moreover, we deduce from part (i) applied to $\tau(\gamma)$ with
$v_{\mathfrak{q}}(\tau(\gamma))<0$ that
\begin{equation} \label{Eq3.8}
\sum_{\tau\in G \atop v_{\mathfrak{q}}(\tau(\gamma))<0}
v_{\mathfrak{q}}\Big(\prod_{k=0}^{n-1}(k+\tau(\gamma)\Big)
=n\sum_{\tau\in G \atop v_{\mathfrak{q}}(\tau(\gamma))<0}
v_{\mathfrak{q}}(\tau(\gamma)).
\end{equation}
Then by (\ref{Eq3.6}) to (\ref{Eq3.8}), one can derive that
$$v_{\mathfrak{q}}\Big(\prod_{k=0}^{n-1}(k+\gamma)\Big)
\le C_1n-v_{\mathfrak{q}}(c)n-n\sum_{\tau\in G \atop v_{\mathfrak{q}}(\tau(\gamma))<0}
v_{\mathfrak{q}}(\tau(\gamma)):=Cn$$
as required. This ends the proof of Lemma \ref{lem3.3}.
\end{proof}

Let $\bm{\alpha}=(\alpha_1,...,\alpha_r)$ and $\bm{\beta}=(\beta_1,...,\beta_s)$ with
$\alpha_i,\beta_j\in K\setminus\mathbb{Z}_{\le 0}$.
Now we can give the proof of Theorem \ref{thm1.3}.\\

{\it Proof of Theorem \ref{thm1.3}.}
(i)$\Rightarrow$(ii). Let $F_{\bm{\alpha,\beta}}(z)$ be N-integral.
Then there is an integer $c\in\mathcal{O}_K$ such that
$$F_{\bm{\alpha,\beta}}(cz)\in\mathcal{O}_K[[z]].$$
Then for any positive integer $n$, one has
$$c^n\frac{(\alpha_1)_n\cdots(\alpha_r)_n}{(\beta_1)_n\cdots(\beta_s)_n}\in\mathcal{O}_K.$$
For a prime ideal $\mathfrak{p}$ of $\mathcal{O}_K$ with $c\not\in\mathfrak{p}$,
we have $v_{\mathfrak{p}}(c)=0$. So one can deduce that
\begin{align*}v_{\mathfrak{p}}\Big(\frac{(\alpha_1)_n\cdots(\alpha_r)_n}{(\beta_1)_n\cdots(\beta_s)_n}\Big)
=& v_{\mathfrak{p}}\Big(c^n\frac{(\alpha_1)_n\cdots(\alpha_r)_n}{(\beta_1)_n\cdots(\beta_s)_n}\Big)-n v_{\mathfrak{p}}(c)\\
=& v_{\mathfrak{p}}\Big(c^n\frac{(\alpha_1)_n\cdots(\alpha_r)_n}{(\beta_1)_n\cdots(\beta_s)_n}\Big)\ge 0.
\end{align*}
Thus for all prime ideals $\mathfrak{p}$ with $c\not\in\mathfrak{p}$, one has
$$F_{\bm{\alpha,\beta}}(cz)\in\mathcal{O}_{K,\mathfrak{p}}[[z]].$$
That is, part (ii) is true.

(ii)$\Rightarrow$(iii).
Let $F_{\bm{\alpha,\beta}}(z)\in\mathcal{O}_{K,\mathfrak{p}}[[z]]$ hold
for almost all prime ideals $\mathfrak{p}$ of $\mathcal{O}_K$.
Let $p$ be a prime number. For any arbitrary given
monomorphism $\sigma_p: K \rightarrow \mathbb{C}_p$,
as we mentioned in the beginning of this section,
there exists exactly a prime ideal $\mathfrak{p}|p$
and a $\widehat{\sigma}_p: K_{\mathfrak{p}}\mapsto \mathbb{C}_p$
such that $\widehat{\sigma}_p|_K=\sigma_p$
and $v_{\mathfrak{p}}(x)=v_p(\widehat{\sigma}_p(x))$
for any $x\in K_{\mathfrak{p}}$. We claim that
$\sigma_p(F_{\bm{\alpha,\beta}}(z))\in\mathcal{O}_p[[z]]$
if and only if
$F_{\bm{\alpha,\beta}}(z)\in\mathcal{O}_{\mathfrak{p},K}[[z]]$
which will be proved in what follows.
Since $\widehat{\sigma}|_K=\sigma_p$ and
$F_{\bm{\alpha,\beta}}(z)\in K[[z]]$, one can derive that
\begin{align} \label{Eq3.9}
\sigma_p(F_{\bm{\alpha,\beta}}(z))=\widehat{\sigma}_p(F_{\bm{\alpha,\beta}}(z)).
\end{align}
But $v_{\mathfrak{p}}(x)=v_p(\widehat{\sigma}_p(x))$
for any $x\in K_{\mathfrak{p}}$, it follows that
\begin{align} \label{Eq3.10}
v_{\mathfrak{p}}\Big(\frac{(\alpha_1)_n\cdots(\alpha_r)_n}{(\beta_1)_n\cdots(\beta_s)_n}\Big)
=v_p\Big(\widehat{\sigma}_p\Big(\frac{(\alpha_1)_n\cdots(\alpha_r)_n}{(\beta_1)_n\cdots(\beta_s)_n}\Big)\Big).
\end{align}
Thus by (\ref{Eq3.10}), one knows that ${\sigma_p}(F_{\bm{\alpha,\beta}}(z))\in\mathcal{O}_p[[z]]$
if and only if
$\widehat{\sigma}_p(F_{\bm{\alpha,\beta}}(z))\in\mathcal{O}_p[[z]]$,
which holds if and only if for all nonnegative integers $n$, one has
\begin{align} \label{Eq3.11}
v_p\Big(\widehat{\sigma}_p\Big(\frac{(\alpha_1)_n\cdots(\alpha_r)_n}{(\beta_1)_n\cdots(\beta_s)_n}\Big)\Big)\ge 0.
\end{align}
By (\ref{Eq3.9}), (\ref{Eq3.11}) is true if and only if
for all nonnegative integers $n$, we have
$$
v_{\mathfrak{p}}\Big(\frac{(\alpha_1)_n\cdots(\alpha_r)_n}{(\beta_1)_n\cdots(\beta_s)_n}\Big)\ge 0,
$$
namely, if and only if $F_{\bm{\alpha,\beta}}(z)\in\mathcal{O}_{\mathfrak{p},K}[[z]]$.
Hence the claim is proved.

Let now $P_K$ denote the set of all the prime ideals of $\mathcal{O}_K$ and let $R_K$
be the set of all the prime ideals $\mathfrak{q}$ of $\mathcal{O}_K$ such that
$F_{\bm{\alpha,\beta}}(z)\not\in\mathcal{O}_{\mathfrak{q},K}[[z]]$. Then by
the hypothesis, $R_K$ is finite. On the other hand, for any $\mathfrak{q}\in R_K$,
there is exactly a prime number $p$ such that $\mathfrak{q}|p$. Let $T_K$ be
the set of all the prime numbers $p$ such that $\mathfrak{q}|p$ for some
$\mathfrak{q}\in R_K$. Then $T_K$ is finite since $R_K$ is finite.
Now pick an arbitrary prime number $p$ such that $p>\max(T_K)$. Let
$\mathfrak{p}$ be any prime ideal dividing $p$. Then $\mathfrak{p}\not\in R_K$
that implies that $F_{\bm{\alpha,\beta}}(z)\in\mathcal{O}_{\mathfrak{p},K}[[z]]$.
But the claim tells us that ${\sigma_p}(F_{\bm{\alpha,\beta}}(z))\in\mathcal{O}_p[[z]]$
if and only if $F_{\bm{\alpha,\beta}}(z)\in\mathcal{O}_{\mathfrak{p},K}[[z]]$.
Thus ${\sigma_p}(F_{\bm{\alpha,\beta}}(z))\in\mathcal{O}_p[[z]]$.
Therefore part (iii) is true.

(iii)$\Rightarrow$(i). Let $\sigma_p(F_{\bm{\alpha,\beta}}(z))\in\mathcal{O}_p[[z]]$ hold
for almost all prime numbers $p$ and all $\sigma_p: K\rightarrow\mathbb{C}_p$.
Then there is an positive integer $P$ such that for all prime number $p>P$ and for all
$\sigma_p: K\rightarrow\mathbb{C}_p$, we have
$\sigma_p(F_{\bm{\alpha,\beta}}(z))\in\mathcal{O}_p[[z]]$,
equivalently, all the $p$-adic valuations of all the coefficients
of $\sigma_p(F_{\bm{\alpha,\beta}}(z))$ are nonnegative.

Let $\mathfrak{p}$ be an arbitrary prime ideal in $\mathcal{O}_K$ that divides a prime number $p$ with $p>P$.
Then there is an embedding $\widehat{\sigma}_p$ from $K_{\mathfrak{p}}$ to $\mathbb{C}_p$ such that
$v_{\mathfrak{p}}(x)=v_p(\widehat{\sigma}_p(x))$ for any $x\in K_{\mathfrak{p}}$.
Let $\sigma_p':=\widehat{\sigma}_p|_K$. Then $\sigma_p'(F_{\bm{\alpha,\beta}}(z))\in\mathcal{O}_p[[z]]$.
Since $F_{\bm{\alpha,\beta}}(z)\in K[[z]]$, one has
$$\widehat{\sigma}_p(F_{\bm{\alpha,\beta}}(z))=\sigma_p'(F_{\bm{\alpha,\beta}}(z)).$$
Hence $\sigma_p'(F_{\bm{\alpha,\beta}}(z))\in\mathcal{O}_p[[z]]$ together with
$v_{\mathfrak{p}}(x)=v_p(\widehat{\sigma}_p(x))$ for any $x\in K_{\mathfrak{p}}$
tells us that for all nonnegative integers $n$, one has
\begin{align} \label{Eq3.12}
v_{\mathfrak{p}}\Big(\frac{(\alpha_1)_n\cdots(\alpha_r)_n}{(\beta_1)_n\cdots(\beta_s)_n}\Big)
= & v_p\Big(\widehat{\sigma}_p\Big(\frac{(\alpha_1)_n\cdots(\alpha_r)_n}{(\beta_1)_n
\cdots(\beta_s)_n}\Big)\Big) \\
= & v_p\Big(\sigma_p'\Big(\frac{(\alpha_1)_n\cdots(\alpha_r)_n}{(\beta_1)_n
\cdots(\beta_s)_n}\Big)\Big)\ge 0. \notag
\end{align}

Now let $\mathfrak{p}$ be an arbitrary prime ideal in $\mathcal{O}_K$ that divides a prime number
$p$ with $p\le P$. It is clear that for all nonnegative integers $n$, one has
$$v_{\mathfrak{p}}\Big(\frac{(\alpha_1)_n\cdots(\alpha_r)_n}{(\beta_1)_n\cdots(\beta_s)_n}\Big)
\ge \sum_{i=1 \atop v_{\mathfrak{p}}(\alpha_i)<0}^{r} v_{\mathfrak{p}}((\alpha_i)_n)
-\sum_{j=1 \atop v_{\mathfrak{p}}(\beta_j)\ge 0}^s v_{\mathfrak{p}}((\beta_j)_n).$$
By Lemma \ref{lem3.3}, we have that
$v_{\mathfrak{p}}((\alpha_i)_n)=nv_{\mathfrak{p}}(\alpha_i)$
if $v_{\mathfrak{p}}(\alpha_i)<0$,
and $v_{\mathfrak{p}}((\beta_j)_n)\le C_{j}(\mathfrak{p})n$
for some positive constant
$C_j(\mathfrak{p})$ if $v_{\mathfrak{p}}(\beta_j)\ge 0$.
Thus for all nonnegative integers $n$, one has
\begin{equation} \label{Eq3.13}
v_{\mathfrak{p}}\Big(\frac{(\alpha_1)_n\cdots(\alpha_r)_n}
{(\beta_1)_n\cdots(\beta_s)_n}\Big)\ge -C({\mathfrak{p}})n,
\end{equation}
where
$$C({\mathfrak{p}}):=\sum_{j=1 \atop v_{\mathfrak{p}}(\beta_j)\ge 0}^s C_{j,\mathfrak{p}}
-\sum_{i=1 \atop v_{\mathfrak{p}}(\alpha_i)<0}^{r}v_{\mathfrak{p}}(\alpha_i)>0.$$

For a prime number $p<P$, let
$C_p:=\max\{\lceil C({\mathfrak{p}})\rceil:\mathfrak{p}\mid p\}$.
Let $c:=\prod_{p<P}p^{C(p)}$. Then $c\in \mathcal{O}_K$ and
$v_{\mathfrak{p}}(c)\ge\lceil C({\mathfrak{p}})\rceil$ for all
prime ideals $\mathfrak{p}$ dividing a prime number $p<P$.
It then follows from (\ref{Eq3.13})
that for all prime ideals $\mathfrak{p}$ dividing
a prime number $p<P$, one has for all nonnegative
integers $n$ that
\begin{align*}
v_{\mathfrak{p}}\Big(c^n\frac{(\alpha_1)_n\cdots(\alpha_r)_n}
{(\beta_1)_n\cdots(\beta_s)_n}\Big)\ge & v_{\mathfrak{p}}(c)n-C({\mathfrak{p}})n \\
\ge & (\lceil C({\mathfrak{p}})\rceil-C({\mathfrak{p}}))n \ge 0.
\end{align*}

Now for any prime ideals $\mathfrak{p}$ dividing a prime number
$p\ge P$, one has $v_{\mathfrak{p}}(c)\ge 0$ since
$c\in \mathcal{O}_K$. It then follows from (\ref{Eq3.12}) that
for all nonnegative integers $n$, we have
$$v_{\mathfrak{p}}\Big(c^n\frac{(\alpha_1)_n
\cdots(\alpha_r)_n}{(\beta_1)_n\cdots(\beta_s)_n}\Big)
=n v_{\mathfrak{p}}(c)+v_{\mathfrak{p}}
\Big(\frac{(\alpha_1)_n\cdots(\alpha_r)_n}
{(\beta_1)_n\cdots(\beta_s)_n}\Big)\ge 0.$$
Thus we can conclude that
$$\frac{(\alpha_1)_n\cdots(\alpha_r)_n}{(\beta_1)_n\cdots(\beta_s)_n}\in\mathcal{O}_K$$
holds for all nonnegative integers $n$. That is,
$F_{\bm{\alpha,\beta}}(cz)\in\mathcal{O}_K[[z]]$.
Therefore $F_{\bm{\alpha,\beta}}(z)$ is N-integral in $K$.
The proof of Theorem \ref{thm1.3} is finished. \hfill$\Box$ \\

Let $\theta $ be any given algebraic integer and $L$ be a field generated by
$\theta $, i.e., $L=\mathbb{Q}(\theta)$. Suppose that $f(z)\in\mathbb{Z}[z]$
is the minimal polynomial of $\theta$. Let $\mathcal{O}_L$ be the ring of
algebraic integers of $L$. Then $\mathbb{Z}[\theta]$ is a subring of $\mathcal{O}_L$
and the quotient group $\mathcal{O}_L/\mathbb{Z}[\theta]$ has finite order,
which is denoted by $\big[\mathcal{O}_L:\mathbb{Z}[\theta]\big]$.
The following result is known.

\begin{lem} \cite{[Nu]} \label{lem3.4}
Let $\theta $ be any given algebraic integer and $L=\mathbb{Q}(\theta)$.
Let $f(z)\in\mathbb{Z}[z]$ be the monic minimal polynomial of $\theta$
and $\mathcal{O}_L$ be the ring of algebraic integers of $L$.
Let $p$ be a prime number such that
$p\nmid\big[\mathcal{O}_L:\mathbb{Z}[\theta]\big]$, and let
$$\bar{f}(z)=\bar{f}_1(z)^{e_1}\cdots\bar{f}_g(z)^{e_g}$$
be the factorization of the polynomial $\bar{f}(z)=f(z) \mod p$ into irreducible polynomials
$\bar{f}_i(z)=f_i(z) \mod p$ over $\mathbb{Z}/p\mathbb{Z}$, with all $f_i(z)\in\mathbb{Z}[z]$ being monic. Then
$$\mathfrak{p}_i:=p\mathcal{O}_L+f_i(\theta)\mathcal{O}_L, i=1,...,g$$
are the all different prime ideals of $\mathcal{O}_L$ above $p$. The inertia degree $f_i$ of
$\mathfrak{p}_i$ is the degree of $\bar{f}_i(z)$, and one has
$$p\mathcal{O}_L=\mathfrak{p}_1^{e_1}\cdots\mathfrak{p}_g^{e_g}.$$
\end{lem}

Let $u$ and $v$ be two integers with $0\le u\le r$ and $0\le v\le s$.
Assume that $\alpha_1,...,\alpha_u,\beta_1,...,\beta_v\in\mathbb{Q}\setminus\mathbb{Z}_{\le 0}$
and $\alpha_{u+1},...,\alpha_r,\beta_{v+1},...,\beta_s\in K\setminus\mathbb{Q}$.
Let $\bm{\mu}:=(\alpha_1,...,\alpha_u)$ and $\bm{\nu}:=(\beta_1,...,\beta_v)$, and
$\delta_{\bm{\bm{\mu,\nu}}}$ and $M_{\bm{\bm{\mu,\nu}}}$
are given by (\ref{Eq1.4}) and (\ref{Eq1.5}) respectively.
Recall that $d_{\bm{\bm{\mu,\nu}}}$ denote the least common multiple of the exact denominators
of the components of $\bm{\mu}$ and $\bm{\nu}$. Recall a well-known fact that states that
for any algebraic number  $\alpha \in K$, there exists a rational integer $C$
such that $C\alpha $ is an algebraic integer of $K$.

\begin{prop} \label{prop3.5}
Let $C\in\mathbb{Z}$ such that $C\alpha_i$ and $C\beta_j$ are
algebraic integers in $K$ for all indexes $i$ and $j$ with $u+1\le i\le r$ and $v+1\le j\le s$.
For any $\theta\in K$, define $L_{\theta}:=\mathbb{Q}(\theta)$ and let
$\mathcal{O}_{L_{\theta}}$ be the ring of integers of $L_{\theta}$.
Let $p$ be a prime number such that $p>\max(M_{\bm{\bm{\mu,\nu}}}, d_{\bm{\bm{\mu,\nu}}})$
and $\mathfrak{p}:=p\mathcal{O}_K$ is a prime ideal in $\mathcal{O}_K$.
Suppose that $p$ is coprime to $C\big[\mathcal{O}_{L_{\theta}}:\mathbb{Z}[C\theta]\big]$
for any $\theta\in\{\alpha_{u+1},...,\alpha_r,\beta_{v+1},...,\beta_s\}$.
Then each of the following is true.

{\rm (i)}. If $u<v$, then $F_{\bm{\alpha,\beta}}(z)\not\in\mathcal{O}_{K,\mathfrak{p}}[[z]]$.

{\rm (ii).} If $u>v$, then $F_{\bm{\alpha,\beta}}(z)\in\mathcal{O}_{K,\mathfrak{p}}[[z]]$
if and only if $\delta_{\bm{\bm{\mu,\nu}}}(x,a)\ge 0$ for all $x\in\mathbb{R}$.

{\rm (iii).} If $u=v$, then $F_{\bm{\alpha,\beta}}(z)\in\mathcal{O}_{K,\mathfrak{p}}[[z]]$
if and only if $\sum_{l=1}^{h}\delta_{\bm{\bm{\mu,\nu}}}(a^l\beta_k,a^l)\ge 0$ for all
$h\in\{1,...,ord\ a\} $ and $k\in\{1,...,v\}$, and
$\delta_{\bm{\bm{\mu,\nu}}}\Big(a^l\Big(\frac{e}{d_{\bm{\bm{\mu,\nu}}}}+m_{\bm{\bm{\mu,\nu}}}\Big),
a^l\Big)\ge 0$ for all $l\in\{1,...,{\rm ord}(a)\}$ and $e\in\{1,...,d_{\bm{\nu,\mu}}\}$,
where ${\rm ord}(a)$ stands for the order of $a$ modulo $d_{\bm{\bm{\mu,\nu}}}$.
\end{prop}

\begin{proof}
Let $\theta\in\{-\alpha_{u+1},...,-\alpha_r,-\beta_{v+1},...,-\beta_s\}$ and
$f(z)\in\mathbb{Q}[z]$ be the monic minimal polynomial of $\theta$.
Then $C^df(z/C)\in\mathbb{Q}[z]$ is the monic  minimal polynomial of $C\theta$,
where $d:=\deg f(z)$. But $C\theta$ is an algebraic integer. One then deduces
that $C^df(z/C)\in\mathbb{Z}[z]$. Since $p\mathcal{O}_K$ is a prime ideal in
$\mathcal{O}_K$, $p\mathcal{O}_{L_{\theta}}$ is also a prime ideal in $\mathcal{O}_{L_{\theta}}$.
Since $p$ is coprime to $\big[\mathcal{O}_{L_{\theta}}:\mathbb{Z}[C\theta]\big]$,
by Proposition \ref{lem3.4} we derive that $C^df(z/C) \mod p$ is irreducible.
It then follows that $f(z) \mod p$ is irreducible.
Hence for any nonnegative integer $k$, one has $f(k)\not\equiv 0 \mod p.$
Thus for any nonnegative integer $k$ and $\theta\in\{-\alpha_{u+1},...,-\alpha_r,
-\beta_{v+1},...,-\beta_s\}$, it follows from $(z-\theta) \mid f(z)$
in $\mathcal{O}_K[z]$ that $(k-\theta)\mid f(k)$ holds in $\mathcal{O}_K$.
Since $f(k)\not\equiv 0 \mod p$, we can deduce that $k-\theta\not\in p\mathcal{O}_K=\mathfrak{p}.$
It implies for any nonnegative integer $k$ that
$$v_{\mathfrak{p}}(k-\theta)=0.$$
Then for all nonnegative integers $n$ and all
$-\theta\in\{\alpha_{u+1},...,\alpha_r,\beta_{v+1},...,\beta_s\}$,
we have $v_{\mathfrak{p}}\big((-\theta)_n\big)=0.$
In other words, for any $\gamma \in\{\alpha_{u+1},...,\alpha_r,\beta_{v+1},...,\beta_s\}$,
one has
$$v_{\mathfrak{p}}\big((\gamma)_n\big)=0.$$
It then follows that for all integers $n\ge 0$, we have
\begin{align*}
v_{\mathfrak{p}}\Big(\frac{(\alpha_1)_n\cdots(\alpha_r)_n}{(\beta_1)_n\cdots(\beta_s)_n}\Big)
=&v_{\mathfrak{p}}\Big(\frac{(\alpha_1)_n\cdots(\alpha_u)_n}{(\beta_1)_n\cdots(\beta_v)_n}\Big) \\
=&v_{p}\Big(\frac{(\alpha_1)_n\cdots(\alpha_u)_n}{(\beta_1)_n\cdots(\beta_v)_n}\Big),
\end{align*}
the last one is because $v_\mathfrak{p}(\alpha)=v_p(\alpha)$ for any
$\alpha\in\mathbb{Q}\setminus\mathbb{Z}_{\le 0}$.
So $F_{\bm{\alpha,\beta}}(z)\in\mathcal{O}_{K,\mathfrak{p}}[[z]]$
if and only if
$$
v_{\mathfrak{p}}\Big(\frac{(\alpha_1)_n\cdots(\alpha_r)_n}
{(\beta_1)_n\cdots(\beta_s)_n}\Big)\ge 0 \ \forall n\ge 0,
$$
which is equivalent to
$$v_{p}\Big(\frac{(\alpha_1)_n\cdots(\alpha_u)_n}{(\beta_1)_n
\cdots(\beta_v)_n}\Big)\ge 0\ \ \forall\ n\ge 0,$$
holds if and only if $F_{\bm{\bm{\mu,\nu}}}(z)\in\mathbb{Z}_p[[z]]$.

Finally, Theorem \ref{thm1.2} applied to $F_{\bm{\bm{\mu,\nu}}}(z)$
gives us the desired result as stated in Proposition \ref{prop3.5}.
This completes the proof of Proposition \ref{prop3.5}.
\end{proof}

{\it Remark.} Assume that $K/\mathbb{Q}$ is a finite Galois extension.
Let $p$ be a prime number such that $\mathfrak{p}:=p\mathcal{O}_{K}$
is a prime ideal in $\mathcal{O}_{K}$.
Let $D_{\mathfrak{p}}$ and $I_{\mathfrak{p}}$ be the decomposition
group and the inertia group of $\mathfrak{p}$, respectively.
The Galois group of $\mathcal{O}_{K}/\mathfrak{p}$ over $\mathbb{Z}/p\mathbb{Z}$
is denoted by $G$. Then we have the following exact sequence:
$$0\rightarrow I_{\mathfrak{p}}\rightarrow D_{\mathfrak{p}}\rightarrow G\rightarrow 0.$$
Since $p$ is inertia in $K$, one has $I_{\mathfrak{p}}=\{1\}$ and
$D_{\mathfrak{p}}={\rm Gal}(K/\mathbb{Q})$. Then ${\rm Gal}(K/\mathbb{Q})=G$
by the above exact sequence, namely, $K/\mathbb{Q}$ is a cyclic extension.

In Section 5, we will make use of the quadratic version of
Proposition \ref{prop3.5}, i.e. Lemma 5.3 below. It will
play an important role in the proof of Theorem \ref{thm1.4}.

\section{Quadratic congruence modulo various primes $p$}

Throughout this section, we let
$f(z)=Az^2+Bz+C\in\mathbb{Z}[z]$ satisfy $A>0$ and $\gcd\big(A,B,C\big)=1$.
We will use the important uniform distribution theorem due to Duke,
Friedlander, Iwaniec \cite{[DFI]} and Toth \cite{[To]} to prove an asymptotic
result which plays an important role in the proof of Theorem \ref{thm1.4}.
Let
$$D:=B^2-4AC$$
and assume that $D$ is not a perfect square in $\mathbb{Z}$.
Let $\mathcal{S}$ be any arithmetic progression which contains infinitely
many primes $p$ with $\big(\frac{D}{p}\big)=1$. The set
\begin{equation} \label{eq4.1}
\mathcal{X}:=\Big\{\frac{v}{p}: p\in\mathcal{S}, 1\le v\le p, f(v)\equiv0\mod p \Big\}
\end{equation}
is an infinite set contained in the interval $[0,1]$.
In this section, we discuss the properties concerning $\mathcal{X}$.
Most importantly, we introduce the uniform distribution theorem of $\mathcal{X}$.
We start with the definition of uniformly distributed sequence.
A sequence $\{x_n\}_{n=1}^{\infty}\subset[0,1]$ is said to be {\it uniformly
distributed} if for all real numbers $a$ and $b$ with $0\le a<b\le 1$, one has
$$\lim_{N\rightarrow\infty}\frac{\#\{x_n:1\leq n\leq N, a\le x_n<b\}}{N}=b-a.$$
It is known that $\mathcal{X}$ is uniformly distributed, which is proved by Duke,
Friedlander and Iwaniec \cite{[DFI]} and Toth \cite{[To]}.

\begin{lem}\cite{[To]} \label{lem4.1}
The set $\mathcal{X}$, roughly ranged by order of $p$, is uniformly distributed.
That is, for $0\le a<b\le 1$,
there is a positive real number $C$ depended only on $\mathcal{S}$ such that
$$\#\Big\{\frac{v}{p}: p\in\mathcal{S}, 1\le v\le p<x, f(v)\equiv0\mod p,
a\le \frac{v}{p}<b\Big\}\thicksim C(b-a)\pi(x),$$
where $\pi(x)$ is the prime counting function.
\end{lem}

From the above uniform distribution theorem, one knows that any rational number in $[0,1]$
can be approximated by a sequence of elements in \ref{eq4.1}.
The following lemma shows that such an approximation may not be too well.

\begin{lem} \label{lem4.2}
Let $f(z)=Az^2+Bz+C$ and $D$ be given as above. Let $M$ and $N$ be positive integers
and $r$ be an integer with $0\le r\le N$. If $p$ is a prime, $v$ and $l$ are
positive integers satisfying that $(p,4AND)=1, v\le p^l, f(v)\equiv 0 \mod p^l$ and
$$p^l>4AN^2(|A|(M+N)^2+|B|(M+N)+|C|),$$
then
$$\Big|\frac{r}{N}-\frac{v}{p^l}\Big| > \frac{M}{p^l}.$$
\end{lem}
\begin{proof}
Let $\frac{r}{N}, p^l,v$ and $M$ be given as in the hypothesises.
Assume on the contrary that
$$\Big|\frac{r}{N}-\frac{v}{p^l}\Big| \leq \frac{M}{p^l}.$$
Then
$$\frac{rp^l}{N}-M\le v\le \frac{rp^l}{N}+M.$$
Denote by $t$ the least nonnegative residue of $p^l$ modulo $N$. Then
$$\Big\lfloor\frac{rp^l}{N}-M\Big\rfloor=\frac{r(p^l-t)}{N}-M$$
and
$$\Big\lfloor\frac{rp^l}{N}+M\Big\rfloor=\frac{r(p^l-t)}{N}+M.$$
So we can write $v=\frac{r(p^l-t)}{N}+a$ for some integer $a\in[-M,M]$.
Since $f(v)\equiv 0 \mod p^l$, it follows that
$$A\Big(\frac{r(p^l-t)}{N}+a\Big)^2+B\Big(\frac{r(p^l-t)}{N}+a\Big)+C\equiv 0 \mod p^l.$$
Equivalently,
\begin{equation} \label{Eq4.2}
4A^2(Na-rt)^2+4ABN(Na-rt)+B^2N^2-DN^2\equiv 0 \mod p^l.
\end{equation}

Since
$$4A^2(Na-rt)^2+4ABN(Na-rt)+B^2N^2=(2A(Na-rt)+BN)^2$$
is a perfect square in $\mathbb{Z}$ while $DN^2$ is not, one derives that
$$4A^2(Na-rt)^2+4ABN(Na-rt)+B^2N^2-DN^2\neq 0.$$
On the other hand, since $r\le N, t\le N$ and $|a|\le M$, one can deduce that
\begin{align*}
&\big| 4A^2(Na-rt)^2+4ABN(Na-rt)+B^2N^2-DN^2 \big| \\
&\le 4AN^2(|A|(M+N)^2+|B|(M+N)+|C|)< p^l.
\end{align*}
Hence the residue of
$$4A^2(Na-rt)^2+4ABN(Na-rt)+B^2N^2-DN^2$$
modulo $p^l$ is nonzero,
which contradicts with (\ref{Eq4.2}). Therefore
$$\Big|\frac{r}{N}-\frac{v}{p^l}\Big| > \frac{M}{p^l}$$
as desired. This finishes the proof of Lemma 4.2.
\end{proof}

\begin{lem} \label{lem4.3}
Let $\alpha$ be a rational number, $\beta$ be a nonzero integer, $M$ be a positive real number
and $\tilde{D}$ be a rational number that is not a perfect square.
Let $p$ be a prime number such that
$$p>4d(\tilde{D})d^2(\alpha)(d(\tilde{D})(M+d(\alpha))^2+d(\tilde{D})|\tilde{D}|\beta^2).$$
Then for all positive integers $l$ and $v$ with $v\le p^l$ and $v^2\equiv \tilde{D} \mod p^l$, one has
\begin{align}
\frac{M}{p^l}< \Big\{\alpha+\beta\frac{v}{p^l}\Big\}< 1-\frac{M}{p^l}.
\end{align}
\end{lem}

\begin{proof}
Let $f(z)=d(\tilde{D})(z^2-\beta^2\tilde{D})$ and suppose that $\tilde{v}$ with
$0\le \tilde{v}\le p^l$ satisfies $f(\tilde{v})\equiv 0\mod p^l$.
It is easy to see that $\{\frac{\beta v}{p^l}\}=\frac{\tilde{v}}{p^l}$.
Obviously, for any positive integer $l$, we have
$$p^l\ge p>4d(\tilde{D})d^2(\alpha)(d(\tilde{D})(M+d(\alpha))^2+d(\tilde{D})|\beta^2\tilde{D}|).$$
Then by Lemma \ref{lem4.2}, one can derive that
$$\Big|1-\big\{\alpha \big\}-\frac{\tilde{v}}{p^l}\Big|>\frac{M}{p^l}.$$
In other words, one has either
$$0\le\Big\{\frac{\beta v}{p^l}\Big\}< 1-\{\alpha\}-\frac{M}{p^l},$$
or
$$1-\{\alpha\}+\frac{M}{p^l}<\Big\{\frac{\beta v}{p^l}\Big\}<1.$$
But
\begin{align*}
\Big\{\alpha+\frac{\beta v}{p^l}\Big\}
=\bigg\{\begin{array}{cl}
       \{\alpha\}+\big\{\frac{\beta v}{p^l}\big\}, & {\rm if}\
       \{\alpha\}+\big\{\frac{\beta v}{p^l}\big\}<1, \\
       \{\alpha\}+\big\{\frac{\beta v}{p^l}\big\}-1, & {\rm otherwise.}
     \end{array}
\end{align*}
Therefore either
$$\{\alpha\}\le \Big\{\alpha+\frac{\beta v}{p^l}\Big\}<1-\frac{M}{p^l}
$$
or
$$\frac{M}{p^l}<\Big\{\alpha+\frac{\beta v}{p^l}\Big\}<\{\alpha\}.$$
That is, (4.3) is true. This concludes the proof of Lemma \ref{lem4.3}.
\end{proof}

In the rest part of this section, we consider a special subset
$\mathcal{Y}$ of $\mathcal{X}$ which is defined by
$$\mathcal{Y}:=\Big\{\frac{v}{p}: p\in\mathcal{S},0\le v<p,  f(v)\equiv 0 \mod p^2 \Big\},$$
and set
$$S_{\mathcal{Y}}(x):=\#\Big\{\frac{v}{p}\in\mathcal{Y}:p\le x\Big\}.$$
Meanwhile, for $1\le a\le A$, we define
$$\mathcal{Y}_a:=\Big\{\frac{v}{p}: p\in\mathcal{S},0\le v<p, f(v)=ap^2 \Big\},$$
and let
$$S_a(x):=\#\Big\{\frac{v}{p}\in\mathcal{Y}_a:p\le x\Big\}.$$
One can deduce from the uniform distribution theorem that the following result is true.

\begin{lem} \label{lem4.4}
Each of the following is true:

{\rm (i).} $S_a(x)=o(\pi(x))$ for all integers $a$ with $1\le a\le A$.

{\rm (ii).} $S_{\mathcal{Y}}(x)=o(\pi(x))$.
\end{lem}

\begin{proof}
(i). If the set $\mathcal{Y}_a$ is finite, then part (i) is clearly true.

If there is an integer $a$ with $1\le a\le A$ such that the set $\mathcal{Y}_a$
is infinite, then as $v/p$ runs through $\mathcal{Y}_a$, we have
$$\lim_{p\rightarrow\infty}\frac{v^2}{p^2}
=\lim_{p\rightarrow\infty}\frac{Av^2+Bv+C}{Ap^2}=\frac{a}{A}.$$
It follows that $\lim_{p\rightarrow\infty}v/p=\sqrt{a/A}$
as $v/p$ runs through $\mathcal{Y}_a$.
So if $\mathcal{Y}_a$ contains infinitely many elements,
then $\mathcal{Y}_a$ holds exactly one accumulation point.

Assume now that $S_{\mathcal{Y}}(x)\neq o(\pi(x))$.
In what follows, we show that we will arrive at a contradiction.

On the one hand, since $0<S_a(x)/\pi(x)\le 1$ for all $x>0$,
there exists a sequence $\{x_n\}_{n\ge 1}$ of positive real
numbers and a constant $c>0$ such that
$x_n\rightarrow\infty$ as $n\rightarrow\infty$ and
$$S_a(x_n)\thicksim c\pi(x_n).$$

On the other hand, since $\sqrt{a/A}$ is the unique accumulation point of
$\mathcal{Y}_a$, it follows that for any $\epsilon>0$,
all but finitely many elements of $\mathcal{Y}_a$ are contained in the interval
$$I:=[\sqrt{a/A}-\epsilon/2, \sqrt{a/A}+\epsilon/2)\cap [0,1].$$
Since $\mathcal{Y}_a\subset\mathcal{X}$, all but finitely many elements of
$\mathcal{Y}_a$ are contained in $\mathcal{X}\cap I$.
Thus one deduces that
$$\#\Big\{\frac{v}{p}\in\mathcal{X}\cap I: p\le x_n\Big\}+O(1)\ge S_a(x_n).$$
But Lemma 4.1 tells us that for all large enough $x_n$,
$$C|I|\pi(x_n)\thicksim\#\Big\{\frac{v}{p}\in\mathcal{X}\cap
I:p\le x_n\Big\}+O(1),$$
where $|I|\le \epsilon$ and $|I|$ stands for the length of $I$.
So one derives that
$$C|I|\pi(x_n)\ge c\pi(x_n).$$
That is, $|I|\ge c/C$.
This is ridiculous since $\epsilon$ can be chosen extremely
small while $c$ is a constant. Thus the assumption
$S_{\mathcal{Y}}(x)\neq o(\pi(x))$ is false and so $S_a(x)=o(\pi(x))$.

(ii). For $v/p\in\mathcal{Y}$, since $f(v)<(A+1)p^2$ for all large primes $p$,
it follows that $f(v)=ap^2$ for some integer $a$ with $1\le a \le A$. Then
$$\mathcal{Y}-\bigcup_{a=1}^{A}\mathcal{Y}_a$$
is a finite set and
$$\Big|S_{\mathcal{Y}}(x)-\sum_{a=1}^{A}S_a(x)\Big|=O(1).$$
Thus the desired result $S_{\mathcal{Y}}(x)=o(\pi(x))$
is implied by part (i). This ends the proof of Lemma \ref{lem4.4}.
\end{proof}

As an application of Lemma \ref{lem4.4}, for any integer $a$ with $1\le a\le A$,
we consider the set
$$\mathcal{Z}_a:=\{p\in\mathcal{S}: \exists\ n\in\mathbb{Z}_{>0}\ {\rm s.t.}\ f(n)=ap^2\}$$
and define
$$T_a(x):=\#\{p\le x: p\in\mathcal{Z}_a \}.$$
Particularly, $T_1(x)$ is the number of prime squares in the sequence
$\{f(n)\}_{n=1}^{\infty}$. We can obtain a rough upper bound for $T_a(x)$.

\begin{cor} \label{cor4.5}
We have $T_a(x)=o(\pi(x))$.
\end{cor}

\begin{proof}
First we show that $\mathcal{Z}_A$ is finite.
Assume conversely that there are infinitely many primes $p\in \mathcal{S}$
and a positive integer $n$ such that $f(n)=Ap^2$. Then
$$4Af(n)=(2An+B)^2-D=4A^2p^2.$$
Thus there exists infinitely many primes $p$ such that both
$4A^2p^2$ and $4A^2p^2-D$ are squares of integers.
On the other hand, one can easily check that if $p>(|D|+1)/4A$,
then
$$(2Ap-1)^2< 4A^2p^2-D<(2Ap+1)^2.$$
But $D$ is not a perfect square in $\mathbb{Z}$. So $D\ne 0$
which implies that
$$4A^2p^2-D\ne (2Ap)^2.$$
Therefore
$4A^2p^2-D$ is not a square of an integer. It follows that there
are at most finitely many primes $p$ such that $4A^2p^2$ and $4A^2p^2-D$
are squares of integers. This is a contradiction.
Hence $\mathcal{Z}_A$ is finite. Thereby $T_A(x)=o(\pi(x))$.

In what follows, let $1\le a<A$. If $\mathcal{Z}_a$ is a finite set,
then it is obvious true that $T_a(x)=o(\pi(x))$.
Hence we need ony to consider the case that $\mathcal{Z}_a$ is infinite.
Let now $\mathcal{Z}_a$ be infinite.
Since $f(n)\thicksim An^2$ as $n\rightarrow\infty$ and $\frac{A}{a}>1$,
there exists a positive integer $N$ such that $An^2<\frac{A}{a}f(n)$
for all $n>N$. Thus for all integers $n$ and primes $p$ with $n>N$
and $f(n)=ap^2$, one has
$$n^2<\frac{f(n)}{a}=p^2.$$
That is, $n<p$. It then follows that $\frac{n}{p}\in\mathcal{Y}_a$.
Therefore for all large enough positive real numbers $x$, one has
$$T_a(x)\le S_a(x)+O(1).$$
Hence by Lemma \ref{lem4.4} (i), the desired result
$T_a(x)=o(\pi(x))$ follows immediately.
The proof of Corollary \ref{cor4.5} is complete.
\end{proof}

\section{N-integrality of the hypergeometric series
with quadratic parameters and Proof of Theorem \ref{thm1.4}}

In this section, we investigate the N-integrality of the
hypergeometric series with parameters from quadratic fields
and show Theorem \ref{thm1.4}.

At first, let us recall the notations. Let $D$ be an arbitrary
given square-free integer other than 1.
Throughout this section, we always let $K=\mathbb{Q}(\sqrt{D})$.
Let $\bm{\alpha}=(\alpha_1,...,\alpha_r)$ and $\bm{\beta}=({\beta_1,...,\beta_s})$
with $\alpha_i$ and $\beta_j$ belonging to $K\setminus\mathbb{Z}_{\le 0}$.
Let
$$u:=\#\{1\le i\le r: \alpha_i\in\mathbb{Q}\}$$
and
$$v:=\#\{1\le j\le s: \beta_j\in\mathbb{Q}\}.$$
Without loss of any generality, suppose that $\alpha_1,...,\alpha_u$ and $\beta_1,...,\beta_v$
are rational numbers while $\alpha_{u+1},...,\alpha_{r}$ and $\beta_{v+1},...,\beta_s$ are not.
Write
$$\bm{\alpha}=(\alpha_{11}+\alpha_{21}\sqrt{D}, ...,
\alpha_{1r}+\alpha_{2r}\sqrt{D})$$
and
$$\bm{\beta}=(\beta_{11}+\beta_{21}\sqrt{D},...,\beta_{1s}+\beta_{2s}\sqrt{D})$$
with $\alpha_{1i},\alpha_{2i},\beta_{1j}$ and $\beta_{2j}$ being elements in
$\mathbb{Q}$ for all integers $i$ and $j$ such that $1\le i\le r$ and $1\le j\le s$.
Then $\alpha_{21}=\cdots\alpha_{2u}=\beta_{21}=\cdots=\beta_{2v}=0$, and
$\alpha_{2i}, \beta_{2j}$ are nonzero for all integers $i$ and $j$ such that
$u+1\le i\le r$ and $v+1\le j\le s$. Set
$$\bm{\alpha}_1:=(\alpha_{11},..., \alpha_{1r}),\ \
\bm{\beta}_1:=(\beta_{11},..., \beta_{1s}),$$
$$\bm{\alpha}_2:=(\alpha_{21},..., \alpha_{2r}),\ \
\bm{\beta}_2:=(\beta_{21},..., \beta_{2s}),$$
and
$$\bm{\mu}:=(\alpha_1,...,\alpha_u),\ \
\bm{\nu}:=(\beta_1,...,\beta_v).$$
Let $d_{\bm{\alpha_1,\beta_1}}$ (resp. $d_{\bm{\alpha_2,\beta_2}}$, $d_{\bm{\bm{\mu,\nu}}}$) be
the least common multiple of the exact denominators of components of $\bm{\alpha_1}$ and
$\bm{\beta_1}$ (resp. $\bm{\alpha_2}$ and $\bm{\beta_2}$, $\bm{\mu}$ and $\bm{\nu}$).
Define
$$E:={\rm lcm}(4D, d_{\bm{\alpha_1,\beta_1}},d_{\bm{\alpha_2,\beta_2}})$$
and let $G$ be the multiplicative group of residue class ring modulo $E$,
namely, $G:=(\mathbb{Z}/E\mathbb{Z})^{\times}$. Throughout this section,
we identity any element $a\in G$ with an integer $a\in\{1, ..., E\}$ satisfying that $(a,E)=1$.

Write $D=(-1)^{\lambda_1}2^{\lambda_2}D'$,
where $\lambda_1,\lambda_2\in\{0,1\}$ and $D'$ is a positive odd square-free integer other
than 1. Let $p$ and $q$ be two distinct prime numbers such that $p\equiv q \mod 4D$. Claim that
$$(-1)^{\frac{\lambda_2(p^2-1)}{8}}=(-1)^{\frac{\lambda_2(q^2-1)}{8}}.$$
In fact, if $\lambda_2=0$, then this is clear true since both sides are 1.
Now let $\lambda_2=1$. Since $p$ and $q$ are odd, one then derives that
$p^2-q^2=(p-q)(p+q)\equiv 0 \mod 16$. It infers that
$$(-1)^{\frac{\lambda_2(p^2-1)}{8}}=(-1)^{\frac{\lambda_2(q^2-1)}{8}}.$$
Thus the claim is proved. Now we deduce from the claim and the quadratic
reciprocity law about the Jacobi symbols (see, for example, Proposition
5.2.2 of \cite{[IR]}) that
\begin{align*}
\Big(\frac{D}{p}\Big)&=\Big(\frac{-1}{p}\Big)^{\lambda_1}\Big(\frac{2}{p}\Big)^{\lambda_2}\Big(\frac{D'}{p}\Big) \\
&=(-1)^{\frac{\lambda_1(p-1)}{2}}(-1)^{\frac{\lambda_2(p^2-1)}{8}}(-1)^{\frac{(D'-1)(p-1)}{4}}\Big(\frac{p}{D'}\Big)\\
&=(-1)^{\frac{\lambda_1(q-1)}{2}}(-1)^{\frac{\lambda_2(q^2-1)}{8}}(-1)^{\frac{(D'-1)(q-1)}{4}}\Big(\frac{q}{D'}\Big)\\
&=\Big(\frac{-1}{q}\Big)^{\lambda_1}\Big(\frac{2}{q}\Big)^{\lambda_2}\Big(\frac{D'}{q}\Big)\\
&=\Big(\frac{D}{q}\Big).
\end{align*}
It follows that for any $a\in\big(\mathbb{Z}/4D\mathbb{Z}\big)^{\times}$,
if there is a prime number $p$ satisfying that
$$p\equiv a \mod 4D\ \ {\rm and}\ \ \Big(\frac{D}{p}\Big)=1,$$
then for all primes $p$ with $p\equiv a \mod 4D$, one has $\big(\frac{D}{p}\big)=1$.
Hence one can define
\begin{align*}
&H:=\{a\in(\mathbb{Z}/E\mathbb{Z})^{\times} : \Big(\frac{D}{p}\Big)=1 \
{\rm holds \ for\ all\ primes}\ p\equiv a \mod E\},\\
&I:=G\setminus H.
\end{align*}
To start the proof of Theorem \ref{thm1.4}, we need to develop a series of lemmas.
The following result shows that $H$ is a subgroup of the group
$(\mathbb{Z}/E\mathbb{Z})^{\times}$.
\begin{lem} \label{lem5.1}
Let $H$ be given as above. Then $H$ is a subgroup of $(\mathbb{Z}/E\mathbb{Z})^{\times}$.
\end{lem}
\begin{proof}
Since $(\mathbb{Z}/E\mathbb{Z})^{\times}$ is a finite group, it follows
that if $H$ is multiplication-closed, then $H$ is a subgroup.
Pick any $a_1, a_2\in H$. Then for any primes $p_1$ and $p_2$
with $p_1\equiv a_1 \mod E$ and $p_2\equiv a_2 \mod E$, one has
$\big(\frac{D}{p_1}\big)=1$ and $\big(\frac{D}{p_2}\big)=1$.
Let $p$ be a prime number such that $p\equiv a_1a_2 \mod E$.
So to finish the proof of Lemma 5.1, it remains to show that
$\big(\frac{D}{p}\big)=1$, which will be done in what follows.

As $D$ is square-free, one may write $D=(-1)^{\lambda_1}2^{\lambda_2}D'$
with $\lambda_1,\lambda_2\in\{0,1\}$ and $D'$ being a positive odd
square-free integer other than 1. Claim that
$$(-1)^{\frac{\lambda_2(p^2-1)}{8}}=(-1)^{\frac{\lambda_2(p_1^2p_2^2-1)}{8}}.$$
If $\lambda_2=0$, then the claim is clear true.
If $\lambda_2=1$, then $q\equiv p_1p_2 \mod 8$, and so
$$p^2-p_1^2p_2^2=(p+p_1p_2)(p-p_1p_2)\equiv 0\mod 16.$$
It then follows that
$$(-1)^{\frac{\lambda_2(p^2-1)}{8}}=(-1)^{\frac{\lambda_2(p_1^2p_2^2-1)}{8}}.$$
But the hypothesis $p\equiv p_1p_2 \mod 4$ implies that
$$(-1)^{\frac{\lambda_1(p-1)}{2}}=(-1)^{\frac{\lambda_1(p_1p_2-1)}{2}}.$$
Noting that $p_1, p_2, p$ and $D'$ are all odd and
$p\equiv p_1p_2 \mod {D'}$, one obtains that
\begin{align*}
\Big(\frac{D}{p}\Big)&=\Big(\frac{-1}{p}\Big)^{\lambda_1}\Big(\frac{2}{p}\Big)^{\lambda_2}\Big(\frac{D'}{p}\Big) \\
&=(-1)^{\frac{\lambda_1(p-1)}{2}}(-1)^{\frac{\lambda_2(p^2-1)}{8}}(-1)^{\frac{(D'-1)(p-1)}{4}}\Big(\frac{p}{D'}\Big) \\
&=(-1)^{\frac{\lambda_1(p_1p_2-1)}{2}}(-1)^{\frac{\lambda_2(p_1^2p_2^2-1)}{8}}(-1)^{\frac{(D'-1)(p_1p_2-1)}{4}}
\Big(\frac{p_1p_2}{D'}\Big) \\
&=(-1)^{\frac{\lambda_1(p_1+p_2-2)}{2}}(-1)^{\frac{\lambda_2(p_1^2+p_2^2-2)}{8}}(-1)^{\frac{(D'-1)(p_1+p_2-2)}{4}}
\Big(\frac{p_1}{D'}\Big)\Big(\frac{p_2}{D'}\Big) \\
&=\Big(\frac{D}{p_1}\Big)\Big(\frac{D}{p_2}\Big)=1
\end{align*}
as required. This completes the proof of Lemma \ref{lem5.1}.
\end{proof}

Let $a\in H$ and $p$ be a prime number such that $ap \equiv 1 \mod E$.
Then $a^{-1}\in H$ and so $\big(\frac{D}{p}\big)=1$.
Hence $p\mathcal{O}_K=\mathfrak{p}_1\mathfrak{p}_2$,
where $\mathfrak{p}_1$ and $\mathfrak{p}_2$
are two different prime ideals in $\mathcal{O}_K$.
So $\mathfrak{p}\mid p$, as discussed in the beginning of Section 3, we have
$$[K_{\mathfrak{p}}:\mathbb{Q}_p]=f(\mathfrak{p}|p)e(\mathfrak{p}|p)=1.$$
Let ${\rm Hom}(K,\mathbb{C}_p)$ denote the set of all the two distinct
monomorphisms from $K$ to $\mathbb{C}_p$.
For any $\sigma_p\in {\rm Hom}(K,\mathbb{C}_p)$, there is a prime ideal
$\mathfrak{p}\in\{\mathfrak{p}_1,\mathfrak{p}_2\}$ and a monomorphism
$\widehat\sigma_p: K_{\mathfrak{p}}\mapsto\mathbb{C}_p$ such that
the following diagram is commutative:
$$
\xymatrix{
                & K_{\mathfrak{p}} \ar[dr]^{\widehat\sigma_p}  \\
 K \ar[ur]^{\iota} \ar[rr]^{\sigma_p} & &  \mathbb{C}_p,    }
$$
where $\iota$ is the canonical monomorphism from $K$ to $K_{\mathfrak{p}}$.
Let $\sigma_p(K)$ be the image of $K$ and $\widehat{\sigma_p(K)}$
be its completion in $\mathbb{C}_p$. Then $\widehat\sigma_p$ is an isomorphism
from $K_{\mathfrak{p}}$ to $\widehat{\sigma_p(K)}$. Thus
$$[\widehat{\sigma_p(K)}:\mathbb{Q}_p]=[K_{\mathfrak{p}}:\mathbb{Q}_p]=1.$$
So $\sigma_p(K)\subset\widehat{\sigma_p(K)}=\mathbb{Q}_p$
for any $\sigma_p\in {\rm Hom}(K,\mathbb{C}_p)$.
We have the following equivalent form of the uniform distribution of roots of
quadratic congruence modulo prime numbers in an arithmetic progression.

\begin{lem} \label{lem5.2}
Let $\gamma\in K\setminus \mathbb{Q}$ and $f_{\gamma}(z)$ be its monic
minimal polynomial over $\mathbb{Q}$.
Let $C_{\gamma}$ be a positive integer such that $C_{\gamma}f_{\gamma}(z)$ is a
primitive polynomial with integer coefficients. Let $a\in H$ and $\mathcal{S}$ be the
arithmetic progression $\{n\ge 0: an\equiv 1\mod E\}$.
Then for all real numbers $t_1$ and $t_2$ with $0\le t_1<t_2\le 1$, there is a
positive real number $c$ depended only on $\mathcal{S}$ and $\gamma$ such that
\begin{align} \label{eq5.2.1}
&\#\Big\{\frac{T_p(\sigma_p(\gamma))}{p}: p\in\mathcal{S}, p\nmid C_{\gamma}, p<x,
\sigma_p\in {\rm Hom}(K,\mathbb{C}_p), t_1\le \frac{T_p(\sigma_p(\gamma))}{p}<t_2 \Big\} \\
&\sim c(t_2-t_1)\pi(x). \notag
\end{align}
\end{lem}

\begin{proof}
Since $a\in H$, for any prime number $p$ with $ap\equiv 1$
and any $\sigma_p\in{\rm Hom}(K,\mathbb{C}_p)$, one has
$\sigma_p(\gamma)\in \sigma_p(K)\subset\mathbb{Q}_p$.
Moreover, for any prime number $p$ with $p\nmid C_{\gamma}$, one has
$f_{\gamma}(z)\in\mathbb{Z}_p[z]$. Since $f_{\gamma}$ is monic,
$f_{\gamma}(\sigma_p(\gamma))=0$
and $\mathbb{Z}_p$ is integral closed, it then follows that
for all prime numbers $p$ with $p\nmid C_{\gamma}$ and $ap\equiv 1\mod E$
and all $\sigma_p\in{\rm Hom}(K,\mathbb{C}_p)$, one has
$\sigma_p(\gamma)\in\mathbb{Z}_p$, and so $R_{p,l}(\sigma_p(\gamma))$
is well defined for any positive integer $l$.
But $R_{p,l}(\sigma_p(\gamma))-\sigma_p(\gamma)\equiv 0\mod p^l$
and $p\nmid C_{\gamma}$. Hence we have
\begin{align*} 
C_{\gamma}f_{\gamma}(R_{p,l}(\sigma_p(\gamma)))\equiv
f_{\gamma}(R_{p,l}(\sigma_p(\gamma)))\equiv f_{\gamma}(\sigma_p(\gamma))\equiv 0\mod p^l.
\end{align*}
So for all the prime numbers $p$ with $p\nmid C_{\gamma}$ and
$ap\equiv 1\mod E$, one has
\begin{align} \label{eq5.2.3}
\Big\{\frac{v}{p}: 0\le v<p, C_{\gamma}f_{\gamma}(v)\equiv 0\mod p \Big\}=
\Big\{\frac{R_p(\sigma_p(\sqrt{\tilde D}))}{p}: \sigma_p\in {\rm Hom}(K,\mathbb{C}_p) \Big\}.
\end{align}

By Lemma \ref{lem4.1}, we know that for all real numbers $t_1$ and $t_2$ with $0\le t_1<t_2\le 1$,
there is a positive real number $c$ depended only on $\mathcal{S}$ and $\gamma$ such that
$$\Big\{\frac{v}{p}: 1\le v\le p<x, p\in\mathcal{S}, t_1\le \frac{v}{p}<t_2,
C_{\gamma}f_{\gamma}(v)\equiv 0\mod p \Big\}\sim c(t_2-t_1)\pi(x).$$
Then from (\ref{eq5.2.3}) one can deduce that
\begin{align*}
&\#\Big\{\frac{R_p(\sigma_p(\gamma))}{p}: p\in\mathcal{S}, p\nmid C_{\gamma}, p<x,
\sigma_p\in {\rm Hom}(K,\mathbb{C}_p), t_1\le \frac{T_p(\sigma_p(\gamma))}{p}<t_2 \Big\} \\
&\sim c(t_2-t_1)\pi(x). \notag
\end{align*}
Since
$$\frac{T_p(\sigma_p(\gamma))}{p}=1-\frac{R_p(\sigma_p(\gamma))}{p}$$
holds for all the prime numbers $p$ with $p\nmid C_{\gamma}$ and $ap\equiv 1\mod E$ and
for all $\sigma_p\in{\rm Hom}(K,\mathbb{C}_p)$, the desired result (\ref{eq5.2.1})
follows immmediately. The proof of Lemma \ref{lem5.2} is complete.
\end{proof}

Now let $a\in I=G\setminus H$ and $p$ be an odd prime
with $ap\equiv 1\mod E$. Then $\big(\frac{D}{p}\big)=-1$.
So there exists a unique prime ideal $\mathfrak{p}$ in $K$ that satisfies $\mathfrak{p}\mid p$.
That is, $\mathfrak{p}=p\mathcal{O}_K$.
The following result is the quadratic version of Proposition \ref{prop3.5}.

\begin{lem} \label{lem5.3}
Let $a\in I$ and $p$ be a prime number such that $ap\equiv 1\mod E$,
\begin{equation} \label{eq5.3.1}
p>\max(M_{\bm{\bm{\mu,\nu}}},3vd_{\bm{\bm{\mu,\nu}}})\ {\rm and}\
p\not |\prod_{\gamma\in\{\alpha_{u+1},...,\alpha_r,
\beta_{v+1,...,\beta_s}\}}[\mathcal{O}_K:\mathbb{Z}[E\gamma]].
\end{equation}
Then $\mathfrak{p}:=p\mathcal{O}_K$ is the unique prime
ideal in $\mathcal{O}_K$ with $\mathfrak{p}\mid p$,
and each of the following is true.

{\rm (i).} If $u<v$, then $F_{\bm{\alpha,\beta}}(z)\not\in\mathcal{O}_{K,\mathfrak{p}}[[z]]$.

{\rm (ii).} If $u>v$, then $F_{\bm{\alpha,\beta}}(z)\in\mathcal{O}_{K,\mathfrak{p}}[[z]]$
if and only if $\delta_{\bm{\mu,\nu}}(x,a)\ge 0$ for all $x\in\mathbb{R}$.

{\rm (iii).} If $u=v$, then $F_{\bm{\alpha,\beta}}(z)\in\mathcal{O}_{K,\mathfrak{p}}[[z]]$
if and only if
$$\sum_{l=1}^{h}\delta_{\bm{\mu,\nu}}(a^l\beta_k,a^l)\ge 0$$
holds for all $h\in\{1,...,{\rm ord}(a)\} $ and $k\in\{1,...,v\}$ and
$$\delta_{\bm{\bm{\mu,\nu}}}\Big(a^l\Big(\frac{e}{d_{\bm{\bm{\mu,\nu}}}}+m_{\bm{\mu,\nu}}\Big)\Big)\ge 0$$
holds for all $e\in\{1,...,d_{\bm{\alpha,\beta}}\}$ and $l\in\{1,...,{\rm ord}(a)\}$,
where ${\rm ord}(a)$ is the order of $a$ modulo $E$.
\end{lem}

Consequently, we provide some properties about the function
$\langle \cdot \rangle$ and the operator $T_{p, l}$.

\begin{lem} \label{lem5.4}
Let $x$ and $y$ be real numbers. Let $p$ be a prime number and $l$ be a positive
integer. Let $w$ and $z$ be $p$-adic integers. Then each of the following is true.

{\rm (i).} One has
\begin{align*}
\langle x+y\rangle
=\bigg\{\begin{array}{cl}
       \langle x\rangle +\langle y\rangle, & {\it if}\
       \langle x\rangle +\langle y\rangle\le 1, \\
       \langle x\rangle +\langle y\rangle-1, & {\it otherwise.}
     \end{array}
\end{align*}

{\rm (ii).} One has
\begin{align*}
\langle x-y\rangle
=\bigg\{\begin{array}{cl}
       \langle x\rangle -\langle y\rangle, & {\it if}\
       \langle x\rangle>\langle y\rangle, \\
       1+\langle x\rangle-\langle y\rangle, & {\it otherwise.}
     \end{array}
\end{align*}

{\rm (iii).} One has
\begin{align*}
\frac{T_{p,l}(w+z)}{p^l}
=\bigg\{\begin{array}{cl}
       0,   & {\it if}\ w+z\equiv 0\mod p^l, \\
       \big\langle\frac{T_{p,l}(w)+T_{p,l}(z)}{p^l}\big\rangle, & {\it otherwise.}
     \end{array}
\end{align*}
\end{lem}

\begin{proof}
(i). Apparently, $x+y-(\langle x\rangle+\langle y\rangle) \in\mathbb{Z}$.
Since $\langle x+y\rangle$ is the unique real number in the interval $(0, 1]$
such that $x+y-\langle x+y\rangle\in \mathbb{Z}$, it follows that
if $\langle x\rangle +\langle y\rangle\le 1$, then
$\langle x+y\rangle=\langle x\rangle+\langle y\rangle$. Furthermore,
if $\langle x\rangle+\langle y\rangle>1$,
then $0<\langle x\rangle+\langle y\rangle-1\le 1$. But
$x+y-(\langle x\rangle+\langle y\rangle-1)\in\mathbb{Z}$.
So one can deduce that $\langle x+y\rangle=\langle x\rangle+\langle y\rangle-1$
as desired. This proves part (i).

(ii). If $y\in\mathbb{Z}$, then $\langle x\rangle\le 1=\langle y\rangle$ and
$\langle x-y\rangle=\langle x\rangle=1+\langle x\rangle-\langle y\rangle$.
That is, part (ii) is true if $y\in\mathbb{Z}$.

Let now $y\not\in\mathbb{Z}$. Then $\langle y\rangle\not\in\mathbb{Z}$. Since
$-y-(-\langle y\rangle)=\langle y\rangle-y\in\mathbb{Z}$, one has
$\langle -y\rangle=\langle -\langle y\rangle \rangle=1-\langle y\rangle$.

If $\langle x\rangle>\langle y\rangle$, then $\langle x\rangle +\langle -y\rangle
=\langle x\rangle+1-\langle y\rangle>1$. It follows from part (i) that
$\langle x-y\rangle=\langle x\rangle +\langle -y\rangle-1=\langle x\rangle-\langle y\rangle$.

If $\langle x\rangle\le \langle y\rangle$, then $\langle x\rangle +\langle -y\rangle
=\langle x\rangle+1-\langle y\rangle\le 1$. So from part (i) it follows that
$\langle x-y\rangle=\langle x\rangle +\langle -y\rangle=1+\langle x\rangle-\langle y\rangle$.
Hence part (ii) is proved.

(iii). For any $w,z\in\mathbb{Z}_p$, one has
$$T_{p,l}(w+z)\equiv -w-z\equiv T_{p,l}(w)+T_{p,l}(z) \mod p^l.$$
Since $T_{p, l}(w), T_{p, l}(z)$ and $T_{p, l}(w+z)$ are all in the set
$\{0, 1, ..., p^l-1\}$, one then derives that
$$\frac{T_{p,l}(w)+T_{p,l}(z)}{p^l}-\frac{T_{p,l}(w+z)}{p^l}\in\mathbb{Z}.$$

If $w+z\equiv 0\mod p^l$, then $T_{p,l}(w+z)\equiv 0 \mod p^l$
which implies that $T_{p,l}(w+z)=0$ and so the desired result
$\frac{T_{p,l}(w+z)}{p^l}=0$ follows immediately.

If $w+z\not\equiv 0\mod p^l$, then $1\le T_{p, l}(w+z)<p^l$. Hence
$$\Big\langle\frac{T_{p,l}(w)+T_{p,l}(z)}{p^l}\Big\rangle
=\Big\langle \frac{T_{p,l}(w+z)}{p^l}\Big\rangle=\frac{T_{p,l}(w+z)}{p^l}$$
as claimed. So part (iii) is proved.

The proof of Lemma \ref{lem5.4} is complete.
\end{proof}

Let us now introduce some notations. Write
$$\tilde{D}:=D/d_{\bm{\alpha_2,\beta_2}}^2, \tilde{\alpha}_{2i}:=\alpha_{2i}d_{\bm{\alpha_2,\beta_2}}
\ {\rm and}\ \tilde{\beta}_{2j}:=\beta_{2j}d_{\bm{\alpha_2,\beta_2}}$$
for all integers $i$ and $j$ with $1\le i\le r$ and $1\le j\le s$.
Then $\tilde{\alpha}_{2i}$ and $\tilde{\beta}_{2j}$ are rational
integers. Moreover, one has
$\alpha_i=\alpha_{1i}+\tilde{\alpha}_{2i}\sqrt{\tilde{D}}$ and
$\beta_j=\beta_{1j}+\tilde{\beta}_{2j}\sqrt{\tilde{D}}$.
We have the following result.

\begin{lem} \label{lem5.5}
Let $l$ be a positive integer.
Let $\gamma _1\in \mathbb{Q}$ and $\tilde{\gamma}_2\in\mathbb{Z}$
such that $\gamma_1+\tilde{\gamma}_2\sqrt{\tilde{D}}\not\in\mathbb{Z}_{\le 0}$.
Let $p$ be an odd prime number and $a\in H$ be such that $ap\equiv 1\mod E$,
$p\ge d(\gamma_1)(|\lfloor 1-\gamma_1 \rfloor|+\langle \gamma_1 \rangle+1)$
and $v_p(\gamma_1^2-\tilde{D}\tilde{\gamma}_2^2)=0$.
Let $\sigma$ be a monomorphism from $K$ to $\mathbb{C}_p$. Then
$$\frac{T_{p,l}(\sigma(\gamma_1+\tilde{\gamma}_2\sqrt{\tilde{D}}))}{p^l}
=\Big\langle a^l\gamma_1-\frac{\gamma_1}{p^l}
+\frac{\tilde{\gamma}_2T_{p,l}(\sigma(\sqrt{\tilde{D}}))}{p^l}\Big\rangle.$$
\end{lem}

\begin{proof}
Since $p\ge d(\gamma_1)(|\lfloor 1-\gamma_1 \rfloor|
+\langle \gamma_1 \rangle+1)>d(\gamma_1)$, one has
$\gamma_1\in\mathbb{Z}_p$. As $a\in H$ and $ap\equiv 1\mod E$, one derives that
$\sigma(\sqrt{\tilde{D}})\in \mathbb{Q}_p$ and $\tilde{D}\in\mathbb{Z}_p$.
This implies that $\sigma(\sqrt{\tilde{D}})\in \mathbb{Z}_p$.
Hence $\sigma(\gamma_1+\tilde{\gamma}_2\sqrt{\tilde{D}})$ and
$\sigma(\gamma_1-\tilde{\gamma}_2\sqrt{\tilde{D}})$ are $p$-adic integers.
Moreover, since $\gamma_1^2-\tilde{D}\tilde{\gamma}_2^2\in\mathbb{Q}$
and $\gamma_1^2-\tilde{D}\tilde{\gamma}_2^2\not\equiv 0\mod p$,
we have
\begin{align*}
&v_p(\sigma(\gamma_1+\tilde{\gamma}_2\sqrt{\tilde{D}}))
+v_p(\sigma(\gamma_1-\tilde{\gamma}_2\sqrt{\tilde{D}}))\\
=&v_p(\sigma(\gamma_1^2-\tilde{D}\tilde{\gamma}_2^2))\\
=&v_p(\gamma_1^2-\tilde{D}\tilde{\gamma}_2^2)=0.
\end{align*}
One then deduces that
$$v_p(\sigma(\gamma_1+\tilde{\gamma}_2\sqrt{\tilde{D}}))
=v_p(\sigma(\gamma_1-\tilde{\gamma}_2\sqrt{\tilde{D}}))=0.$$
It follows that $\sigma(\gamma_1+\tilde{\gamma}_2\sqrt{\tilde{D}})\not\equiv 0 \mod p^l$.
So by Lemmas \ref{lem5.4} (iii) and \ref{lem2.3}, one obtains that
\begin{align*}
\frac{T_{p,l}(\sigma(\gamma_1+\tilde{\gamma}_2\sqrt{\tilde{D}}))}{p^l}
&=\Big\langle \frac{T_{p,l}(\gamma_1)}{p^l}
+\frac{T_{p,l}(\tilde{\gamma}_2\sigma(\sqrt{\tilde{D}}))}{p^l}\Big\rangle \\
&=\Big\langle \mathfrak{D}_p^l(\gamma_1)-\frac{\gamma_1}{p^l}
+\frac{T_{p,l}(\tilde{\gamma}_2\sigma(\sqrt{\tilde{D}}))}{p^l}\Big\rangle \\
&=\Big\langle \langle a^l\gamma_1\rangle-\frac{\gamma_1}{p^l}
+\frac{\tilde{\gamma}_2T_{p,l}(\sigma(\sqrt{\tilde{D}}))}{p^l}\Big\rangle\\
&=\Big\langle a^l\gamma_1-\frac{\gamma_1}{p^l}
+\frac{\tilde{\gamma}_2T_{p,l}(\sigma(\sqrt{\tilde{D}}))}{p^l}\Big\rangle
\end{align*}
since $a^l\gamma_1-\langle a^l\gamma_1\rangle\in\mathbb{Z}$. So Lemma \ref{lem5.5} is proved.
\end{proof}

Furthermore, we need the following property about
$\Delta_{\bm{\alpha,\beta}}(x,a,\epsilon)$. In what follows,
we define the set $Q_{\bm{\alpha,\beta}}$ by
\begin{align} \label{eq5.0.1}
Q_{\bm{\alpha,\beta}}:=\{(\alpha_{11}, \tilde{\alpha}_{21}),...,(\alpha_{1r}, \tilde{\alpha}_{2r}),
(\beta_{11}, \tilde{\beta}_{21}),...,(\beta_{1s}, \tilde{\beta}_{2s})\}.
\end{align}

\begin{lem} \label{lem5.6}
With $S_{\bm{\alpha,\beta}}$ being defined as in (\ref{Eq1.9}),
let $\epsilon\in (0,1)\setminus S_{\bm{\alpha,\beta}}$ and $a\in H$.
If $\Delta_{\bm{\alpha,\beta}}(x,a,\epsilon)\ge 0$ for all $x\in\mathbb{R}$, then
$\Delta_{\bm{\alpha,\beta}}(x,b,\epsilon)\ge 0$ for all $x\in\mathbb{R}$ and
positive integers $b$ with $b\equiv a \mod E$.
\end{lem}
\begin{proof}
Let $(\xi_1,\tilde\xi_2)$ and $(\eta_1,\tilde\eta_2)$ be any two
given elements of the set $Q_{\bm{\alpha,\beta}}$.
Since $b\equiv a \mod E$ and $d_{\bm{\alpha_1,\beta_1}} \mid E$, one has
$b\equiv a\mod d_{\bm{\alpha_1,\beta_1}}$, and so $(b-a)\xi_1$
and $(b-a)\eta_1$ are all rational integers. Then
\begin{align} \label{eq5.6.1}
\langle a\xi_1+\tilde{\xi}_2\epsilon\rangle
=\langle b\xi_1+\tilde{\xi}_2\epsilon\rangle
\ \ {\rm and}\ \
\langle a\eta_1+\tilde{\eta}_2\epsilon\rangle
=\langle b\eta_1+\tilde{\eta}_2\epsilon\rangle.
\end{align}

First of all, we claim that
$a\xi_1+\tilde{\xi}_2\epsilon\preccurlyeq a\eta_1+\tilde{\eta}_2\epsilon$
holds if and only if
$b\xi_1+\tilde{\xi}_2\epsilon\preccurlyeq b\eta_1+\tilde{\eta}_2\epsilon$
holds. To prove this claim, we consider the following two cases:

{\sc Case 1}. $\tilde{\xi}_2=\tilde{\eta}_2$. By (\ref{eq5.6.1}), we know that
$\langle a\xi_1+\tilde{\xi}_2\epsilon\rangle<
\langle a\eta_1+\tilde{\eta}_2\epsilon\rangle$
holds if and only if
$\langle b\xi_1+\tilde{\xi}_2\epsilon\rangle<
\langle b\eta_1+\tilde{\eta}_2\epsilon\rangle$ holds.
Since $\tilde{\xi}_2=\tilde{\eta}_2$, by (\ref{eq5.6.1}) one can derive that
both of $\langle a\xi_1+\tilde{\xi}_2\epsilon\rangle=
\langle a\eta_1+\tilde{\eta}_2\epsilon\rangle$ and
$a\xi_1+\tilde{\xi}_2\epsilon\ge a\eta_1+\tilde{\eta}_2\epsilon$
are true if and only if both of
$\langle b\xi_1+\tilde{\xi}_2\epsilon\rangle=
\langle b\eta_1+\tilde{\eta}_2\epsilon\rangle$ and
$b\xi_1+\tilde{\xi}_2\epsilon\ge b\eta_1+\tilde{\eta}_2\epsilon$
are true. Hence by definition of $\preccurlyeq$, we can deduce that
$a\xi_1+\tilde{\xi}_2\epsilon\preccurlyeq a\eta_1+\tilde{\eta}_2\epsilon$
holds if and only if either $\langle a\xi_1+\tilde{\xi}_2\epsilon\rangle<
\langle a\eta_1+\tilde{\eta}_2\epsilon\rangle$, or
$\langle a\xi_1+\tilde{\xi}_2\epsilon\rangle=
\langle a\eta_1+\tilde{\eta}_2\epsilon\rangle$ and
$a\xi_1+\tilde{\xi}_2\epsilon\ge a\eta_1+\tilde{\eta}_2\epsilon$,
if and only if either $\langle b\xi_1+\tilde{\xi}_2\epsilon\rangle<
\langle b\eta_1+\tilde{\eta}_2\epsilon\rangle$, or
$\langle b\xi_1+\tilde{\xi}_2\epsilon\rangle=
\langle b\eta_1+\tilde{\eta}_2\epsilon\rangle$ and
$b\xi_1+\tilde{\xi}_2\epsilon\ge b\eta_1+\tilde{\eta}_2\epsilon$,
if and only if
$b\xi_1+\tilde{\xi}_2\epsilon\preccurlyeq b\eta_1+\tilde{\eta}_2\epsilon$
holds. So the claim is true in this case.

{\sc Case 2}. $\tilde{\xi}_2\neq\tilde{\eta}_2$. This together with
$\epsilon\not\in S_{\bm{\alpha,\beta}}$ gives us that
$\langle a\xi_1+\tilde{\xi}_2\epsilon\rangle\neq
\langle a\eta_1+\tilde{\eta}_2\epsilon\rangle$.
Therefore (\ref{eq5.6.1}) implies that
$\langle b\xi_1+\tilde{\xi}_2\epsilon\rangle\neq
\langle b\eta_1+\tilde{\eta}_2\epsilon\rangle$.
Since $\langle a\xi_1+\tilde{\xi}_2\epsilon\rangle\neq
\langle a\eta_1+\tilde{\eta}_2\epsilon\rangle$,
one then deduces that
$a\xi_1+\tilde{\xi}_2\epsilon\preccurlyeq a\eta_1+\tilde{\eta}_2\epsilon$
holds if and only if
$\langle a\xi_1+\tilde{\xi}_2\epsilon\rangle<
\langle a\eta_1+\tilde{\eta}_2\epsilon\rangle$ holds,
by (\ref{eq5.6.1}), if and only if
$\langle b\xi_1+\tilde{\xi}_2\epsilon\rangle<
\langle b\eta_1+\tilde{\eta}_2\epsilon\rangle$ holds,
if and only if
$b\xi_1+\tilde{\xi}_2\epsilon\preccurlyeq b\eta_1+\tilde{\eta}_2\epsilon$
holds since by (\ref{eq5.6.1}), one has $\langle b\xi_1+\tilde{\xi}_2\epsilon\rangle\neq
\langle b\eta_1+\tilde{\eta}_2\epsilon\rangle$.
So the claim is proved in this case. It concludes the proof of the claim.

Consequently, we come back to the proof of Lemma \ref{lem5.6}.
Let $\Delta_{\bm{\alpha,\beta}}(x,a,\epsilon)\ge 0$ hold
for all $x\in\mathbb{R}$. By Lemma \ref{lem2.5}, we have
$\Delta_{\bm{\alpha,\beta}}(a\beta_{1k}+\tilde{\beta}_{2k}\epsilon, a,\epsilon)\ge 0$
for all integers $k$ with $1\le k\le s$.
On the other hand, one can deduce from the claim that
$$
\Delta_{\bm{\alpha,\beta}}(b\beta_{1k}+\tilde{\beta}_{2k}\epsilon, b,\epsilon)=
\Delta_{\bm{\alpha,\beta}}(a\beta_{1k}+\tilde{\beta}_{2k}\epsilon, a,\epsilon)\ge 0
$$
holds for all integers $k$ with $1\le k\le s$. Again by Lemma \ref{lem2.5},
we have that $\Delta_{\bm{\alpha,\beta}}(x, b,\epsilon)\ge 0$
for all $x\in\mathbb{R}$. The proof of Lemma \ref{lem5.6} is finished.
\end{proof}

Define
$$M_1:=2\max_{1\le i\le r \atop 1\le j\le s} \{|\alpha_{1i}|, |\beta_{1j}|\} \ \ {\rm and}\ \
M_2:=2\max_{1\le i\le r \atop 1\le j\le s} \{|\tilde{\alpha}_{2i}|, |\tilde{\beta}_{2j}|\}.$$
Then we have the following result.

\begin{lem} \label{lem5.7}
Let $\epsilon\in(0,1)\setminus S_{\bm{\alpha,\beta}}$. Let
\begin{align} \label{eq5.7.1}
\delta_1:=\min_{(\xi_1,\tilde\xi_2)\in Q_{\bm{\alpha,\beta}}}\{\langle a\xi_1
+\tilde\xi_2\epsilon \rangle, 1-\langle a\xi_1+\tilde\xi_2\epsilon \rangle \}\setminus \{0\},
\end{align}
\begin{align} \label{eq5.7.2}
\delta_2:=\min_{(\xi_1,\tilde\xi_2),(\eta_1,\tilde\eta_2)\in Q_{\bm{\alpha,\beta}}}
\{ |\langle a\xi_1+\tilde\xi_2\epsilon \rangle-\langle a\eta_1
+\tilde\eta_2\epsilon \rangle|\}\setminus \{0\},
\end{align}
and $\delta_{\epsilon}:=\min(\delta_1,\delta_2)/M_2$.
Then for any $\epsilon'\in (\epsilon-\delta_{\epsilon},\epsilon
+\delta_{\epsilon})\setminus S_{\bm{\alpha,\beta}}$
and all integers $k$ with $1\le k\le s$, we have
$\Delta_{\bm{\alpha,\beta}}(a\beta_{1k}+\tilde{\beta}_{2k}\epsilon, a, \epsilon)
=\Delta_{\bm{\alpha,\beta}}(a\beta_{1k}+\tilde{\beta}_{2k}\epsilon', a, \epsilon').$
\end{lem}

\begin{proof}
Let $(\xi_1,\tilde\xi_2)\in Q_{\bm{\alpha,\beta}}$.
Pick $\epsilon'\in (\epsilon-\delta_{\epsilon},\epsilon+\delta_{\epsilon})$.
Then
\begin{equation} \label{eq5.7.3}
|\tilde\xi_2(\epsilon'-\epsilon)|<\delta_{\epsilon}M_2/2
=\min(\delta_1,\delta_2)/2\le \frac{1}{2}<1.
\end{equation}
It follows that
\begin{align} \label{eq5.7.4}
\langle \tilde\xi_2(\epsilon'-\epsilon) \rangle
=\bigg\{\begin{array}{cl}
       \tilde\xi_2(\epsilon'-\epsilon), & {\rm if}\
       \tilde\xi_2(\epsilon'-\epsilon)>0, \\
       1+\tilde\xi_2(\epsilon'-\epsilon), & {\rm otherwise.}
     \end{array}
\end{align}

In what follows, we show that the following inequality holds:
\begin{equation} \label{eq5.7.5}
\langle a\xi_1+\tilde\xi_2\epsilon\rangle-\delta_2/2<\langle a\xi_1+\tilde\xi_2
\epsilon'\rangle<\langle a\xi_1+\tilde\xi_2\epsilon\rangle+\delta_2/2
\end{equation}

First, let $\langle a\xi_1+\tilde\xi_2\epsilon\rangle=1$.
Since $\epsilon\not\in S_{\bm{\alpha,\beta}}$, one has $\tilde\xi_2=0$, and so
$\langle a\xi_1+\tilde\xi_2\epsilon'\rangle=\langle a\xi_1+\tilde\xi_2\epsilon\rangle=1$.
Therefore (\ref{eq5.7.5}) is true if $\langle a\xi_1+\tilde\xi_2\epsilon\rangle=1$.

Second, let $\langle a\xi_1+\tilde\xi_2\epsilon\rangle<1$. Then
$\langle a\xi_1+\tilde\xi_2\epsilon\rangle\le 1-\delta_1$.
If $\tilde\xi_2(\epsilon'-\epsilon)>0$, then by (\ref{eq5.7.4})
and (\ref{eq5.7.3}), we have
$\langle \tilde\xi_2(\epsilon' -\epsilon) \rangle
=\tilde\xi_2(\epsilon'-\epsilon)<\min(\delta_1,\delta_2)/2$.
So one can deduce that
\begin{align*}
&\langle a\xi_1+\tilde\xi_2\epsilon\rangle+
\langle \tilde\xi_2(\epsilon' -\epsilon) \rangle \\
<& 1-\delta_1+\min(\delta_1,\delta_2)/2\\
<& 1-\delta_1+\delta_1/2=1-\delta_1/2<1.
\end{align*}
Hence we can use Lemma \ref{lem5.4} (i) to get that
\begin{align*} 
\langle a\xi_1+\tilde\xi_2\epsilon'\rangle
=&\langle a\xi_1+\tilde\xi_2\epsilon+\tilde\xi_2(\epsilon'-\epsilon)\rangle \\
=&\langle a\xi_1+\tilde\xi_2\epsilon\rangle
 +\langle\tilde\xi_2(\epsilon'-\epsilon)\rangle \notag \\
<&\langle a\xi_1+\tilde\xi_2\epsilon\rangle+\min(\delta_1,\delta_2)/2 \notag\\
\le & \langle a\xi_1+\tilde\xi_2\epsilon\rangle+\delta_2/2 \notag
\end{align*}
and
\begin{align*}
\langle a\xi_1+\tilde\xi_2\epsilon'\rangle
=&\langle a\xi_1+\tilde\xi_2\epsilon\rangle
 +\langle\tilde\xi_2(\epsilon'-\epsilon)\rangle  \\
>&\langle a\xi_1+\tilde\xi_2\epsilon\rangle \\
> & \langle a\xi_1+\tilde\xi_2\epsilon\rangle-\delta_2/2.
\end{align*}
Hence (\ref{eq5.7.5}) holds if $\tilde\xi_2(\epsilon'-\epsilon)>0$.
Let now $\tilde\xi_2(\epsilon'-\epsilon)\le 0$.
Then by (\ref{eq5.7.3}), we have
$\tilde\xi_2(\epsilon'-\epsilon)>-\min(\delta_1,\delta_2)/2$.
Hence applying (\ref{eq5.7.4}) gives us that
$$\langle\tilde \xi_2(\epsilon'-\epsilon)\rangle
=1+\tilde\xi_2(\epsilon'-\epsilon)> 1-\min(\delta_1,\delta_2)/2.$$
So by (\ref{eq5.7.1}), one has
$$\langle a\xi_1+\tilde\xi_2\epsilon\rangle+\langle\tilde\xi_2(\epsilon'
-\epsilon)\rangle>\delta_1+1-\min(\delta_1,\delta_2)/2>1.$$
It then follows from Lemma \ref{lem5.4} (i) that
\begin{align*} 
\langle a\xi_1+\tilde\xi_2\epsilon'\rangle
=&\langle a\xi_1+\tilde\xi_2\epsilon+\tilde\xi_2(\epsilon'-\epsilon)\rangle \\
=&\langle a\xi_1+\tilde\xi_2\epsilon\rangle
 +\langle\tilde\xi_2(\epsilon'-\epsilon)\rangle-1 \notag \\
=&\langle a\xi_1+\tilde\xi_2\epsilon\rangle+\tilde\xi_2(\epsilon'-\epsilon) \notag\\
> & \langle a\xi_1+\tilde\xi_2\epsilon\rangle-\min(\delta_1,\delta_2)/2 \notag \\
\ge & \langle a\xi_1+\tilde\xi_2\epsilon\rangle-\delta_2/2 \notag
\end{align*}
and since $\tilde\xi_2(\epsilon'-\epsilon)\le 0$, one has
$$
\langle a\xi_1+\tilde\xi_2\epsilon'\rangle
=\langle a\xi_1+\tilde\xi_2\epsilon\rangle+\tilde\xi_2(\epsilon'-\epsilon)
\le  \langle a\xi_1+\tilde\xi_2\epsilon\rangle
<  \langle a\xi_1+\tilde\xi_2\epsilon\rangle+\delta_2/2.
$$
Thereby (\ref{eq5.7.5}) holds if $\tilde\xi_2(\epsilon'-\epsilon)\le 0$.
The proof of (\ref{eq5.7.5}) is complete.

Now pick $(\eta_1,\tilde\eta_2)\in Q_{\bm{\alpha,\beta}}$. Then by
the arbitrariness of $(\xi_1,\tilde\xi_2)\in Q_{\bm{\alpha,\beta}}$,
one can deduce from (\ref{eq5.7.5}) that
\begin{equation} \label{eq5.7.6}
\langle a\eta_1+\tilde\eta_2\epsilon\rangle-\delta_2/2<\langle a\eta_1+\tilde\eta_2
\epsilon'\rangle<\langle a\eta_1+\tilde\eta_2\epsilon\rangle+\delta_2/2.
\end{equation}
In the following, we show that
\begin{equation} \label{eq5.7.9}
a\xi_1+\tilde\xi_2\epsilon\preccurlyeq a\eta_1+\tilde\eta_2\epsilon \Leftrightarrow
a\xi_1+\tilde\xi_2\epsilon'\preccurlyeq a\eta_1+\tilde\eta_2\epsilon'.
\end{equation}

If $\langle a\xi_1+\tilde\xi_2\epsilon\rangle<\langle a\eta_1+\tilde\eta_2\epsilon\rangle$,
then
$\langle a\eta_1+\tilde\eta_2\epsilon\rangle-\langle a\xi_1+\tilde\xi_2\epsilon\rangle\ge \delta_2$
that implies that
$\langle a\xi_1+\tilde\xi_2\epsilon\rangle\le \langle a\eta_1+\tilde\eta_2\epsilon\rangle-\delta_2$.
So by (\ref{eq5.7.5}) and (\ref{eq5.7.6}), one deduces that
$$\langle a\xi_1+\tilde\xi_2\epsilon'\rangle
<\langle a\xi_1+\tilde\xi_2\epsilon\rangle+\delta_2/2
\le \langle a\eta_1+\tilde\eta_2\epsilon\rangle-\delta_2/2
< \langle a\eta_1+\tilde\eta_2\epsilon'\rangle.$$

Conversely, if
$\langle a\xi_1+\tilde\xi_2\epsilon'\rangle < \langle a\eta_1+\tilde\eta_2\epsilon'\rangle$,
then by (\ref{eq5.7.5}) and (\ref{eq5.7.6}) we have
$$\langle a\xi_1+\tilde\xi_2\epsilon\rangle
<\langle a\xi_1+\tilde\xi_2\epsilon'\rangle+\delta_2/2
< \langle a\eta_1+\tilde\eta_2\epsilon'\rangle+\delta_2/2
< \langle a\eta_1+\tilde\eta_2\epsilon\rangle+\delta_2.$$
So if $\langle a\xi_1+\tilde\xi_2\epsilon\rangle\ge \langle a\eta_1+\tilde\eta_2\epsilon\rangle$,
then
$$0\le \langle a\xi_1+\tilde\xi_2\epsilon\rangle-\langle a\eta_1+\tilde\eta_2\epsilon\rangle<\delta_2,$$
which contradicts with the definition of $\delta_2$ (see (\ref{eq5.7.2})).
Hence we must have $\langle a\xi_1+\tilde\xi_2\epsilon\rangle
<\langle a\eta_1+\tilde\eta_2\epsilon\rangle$.
So
$\langle a\xi_1+\tilde\xi_2\epsilon\rangle<\langle a\eta_1+\tilde\eta_2\epsilon\rangle$
holds if and only if
$\langle a\xi_1+\tilde\xi_2\epsilon'\rangle<\langle a\eta_1+\tilde\eta_2\epsilon'\rangle$
holds.

If $\langle a\xi_1+\tilde\xi_2\epsilon\rangle=\langle a\eta_1+\tilde\eta_2\epsilon\rangle$
and $a\xi_1+\tilde\xi_2\epsilon\ge a\eta_1+\tilde\eta_2\epsilon$,
then $\tilde\xi_2=\tilde\eta_2$ since $\epsilon\not\in S_{\bm{\alpha,\beta}}$,
and so $a\xi_1\ge a\eta_1$.
It then follows that
$\langle a\xi_1+\tilde\xi_2\epsilon'\rangle=\langle a\eta_1+\tilde\eta_2\epsilon'\rangle$
and $a\xi_1+\tilde\xi_2\epsilon'\ge a\eta_1+\tilde\eta_2\epsilon'$.
Likewise, we can show the converse direction. Hence we conclude that
both of $\langle a\xi_1+\tilde\xi_2\epsilon\rangle=\langle a\eta_1+\tilde\eta_2\epsilon\rangle$
and $a\xi_1+\tilde\xi_2\epsilon\ge a\eta_1+\tilde\eta_2\epsilon$ are true if and only if
both of $\langle a\xi_1+\tilde\xi_2\epsilon'\rangle=\langle a\eta_1+\tilde\eta_2\epsilon'\rangle$
and $a\xi_1+\tilde\xi_2\epsilon'\ge a\eta_1+\tilde\eta_2\epsilon'$
are true. Thus (\ref{eq5.7.9}) is proved.

Finally, by using (\ref{eq5.7.9}), it follows immediately
that for any integer $k$ with $1\le k\le s$, one has
\begin{align*}
&\Delta_{\bm{\alpha,\beta}}(a\beta_{1k}+\tilde{\beta}_{2k}\epsilon, a, \epsilon) \\
=&\#\{i: a\alpha_{1i}+\tilde\alpha_{2i}\epsilon\preccurlyeq a\beta_{1k}+\tilde{\beta}_{2k}\epsilon\}
-\#\{j: a\beta_{1j}+\tilde\beta_{2j}\epsilon\preccurlyeq a\beta_{1k}+\tilde{\beta}_{2k}\epsilon\} \\
=&\#\{i: a\alpha_{1i}+\tilde\alpha_{2i}\epsilon'\preccurlyeq a\beta_{1k}+\tilde{\beta}_{2k}\epsilon'\}
-\#\{j: a\beta_{1j}+\tilde\beta_{2j}\epsilon'\preccurlyeq a\beta_{1k}+\tilde{\beta}_{2k}\epsilon'\}\\
=&\Delta_{\bm{\alpha,\beta}}(a\beta_{1k}+\tilde{\beta}_{2k}\epsilon', a, \epsilon')
\end{align*}
as one desires. This ends the proof of Lemma \ref{lem5.7}.
\end{proof}

\begin{lem} \label{lem5.8}
Let $\xi=\xi_1+\tilde{\xi}_2\sqrt{\tilde{D}}$ and $\eta=\eta_1+\tilde{\eta}_2\sqrt{\tilde{D}}$
be the elements of the set $\Psi=\{\alpha_1,...,\alpha_r,\beta_1,...,\beta_s\}$ with
$\xi_1, \eta_1\in\mathbb{Q}$ and $\tilde\xi_2, \tilde\eta_2\in\mathbb{Z}$.
Let $a\in H$ and $p$ be a prime number such that $ap\equiv 1\mod E$ and
\begin{align} \label{eq5.8.1}
p>4d_{\bm{\alpha_2,\beta_2}}^2 d_{\bm{\alpha_1,\beta_1}}^2
(d_{\bm{\alpha_2,\beta_2}}^2 (M_1+d_{\bm{\alpha_1,\beta_1}})^2+|D|M_2^2).
\end{align}
Let $\sigma$ be a monomorphism from $K$ to $\mathbb{C}_p$.
Then for any positive integer $l$, the following two inequalities are equivalent:
\begin{equation} \label{eq5.8.2}
\Big\langle a^l\xi_1-\frac{\xi_1}{p^l}
+\frac{\tilde{\xi}_2T_{p,l}(\sigma(\sqrt{\tilde{D}}))}{p^l}\Big\rangle
\le \Big\langle a^l\eta_1-\frac{\eta_1}{p^l}
+\frac{\tilde{\eta}_2T_{p,l}(\sigma(\sqrt{\tilde{D}}))}{p^l}\Big\rangle
\end{equation}
and
\begin{equation} \label{eq5.8.3}
a^l\xi_1+\frac{\tilde{\xi}_2T_{p,l}(\sigma(\sqrt{\tilde{D}}))}{p^l} \preccurlyeq
a^l\eta_1+\frac{\tilde{\eta}_2T_{p,l}(\sigma(\sqrt{\tilde{D}}))}{p^l}.
\end{equation}
\end{lem}

\begin{proof}
Let $\gamma=\gamma_1+\tilde{\gamma}_2\sqrt{\tilde{D}}\in\Psi$
with $\gamma_1\in\mathbb{Q}$ and $\tilde\gamma_2\in\mathbb{Z}$.
Since $d(\tilde D)\mid d_{\bm{\alpha_2, \beta_2}}^2$,
$d(\gamma_1)\mid d_{\bm{\alpha_1, \beta_1}}$
and $M_2>|\tilde\gamma_2|$, by (\ref{eq5.8.1})
one then derives that for any positive integer $l$, one has
$$p^l>4d(\tilde{D})d(\gamma_1)^2(d(\tilde{D})
(M_1+d(\gamma_1))^2+d(\tilde{D})|\tilde{D}|\tilde{\gamma}_2^2).$$
Claim that
\begin{equation} \label{eq5.8.4}
\Big\langle a^l\gamma_1-\frac{\gamma_1}{p^l}+\frac{\tilde{\gamma}_2T_{p,l}
(\sigma(\sqrt{\tilde{D}}))}{p^l}\Big\rangle
=\Big\langle a^l\gamma_1+\frac{\tilde{\gamma}_2T_{p,l}
(\sigma(\sqrt{\tilde{D}}))}{p^l}\Big\rangle-\frac{\gamma_1}{p^l}.
\end{equation}
that will be shown in what follows.

First let $\tilde{\gamma}_2\neq 0$. Then it follows from Lemma \ref{lem4.3} that
$$\frac{M_1}{p^l}<\Big\{a^l\gamma_1
+\frac{\tilde{\gamma}_2T_{p,l}(\sigma(\sqrt{\tilde{D}}))}
{p^l}\Big\}<1-\frac{M_1}{p^l}.$$
Then
$$a \gamma_1+\frac{\tilde{\gamma}_2T_{p,l}(\sigma(\sqrt{\tilde{D}}))}{p^l}$$
is not a rational integer and so we have
$$\frac{M_1}{p^l}<\Big\langle a^l\gamma_1+\frac{\tilde{\gamma}_2T_{p,l}(\sigma(\sqrt{\tilde{D}}))}
{p^l}\Big\rangle<1-\frac{M_1}{p^l}.$$
Since
$$\Big|\frac{\gamma_1}{p^l}\Big|<\frac{M_1}{p^l},$$
one then deduces that
$$0<\Big\langle a^l\gamma_1+\frac{\tilde{\gamma}_2T_{p,l}(\sigma(\sqrt{\tilde{D}}))}
{p^l}\Big\rangle-\frac{\gamma_1}{p^l}<1.$$
Obviously,
$$a^l\gamma_1-\frac{\gamma_1}{p^l}+\frac{\tilde{\gamma}_2T_{p,l}(\sigma(\sqrt{\tilde{D}}))}{p^l}
-\Big(\Big\langle a^l\gamma_1+\frac{\tilde{\gamma}_2T_{p,l}(\sigma(\sqrt{\tilde{D}}))}{p^l}
\Big\rangle-\frac{\gamma_1}{p^l}\Big)\in\mathbb{Z}.$$
Hence by the definition of $\langle \cdot \rangle$,
we know that (\ref{eq5.8.4}) is true if $\tilde{\gamma}_2\neq 0$.

Let now $\tilde{\gamma}_2=0$. Then $\gamma_1=\gamma$. If $\gamma_1$ is an integer,
then $\langle a^l\gamma_1 \rangle=1$, and by the assumption that
$\gamma\in \Psi \subset K\setminus\mathbb{Z}_{\le 0}$
one can deduce that $\gamma_1$ is a positive integer. Since $0<\gamma_1<M_1<p\le p^l$,
one has $0<\frac{\gamma_1}{p^l}<1$ and so
$$\Big\langle a^l\gamma_1-\frac{\gamma_1}{p^l}\Big\rangle
=\Big\langle -\frac{\gamma_1}{p^l}\Big\rangle
=1-\Big\langle \frac{\gamma_1}{p^l}\Big\rangle
=1-\frac{\gamma_1}{p^l}
=\langle a^l\gamma_1\rangle-\frac{\gamma_1}{p^l}$$
as (\ref{eq5.8.4}) claimed. Now let $\gamma_1$ be non-integer.
Since $a\in H$, we have $a$ is coprime to $E$.
But $d(\gamma_1)\mid d_{\bm{\alpha_1,\beta_1}}\mid E$. Then $a^l$
is coprime to both of $d(\gamma_1)$ and $d_{\bm{\alpha_1,\beta_1}}$.
Hence $a^l\gamma_1$ is non-integer and so
$\langle a^l\gamma_1 \rangle=i/d_{\bm{\alpha_1,\beta_1}}$ for some integer
$i$ with $1\le i<d_{\bm{\alpha_1,\beta_1}}$.
By (\ref{eq5.8.1}), one has
$$p^l>d_{\bm{\alpha_1,\beta_1}}M_1>d_{\bm{\alpha_1,\beta_1}}|\gamma_1|.$$
It follows that $\frac{|\gamma_1|}{p^l}<1/d_{\bm{\alpha_1,\beta_1}}$.
Thus we have
$$0<\frac{i}{d_{\bm{\alpha_1,\beta_1}}}-\frac{\gamma_1}{p^l}=\langle a^l\gamma_1
\rangle-\frac{\gamma_1}{p^l}<1$$
for $1\le i<d_{\bm{\alpha_1,\beta_1}}.$
This together with the fact
$$a^l\gamma_1-\frac{\gamma_1}{p^l}-\Big(\langle a^l\gamma_1 \rangle-\frac{\gamma_1}{p^l}\Big)\in\mathbb{Z}$$
gives us that
$$\Big\langle a^l\gamma_1-\frac{\gamma_1}{p^l} \Big\rangle=\langle a^l\gamma_1 \rangle-\frac{\gamma_1}{p^l}.$$
So (\ref{eq5.8.4}) is true when $\tilde{\gamma_2}=0$. It concludes the proof of the claim (\ref{eq5.8.4}).
Now the claim applied to $\xi$ and $\eta$ gives us that
\begin{equation} \label{eq5.8.5}
\Big\langle a^l\xi_1-\frac{\xi_1}{p^l}+\frac{\tilde{\xi}_2T_{p,l}
(\sigma(\sqrt{\tilde{D}}))}{p^l}\Big\rangle=\Big\langle a^l\xi_1
+\frac{\tilde{\xi}_2T_{p,l}(\sigma(\sqrt{\tilde{D}}))}{p^l}
\Big\rangle-\frac{\xi_1}{p^l}
\end{equation}
and
\begin{equation} \label{eq5.8.6}
\Big\langle a^l\eta_1-\frac{\eta_1}{p^l}+\frac{\tilde{\eta}_2T_{p,l}
(\sigma(\sqrt{\tilde{D}}))}{p^l}\Big\rangle=\Big\langle a^l\eta_1
+\frac{\tilde{\eta}_2T_{p,l}(\sigma(\sqrt{\tilde{D}}))}{p^l}
\Big\rangle-\frac{\eta_1}{p^l}.
\end{equation}
Then it follows from (\ref{eq5.8.5}) and (\ref{eq5.8.6}) that (\ref{eq5.8.2})
holds if and only if the following inequality is true:
\begin{equation} \label{eq5.8.7}
\Big\langle a^l\xi_{1}+\frac{\tilde{\xi}_{2}T_{p,l}
(\sigma(\sqrt{\tilde{D}}))}{p^l}\Big\rangle-\frac{\xi_{1}}{p^l}
\le \Big\langle a^l\eta_{1}+\frac{\tilde{\eta}_{2}T_{p,l}
(\sigma(\sqrt{\tilde{D}}))}{p^l}\Big\rangle-\frac{\eta_{1}}{p^l}.
\end{equation}

On the other hand, since $d(\tilde D)\mid d_{\bm{\alpha_2, \beta_2}}^2$,
$d(a^l\xi_1-a^l\eta_1)\mid d_{\bm{\alpha_1, \beta_1}}$ and $M_2>|\tilde{\xi}_2-\tilde{\eta}_2|$,
by (\ref{eq5.8.1}), one has
$$p^l\ge p>4d(\tilde{D})d^2(a^l\xi_1-a^l\eta_1)(d(\tilde{D})(M_1+d(a^l\xi_1-a^l\eta_1))^2
+d(\tilde{D})|\tilde{D}|(\tilde{\xi}_2-\tilde{\eta}_2)^2).$$
Let $\tilde\xi_2\neq\tilde\eta_2$. Then one can deduce from Lemma \ref{lem4.3} that
\begin{equation} \label{eq5.8.8}
\frac{M_1}{p^l}< \Big\langle a^l\xi_{1}+\frac{\tilde{\xi}_{2}T_{p,l}(\sigma(\sqrt{\tilde{D}}))}{p^l}
-a^l\eta_{1}-\frac{\tilde{\eta}_{2}T_{p,l}(\sigma(\sqrt{\tilde{D}}))}{p^l}\Big\rangle
<1-\frac{M_1}{p^l}.
\end{equation}
Then
\begin{align} \label{eq5.8.9}
a^l\xi_{1}+\frac{\tilde{\xi}_{2}T_{p,l}(\sigma(\sqrt{\tilde{D}}))}{p^l}
-a^l\eta_{1}-\frac{\tilde{\eta}_{2}T_{p,l}(\sigma(\sqrt{\tilde{D}}))}{p^l}\not\in\mathbb{Z}.
\end{align}
But from Lemma \ref{lem5.4} (ii), one can derive the following fact:
For any real numbers $X$ and $Y$, $\langle X\rangle=\langle Y\rangle$
if and only if $X-Y\in\mathbb{Z}$.
Therefore by this fact and (\ref{eq5.8.9}), we obtain that
$$\Big\langle a^l\xi_{1}+\frac{\tilde{\xi}_{2}T_{p,l}
(\sigma(\sqrt{\tilde{D}}))}{p^l}\Big\rangle
\ne \Big\langle a^l\eta_{1}+\frac{\tilde{\eta}_{2}T_{p,l}
(\sigma(\sqrt{\tilde{D}}))}{p^l}\Big\rangle.$$
Equivalently, if one has
$$\Big\langle a^l\xi_{1}+\frac{\tilde{\xi}_{2}T_{p,l}
(\sigma(\sqrt{\tilde{D}}))}{p^l}\Big\rangle
=\Big\langle a^l\eta_{1}+\frac{\tilde{\eta}_{2}T_{p,l}
(\sigma(\sqrt{\tilde{D}}))}{p^l}\Big\rangle,$$
then $\tilde\xi_2=\tilde\eta_2$. Hence by the definition of $\preccurlyeq$,
(\ref{eq5.8.3}) is true if and only if one has either
\begin{equation} \label{eq5.8.10}
\Big\langle a^l\xi_{1}+\frac{\tilde{\xi}_{2}T_{p,l}
(\sigma(\sqrt{\tilde{D}}))}{p^l}\Big\rangle
< \Big\langle a^l\eta_{1}+\frac{\tilde{\eta}_{2}T_{p,l}
(\sigma(\sqrt{\tilde{D}}))}{p^l}\Big\rangle,
\end{equation}
or
\begin{equation} \label{eq5.8.11}
\Big\langle a^l\xi_{1}+\frac{\tilde{\xi}_{2}T_{p,l}
(\sigma(\sqrt{\tilde{D}}))}{p^l}\Big\rangle
= \Big\langle a^l\eta_{1}+\frac{\tilde{\eta}_{2}T_{p,l}
(\sigma(\sqrt{\tilde{D}}))}{p^l}\Big\rangle\ \ {\rm and}\ \ \xi_1\ge \eta_1.
\end{equation}

Now to finish the proof of Lemma \ref{lem5.8}, we need only
to show that the truth of (\ref{eq5.8.7}) is equivalent to the truth of
one of (\ref{eq5.8.10}) and (\ref{eq5.8.11}).

Let us first show that the truth of (\ref{eq5.8.7}) implies the truth of
one of (\ref{eq5.8.10}) and (\ref{eq5.8.11}). Assume that (\ref{eq5.8.7}) is true.
Claim that
\begin{equation} \label{eq5.8.12}
\Big\langle a^l\xi_{1}+\frac{\tilde{\xi}_{2}T_{p,l}
(\sigma(\sqrt{\tilde{D}}))}{p^l}\Big\rangle
\le\Big\langle a^l\eta_{1}+\frac{\tilde{\eta}_{2}T_{p,l}
(\sigma(\sqrt{\tilde{D}}))}{p^l}\Big\rangle.
\end{equation}
Suppose that (\ref{eq5.8.12}) is not true. Then
$$\Big\langle a^l\xi_{1}+\frac{\tilde{\xi}_{2}T_{p,l}(\sigma(\sqrt{\tilde{D}}))}{p^l}\Big\rangle
> \Big\langle a^l\eta_{1}+\frac{\tilde{\eta}_{2}T_{p,l}(\sigma(\sqrt{\tilde{D}}))}{p^l}\Big\rangle.$$
Then by Lemma \ref{lem5.4} (ii), we have
\begin{align*}
&\Big\langle a^l\xi_{1}+\frac{\tilde{\xi}_{2}T_{p,l}
(\sigma(\sqrt{\tilde{D}}))}{p^l}\Big\rangle-\frac{\xi_{1}}{p^l}
-\Big\langle a^l\eta_{1}+\frac{\tilde{\eta}_{2}T_{p,l}
(\sigma(\sqrt{\tilde{D}}))}{p^l}\Big\rangle+\frac{\eta_{1}}{p^l} \\
=&\Big\langle a^l\xi_{1}+\frac{\tilde{\xi}_{2}T_{p,l}
(\sigma(\sqrt{\tilde{D}}))}{p^l}-a^l\eta_{1}-\frac{\tilde{\eta}_{2}T_{p,l}
(\sigma(\sqrt{\tilde{D}}))}{p^l}\Big\rangle-\frac{\xi_{1}}{p^l}
+\frac{\eta_{1}}{p^l}.
\end{align*}

If $\tilde{\xi_2}\neq\tilde{\eta}_2$, then by (\ref{eq5.8.8}) and
$$\Big|\frac{\xi_{1}}{p^l}-\frac{\eta_{1}}{p^l}\Big|<\frac{M_1}{p^l},$$
one can derive that
$$
\Big\langle a^l\xi_{1}+\frac{\tilde{\xi}_{2}T_{p,l}
(\sigma(\sqrt{\tilde{D}}))}{p^l}-a^l\eta_{1}-\frac{\tilde{\eta}_{2}T_{p,l}
(\sigma(\sqrt{\tilde{D}}))}{p^l}\Big\rangle-\frac{\xi_{1}}{p^l}
+\frac{\eta_{1}}{p^l}>0.
$$

If $\tilde{\xi_2}=\tilde{\eta}_2$, then by $p^l\ge p>d_{\bm{\alpha_1,\beta_1}}M_1$, one has
$$\langle a^l\xi_1-a^l\eta_1\rangle \ge \frac{1} {d(a^l\xi_1-a^l\eta_1)}
\ge \frac{1}{d_{\bm{\alpha_1,\beta_1}}}>\frac{M_1}{p^l}>\Big|\frac{\xi_{1}}{p^l}-\frac{\eta_{1}}{p^l}\Big|.$$
Therefore
$$\langle a^l\xi_1-a^l\eta_1\rangle-\frac{\xi_{1}}{p^l}+\frac{\eta_{1}}{p^l}>0.$$
Hence we conclude that
$$\Big\langle a^l\xi_{1}+\frac{\tilde{\xi}_{2}T_{p,l}
(\sigma(\sqrt{\tilde{D}}))}{p^l}\Big\rangle-\frac{\xi_{1}}{p^l}
-\Big\langle a^l\eta_{1}+\frac{\tilde{\eta}_{2}T_{p,l}
(\sigma(\sqrt{\tilde{D}}))}{p^l}\Big\rangle+\frac{\eta_{1}}{p^l}> 0,$$
which contradicts with (\ref{eq5.8.7}). Therefore (\ref{eq5.8.12}) is true.
The claim is proved.

Now by the claim (\ref{eq5.8.12}), we know that either (\ref{eq5.8.10}) is true,
or the equality in (\ref{eq5.8.11}) holds.
For the latter case, one can then deduce from (\ref{eq5.8.7}) that
$-\frac{\xi_{1}}{p^l}\le -\frac{\eta_{1}}{p^l}$, that is, $\xi_1\ge \eta_1$. Hence
(\ref{eq5.8.11}) is true. So the necessity part is proved.

Now we show the sufficiency part. Let one of (\ref{eq5.8.10}) and (\ref{eq5.8.11}) hold.
In the following we show that (\ref{eq5.8.7}) is true.
If (\ref{eq5.8.11}) holds, then (\ref{eq5.8.7}) is clear true. Now let (\ref{eq5.8.10}) hold.
Then by Lemma \ref{lem5.4} (ii), one has
\begin{align} \label{eq5.8.13}
&\Big\langle a^l\eta_{1}+\frac{\tilde{\eta}_{2}T_{p,l}
(\sigma(\sqrt{\tilde{D}}))}{p^l}\Big\rangle-\frac{\eta_{1}}{p^l}
-\Big\langle a^l\xi_{1}+\frac{\tilde{\xi}_{2}T_{p,l}
(\sigma(\sqrt{\tilde{D}}))}{p^l}\Big\rangle+\frac{\xi_{1}}{p^l} \\
=&\Big\langle a^l\eta_{1}+\frac{\tilde{\eta}_{2}T_{p,l}(\sigma(\sqrt{\tilde{D}}))}{p^l}
-a^l\xi_{1}-\frac{\tilde{\xi}_{2}T_{p,l}(\sigma(\sqrt{\tilde{D}}))}{p^l}\Big\rangle
+\frac{\xi_{1}}{p^l}-\frac{\eta_{1}}{p^l}. \notag
\end{align}
Likewise, we can show that the right-hand side of (\ref{eq5.8.13}) is positive.
Namely, (\ref{eq5.8.7}) holds. This finishes the proof of sufficiency part.

The proof of Lemma \ref{lem5.8} is complete.
\end{proof}

In the following lemmas, we will use the following notations:
For the integers $i$ and $j$ with $1\le i\le r$ and $1\le j\le s$, we define
\begin{equation} \label{eq5.0.2}
y_i:=a^l\alpha_{1i}+\frac{\tilde{\alpha}_{2i}T_{p,l}(\sigma(\sqrt{\tilde{D}}))}{p^l},\ \
w_i:=\Big\langle a^l\alpha_{1i}-\frac{\alpha_{1i}}{p^l}+\frac
{\tilde{\alpha}_{2i}T_{p,l}(\sigma(\sqrt{\tilde{D}}))}{p^l}\Big\rangle,
\end{equation}
\begin{equation} \label{eq5.0.3}
x_j:=a^l\beta_{1j}+\frac{\tilde{\beta}_{2j}T_{p,l}(\sigma(\sqrt{\tilde{D}}))}{p^l}, \ \
z_j:=\Big\langle a^l\beta_{1j}-\frac{\beta_{1j}}{p^l}+\frac{\tilde
{\beta}_{2j}T_{p,l}(\sigma(\sqrt{\tilde{D}}))}{p^l}\Big\rangle.
\end{equation}

\begin{lem} \label{lem5.9}
Let $l$ be any positive integers. Let $a\in H$ and $p$ be a
prime number satisfying (\ref{eq5.8.1}) and $ap\equiv 1\mod E$.
Let $\sigma$ be a monomorphism from $K$ to $\mathbb{C}_p$.
The following statements are equivalent:

{\rm (i).} $\#\{i:y_i\preccurlyeq x_k \}\ge \#\{j: x_j\preccurlyeq x_k \}$.

{\rm (ii).} For any integer $n$ with $1\le n\le p^l$, one has
$$\sum_{i=1}^{r}\Big\lceil\frac{n}{p^l}-w_i\Big\rceil \ge
\sum_{j=1}^{s}\Big\lceil\frac{n}{p^l}-z_j\Big\rceil.$$
\end{lem}
\begin{proof}

On the one hand, since Lemma \ref{lem5.8} tells us that
$y_i\preccurlyeq x_k \Leftrightarrow w_i\le z_k$
and $x_j\preccurlyeq x_k \Leftrightarrow z_j\le z_k$, it then follows that
\begin{align} \label{eq5.9.1}
\#\{i:y_i\preccurlyeq x_k \}-\#\{j: x_j\preccurlyeq x_k\}=\#\{i:w_i\le z_k\}-\#\{j:z_j\le z_k\}.
\end{align}

On the other hand, let $\gamma=\gamma_1+\tilde{\gamma}_2\sqrt{\tilde{D}}\in
\Psi $
with $\gamma_1\in\mathbb{Q}$ and $\tilde\gamma_2\in\mathbb{Z}$.
Since $p$ satisfies (\ref{eq5.8.1}), we have
\begin{align} \label{eq5.9.2}
p^l\ge &p>4d_{\bm{\alpha_1,\beta_1}}(M_1+1)^2 \\
>& d_{\bm{\alpha_1,\beta_1}}(4M_1+4)
> d(\gamma_1)(|\gamma_1|+3) \notag \\
\ge & d(\gamma_1)(|\lfloor 1-\gamma_1\rfloor|+\langle \gamma_1\rangle+1), \notag
\end{align}
where the last inequality of (\ref{eq5.9.2}) is induced by (\ref{eq2.0.5}), and
\begin{align}  \label{eq5.9.3}
p^l\ge & p>d_{\bm{\alpha_1,\beta_1}}^2d_{\bm{\alpha_2,\beta_2}}^2(M_1^2+|D|M_2^2) \\
> & d_{\bm{\alpha_1,\beta_1}}^2d_{\bm{\alpha_2,\beta_2}}^2
(\gamma_1^2+|\tilde{D}|\tilde\gamma_2^2) \notag \\
\ge & d(\gamma_1)^2d(\tilde{D})|\gamma_1^2-\tilde{D}\tilde\gamma_2^2|>0  \notag
\end{align}
since the hypothesis that $D$ is square-free implies that
$\gamma_1^2-\tilde{D}\tilde\gamma_2^2\ne 0$.
Clearly, $d(\gamma_1)^2d(\tilde{D})(\gamma_1^2-\tilde{D}\tilde\gamma_2^2)$ is a
rational integer. Then by (\ref{eq5.9.3}), we have
$$v_p(d(\gamma_1)^2d(\tilde{D})^2(\gamma_1^2-\tilde{D}\tilde\gamma_2^2))=0.$$
Moreover, one has $v_p(d(\gamma_1)^2d(\tilde{D})^2)=0$ since $(p,E)=1$ and
both of $d(\gamma_1)$ and $d(\tilde D)$ divide $E$.
It follows that
\begin{align} \label{eq5.9.4}
v_p(\gamma_1^2-\tilde{D}\tilde\gamma_2^2)=0.
\end{align}
Then by Lemma 5.4, this together with (\ref{eq5.9.2}) tells us that
$$w_i=\frac{T_{p,l}(\sigma(\alpha_{1i}+\tilde{\alpha}_{2i}\sqrt{\tilde{D}}))}{p^l}
\ {\rm and}\
z_j=\frac{T_{p,l}(\sigma(\beta_{1j}+\tilde{\beta}_{2j}\sqrt{\tilde{D}}))}{p^l}.$$
So $w_i$ and $z_j$ are both in the set $\{0,\frac{1}{p^l},...,\frac{p^l-1}{p^l}\}$.
Let $n$ be an integer such that $1\le n\le p^l$. Then
$$-1<\frac{n}{p^l}-w_i, \frac{n}{p^l}-z_j<1.$$
Thus
\begin{align*}
\Big\lceil\frac{n}{p^l}-\gamma\Big\rceil
=\bigg\{\begin{array}{cl}
       1, & {\rm if}\ \gamma\le \frac{n-1}{p^l}, \\
       0, & {\rm otherwise},
     \end{array}
\end{align*}
where $\gamma\in \{w_1,...,w_r,z_1,...,z_s\}$. It then follows that
\begin{align}  \label{eq5.9.5}
\sum_{i=1}^{r}\Big\lceil\frac{n}{p^l}-w_i\Big\rceil
-\sum_{j=1}^{s}\Big\lceil\frac{n}{p^l}-z_j\Big\rceil
=\#\Big\{i:w_i\le \frac{n-1}{p^l}\Big\}-\#\Big\{j:z_j\le \frac{n-1}{p^l}\Big\}.
\end{align}

Applied Lemma \ref{lem2.5} to the set $\{0,1,...,\frac{p^l-1}{p^l}\}$
with the order $\le $, we can obtain that
\begin{align}
\#\Big\{i:w_i\le \frac{n-1}{p^l}\Big\}-\#\Big\{j:z_j\le \frac{n-1}{p^l}\Big\}\ge 0
\end{align}
holds for all integer $n$ with $1\le n\le p^l$ if and only if
\begin{align} \label{eq5.9.6}
\#\{i:w_i\le z_k\}-\#\{j:z_j\le z_k\}\ge 0
\end{align}
holds for all integers $k$ with $1\le k\le s$.
Then the equivalence of parts (i) and (ii) follows immediately from
(\ref{eq5.9.5}) to (\ref{eq5.9.6}) and (\ref{eq5.9.1}).
This ends the proof of Lemma \ref{lem5.9}.
\end{proof}

\begin{lem} \label{lem5.10}
Let $\Delta_{\bm{\alpha}, \bm{\beta}}(x,a,\epsilon)\ge 0$
hold for all $x\in\mathbb{R}$, all $a\in H$ and all
$\epsilon\in(0,1)\setminus S_{\bm{\alpha,\beta}}$.
Let $p$ be a prime numbers such that (\ref{eq5.8.1})
holds and that the inverse
of $p$ modulo $E$ is belong to $H$. Then $r\ge s$ and
$\sigma_p(F_{\bm{\alpha,\beta}}(z))\in\mathcal{O}_p[[z]]$
for all monomorphisms $\sigma_p: K\rightarrow\mathbb{C}_p$.
\end{lem}

\begin{proof}
Take an $a\in H$ such that $ap\equiv 1\mod E$.
First we show that $r\ge s$. For any fixed
$\epsilon\in (0,1)\setminus S_{\bm{\alpha,\beta}}$,
let $N$ be a positive integer such that
$$N\ge \max_{1\le i\le r \atop 1\le j\le s}\{a\alpha_{1i}+\tilde\alpha_{2i}\epsilon,
a\beta_{1j}+\tilde\beta_{2j}\epsilon\}.$$
If $a\alpha_{1i}+\tilde\alpha_{2i}\epsilon$ is not an rational integer, then
$\langle a\alpha_{1i}+\tilde\alpha_{2i}\epsilon\rangle<1=\langle N\rangle$,
and so $a\alpha_{1i}+\tilde\alpha_{2i}\epsilon\preccurlyeq N$.
If $a\alpha_{1i}+\tilde\alpha_{2i}\epsilon$ is an integer, then the equalities
$\langle a\alpha_{1i}+\tilde\alpha_{2i}\epsilon\rangle=1=\langle N\rangle$
together with $N\ge a\alpha_{1i}+\tilde\alpha_{2i}\epsilon$ tell us that
$a\alpha_{1i}+\tilde\alpha_{2i}\epsilon\preccurlyeq N$.
Hence we conclude that $a\alpha_{1i}+\tilde\alpha_{2i}\epsilon\preccurlyeq N$
for all integers $i$ with $1\le i\le r$. Likewise, we can show that
$a\beta_{1j}+\tilde\beta_{2j}\epsilon\preccurlyeq N$
for all integer $j$ with $1\le j\le s$. Therefore
$$\Delta_{\bm{\alpha}, \bm{\beta}}(N,a,\epsilon)=r-s\ge 0.$$
So $r\ge s$ as one desires.

Let $\sigma_p$ be any given monomorphism
from $K$ to $\mathbb{C}_p$. Then $\sigma_p(K)\subset \mathbb{Q}_p$. Let
$\gamma=\gamma_1+\tilde{\gamma}_2\sqrt{\tilde{D}}\in\Psi $
with $\gamma_1\in\mathbb{Q}$ and $\tilde\gamma_2\in\mathbb{Z}$.
Since $(p,E)=1$ and $d(\gamma_1)\mid E$, $\gamma_1$ is a $p$-adic integer.
Note that $v_p(\tilde{D})=0$ since $\tilde D=D/d_{\bm{\alpha_2,\beta_2}}^2$,
$(p,E)=1$ and both of $d_{\bm{\alpha_2,\beta_2}}$ and $D$ divide $E$.
It follows that
$$v_p(\sigma_p(\sqrt{\tilde{D}}))=\frac{1}{2}v_p(\tilde{D})=0.$$
So $\sigma_p(\sqrt{\tilde{D}})$ is a $p$-adic unit.
One derives that $\sigma_p(\gamma)$ is a $p$-adic integer.
Hence $T_{p,l}(\sigma_p(\gamma))$ is well defined for all positive integers $l$ and
$\gamma\in\Psi $.
By (\ref{eq5.9.4}), one has $v_p(\gamma_1^2-\tilde{D}\tilde\gamma_2^2)=0$. As in
the proof of Lemma \ref{lem5.9}, we know that (\ref{eq5.9.2}) is true. Then
it follows from Lemma \ref{lem5.5} that
$$\frac{T_{p,l}(\sigma_p(\gamma))}{p^l}=\Big\langle a^l\gamma_1-\frac{\gamma_1}{p^l}
+\frac{\tilde{\gamma}_2T_{p,l}(\sigma_p(\sqrt{\tilde{D}}))}{p^l}\Big\rangle.$$
So by Lemma \ref{lem2.1} one deduces that
\begin{align} \label{eq5.10.1}
&v_p\Big(\frac{(\sigma_p(\alpha_1))_n\cdots(\sigma_p(\alpha_r))_n}
{(\sigma_p(\beta_1))_n\cdots(\sigma_p(\beta_s))_n}\Big) \\
=&\sum_{i=1}^{r}\sum_{l=1}^{\infty}\Big\lceil\frac{n-T_{p,l}(\sigma_p(\alpha_i))}{p^l}\Big\rceil
-\sum_{j=1}^{s}\sum_{l=1}^{\infty}\Big\lceil\frac{n-T_{p,l}(\sigma_p(\beta_j))}{p^l}\Big\rceil \notag \\
=&\sum_{i=1}^{r}\sum_{l=1}^{\infty}\Big\lceil\frac{n}{p^l}
-\Big\langle a^l\alpha_{1i}-\frac{\alpha_{1i}}{p^l}+\frac{\tilde{\alpha}_{2i}
T_{p,l}(\sigma_p(\sqrt{\tilde{D}}))}{p^l}\Big\rangle\Big\rceil \notag \\
&-\sum_{j=1}^{s}\sum_{l=1}^{\infty}\Big\lceil\frac{n}{p^l}
-\Big\langle a^l\beta_{1j}-\frac{\beta_{1j}}{p^l}+\frac{\tilde{\beta}_{2j}T_{p,l}
(\sigma_p(\sqrt{\tilde{D}}))}{p^l}\Big\rangle\Big\rceil. \notag
\end{align}

For a fixed positive integer $l$, let
$$\epsilon_l:=\frac{T_{p,l}(\sigma_p(\sqrt{\tilde{D}}))}{p^l}.$$
Since $v_p(\sigma_p(\sqrt{\tilde{D}}))=0$, one has
$$T_{p,l}(\sigma_p(\sqrt{\tilde{D}}))\equiv -\sigma_p(\sqrt{\tilde{D}})\not\equiv 0\mod p.$$
Then $\epsilon_l=\frac{i}{p^l}$ for some integer $i$ with $i\in [1, p^l-1]$ and $i$ coprime to $p$.
By the definition of $S_{\bm{\alpha,\beta}}$, the exact denominator of every element of
$S_{\bm{\alpha,\beta}}$ is not exceeding $M_2$. Since $M_2<p^l$ by (\ref{eq5.8.1}),
one has $\epsilon_l\not\in S_{\bm{\alpha,\beta}}$.
So we conclude that $\epsilon_l\in (0,1)\setminus S_{\bm{\alpha,\beta}}$.
Since Lemma \ref{lem5.1} tells us that $H$ is a group, there is
a $b\in H$ such that $a^l\equiv b \mod E$. Then
$\Delta_{\bm{\alpha,\beta}}(x, b, \epsilon_l)\ge 0$ holds for all $x\in \mathbb{R}$.
So it follows from Lemma \ref{lem5.6} that
\begin{align} \label{eq5.10.2}
\Delta_{\bm{\alpha,\beta}}(x, a^l,\epsilon_l)\ge 0\ \forall x\in \mathbb{R}.
\end{align}
For all integers $i$ and $j$ with $1\le i\le r$ and $1\le j\le s$,
let $y_i, w_i$ and $x_j, z_j$ be given as in (\ref{eq5.0.2}) and (\ref{eq5.0.3}),
respectively. Then by (\ref{eq5.10.2}) we have
$$\Delta_{\bm{\alpha,\beta}}(x_k, a^l,\epsilon_l)\ge 0\ \ \forall\ 1\le k\le s.$$
That is, for all integers $k$ with $1\le k\le s$, we have
$$\#\{i:y_i\preccurlyeq x_k\}-\#\{j:x_j\preccurlyeq x_k\}\ge 0.$$
Thus by Lemma \ref{lem5.9}, we have for all integer $n$ with $1\le n\le p^l$ that
$$
\sum_{i=1}^{r}\Big\lceil\frac{n}{p^l}-w_i\Big\rceil
-\sum_{j=1}^{s}\Big\lceil\frac{n}{p^l}-z_j\Big\rceil\ge 0.
$$
Then for all positive integers $n$, one has
\begin{align} \label{eq5.10.3}
&\sum_{i=1}^{r}\Big\lceil\frac{n}{p^l}-w_i\Big\rceil
-\sum_{j=1}^{s}\Big\lceil\frac{n}{p^l}-z_j\Big\rceil \\
=&\sum_{i=1}^{r}\Big\lceil \frac{n}{p^l}-\Big\langle\frac{n}{p^l}\Big\rangle
+\Big\langle\frac{n}{p^l}\Big\rangle-w_i\Big\rceil
-\sum_{j=1}^{s}\Big\lceil \frac{n}{p^l}-\Big\langle\frac{n}{p^l}\Big\rangle
+\Big\langle\frac{n}{p^l}\Big\rangle-z_j\Big\rceil \notag \\
=&(r-s)\Big(\frac{n}{p^l}-\Big\langle\frac{n}{p^l}\Big\rangle\Big)
+\sum_{i=1}^{r}\Big\lceil\Big\langle\frac{n}{p^l}\Big\rangle-w_i\Big\rceil
-\sum_{j=1}^{s}\Big\lceil\Big\langle\frac{n}{p^l}\Big\rangle-z_j\Big\rceil\ge 0 \notag
\end{align}
since $r\ge s$. Therefore by (\ref{eq5.10.1}) and
(\ref{eq5.10.3}), we deduce that
\begin{align*}
& v_p\Big(\frac{(\sigma_p(\alpha_1))_n\cdots(\sigma_p(\alpha_r))_n}
{(\sigma_p(\beta_1))_n\cdots(\sigma_p(\beta_s))_n}\Big)
= \sum_{l=1}^\infty \Big(\sum_{i=1}^{r}\Big\lceil\frac{n}{p^l}-w_i\Big\rceil
-\sum_{j=1}^{s}\Big\lceil\frac{n}{p^l}-z_j\Big\rceil\Big) \ge 0
\end{align*}
for any positive integer $n$. Hence
$\sigma_p(F_{\bm{\alpha,\beta}}(z))\in\mathcal{O}_p[[z]]$ for
all monomorphism $\sigma_p\in{\rm Hom}(K, \mathbb{C}_p)$.
So Lemma \ref{lem5.10} is proved.
\end{proof}

Now we are in the position to prove Theorem \ref{thm1.4}.  \\

{\it Proof of Theorem \ref{thm1.4}.}
At first, we prove part (i). Let $u<v$. Let $a$ be an arbitrary element of
$I=G\setminus H$ and $p$ be any prime number satisfying (\ref{eq5.3.1}),
(\ref{eq5.8.1}) and $ap\equiv 1\mod E$.
Then by Lemma \ref{lem5.3}, one has that $\mathfrak{p}=p\mathcal{O}_K$
is the unique prime ideal in $K$ that divides $p$ and
$F_{\bm{\alpha,\beta}}(z)\not\in\mathcal{O}_{K, \mathfrak{p}}[[z]]$.
So one can deduce that there are infinitely many prime ideals
$\mathfrak{p}$ in $K$ such that
$F_{\bm{\alpha,\beta}}(z)\not\in\mathcal{O}_{K, \mathfrak{p}}[[z]]$.
It then follows from the equivalence of parts (i) and (ii)
of Theorem \ref{thm1.3} that $F_{\bm{\alpha,\beta}}(z)$
is not N-integral. Hence part (i) is true.

Consequently, we show the sufficiency parts of parts (ii) and (iii).
Let the statements I and IV be true if $u>v$, and the
statements II, III and IV be true if $u=v$.
We show that $F_{\bm{\alpha,\beta}}(z)$ is N-integral.
Let $p$ be any prime number satisfying both of
(\ref{eq5.3.1}) and (\ref{eq5.8.1}),
and let ${\rm Hom}(K,\mathbb{C}_p)$ be the set of
all monomorphisms from $K$ to $\mathbb{C}$.

If $ap \equiv 1\mod E$ for some $a\in H$, since the
statement IV holds, it then follows from Lemma \ref{lem5.10} that
$\sigma(F_{\bm{\alpha,\beta}}(z))\in\mathcal{O}_p[[z]]$
for all monomorphisms $\sigma_p\in {\rm Hom}(K,\mathbb{C}_p)$.

If $ap \equiv 1\mod E$ for some $a\in I$, then $\mathfrak{p}:=p\mathcal{O}_K$
is the unique prime ideal in $K$ such that $\mathfrak{p}\mid p$.
So for any $\sigma_p\in{\rm Hom}(K,\mathbb{C}_p)$, there exists a
monomorphism $\widehat\sigma_p: K_{\mathfrak{p}}\rightarrow \mathbb{C}_p$ such that
$\widehat\sigma_p|_K=\sigma_p$ and $v_{\mathfrak{p}}=v_p\circ\widehat\sigma_p$.
Thus
\begin{equation} \label{eq5.11.1}
\sigma_p(F_{\bm{\alpha,\beta}}(z))=\widehat\sigma_p(F_{\bm{\alpha,\beta}}(z))
\end{equation}
and $v_p(\widehat\sigma_p(\gamma))=v_{\mathfrak{p}}(\gamma)\ge 0$
for all $\gamma\in\mathcal{O}_{K,\mathfrak{p}}$.
So if $\gamma\in\mathcal{O}_{K,\mathfrak{p}}$, then
$\widehat\sigma_p(\gamma)$ belongs to the valuation ring
$\mathcal{O}_p$ of $\mathbb{C}_p$. Hence
\begin{equation} \label{eq5.11.2}
\widehat\sigma_p(\mathcal{O}_{K,\mathfrak{p}})\subset\mathcal{O}_p.
\end{equation}

First let $u>v$. Since the statement I is true, by Lemma \ref{lem5.3}
we have $F_{\bm{\alpha,\beta}}(z)\in \mathcal{O}_{K,\mathfrak{p}}[[z]]$.
It then follows from (\ref{eq5.11.1}) and (\ref{eq5.11.2}) that for all
$\sigma_p\in{\rm Hom}(K,\mathbb{C}_p)$, one has
\begin{equation} \label{eq5.11.3}
\sigma_p(F_{\bm{\alpha,\beta}}(z))=\widehat\sigma_p(F_{\bm{\alpha,\beta}}(z))\in
\widehat\sigma_p(\mathcal{O}_{K,\mathfrak{p}})[[z]]\subset \mathcal{O}_p[[z]].
\end{equation}

Now let $u=v$. Since at this moment, the statements II and III are true,
one can deduce from Lemma \ref{lem5.3} that
$F_{\bm{\alpha,\beta}}(z)\in \mathcal{O}_{K,\mathfrak{p}}[[z]]$.
Again by (\ref{eq5.11.1}) and (\ref{eq5.11.2}), we know that for all
$\sigma_p\in{\rm Hom}(K,\mathbb{C}_p)$, (\ref{eq5.11.3}) still holds.

In conclusion, we know that $\sigma_p(F_{\bm{\alpha,\beta}}(z))\in \mathcal{O}_p[[z]]$
for all prime numbers $p$ satisfying both of
(\ref{eq5.3.1}) and (\ref{eq5.8.1}) and all $\sigma\in{\rm Hom}(K,\mathbb{C}_p)$.
Hence by the equivalence of parts (i) and (iii) of Theorem 1.2, we know that
$F_{\bm{\alpha,\beta}}(z)$ is N-integral in $K$.
This ends the proof of the sufficiencies of parts (ii) and (iii).

Finally, it remains to show the necessity parts of parts (ii) and (iii)
which will be done in the following. Let $F_{\bm{\alpha,\beta}}(z)$
be N-integral in $K$. Then by the equivalence of parts (i) and (ii)
of Theorem \ref{thm1.3}, there is a positive integer $P$ such that
for all prime ideals $\mathfrak{p}$ of $K$ dividing a prime number
$p>P$, we have $F_{\bm{\alpha,\beta}}(z)\in\mathcal{O}_{K,\mathfrak{p}}[[z]]$.

For any $a\in I$, let $p$ be a prime number satisfying
(\ref{eq5.3.1}), (\ref{eq5.8.1}), $p>P$ and $ap\equiv 1\mod E$.
Then $\mathfrak{p}=p\mathcal{O}_K$ is the unique prime ideal of $K$ dividing $p$
and $F_{\bm{\alpha,\beta}}(z)\in\mathcal{O}_{K,\mathfrak{p}}[[z]]$.
If $u>v$, then one can deduce from Lemma \ref{lem5.3} that
$\delta_{\bm{\mu,\nu}}(x,a)\ge 0$ for all $x\in\mathbb{R}$ and $a\in I$.
Thus the statement I is true.

If $u=v$, then one can derive from Lemma \ref{lem5.3} that
$$\sum_{l=1}^{h}\delta_{\bm{\mu,\nu}}(a^l\beta_k,a^l)\ge 0$$
for all $a\in I$, $h\in\{1,...,{\rm ord}(a)\}$ and $k\in\{1,...,v\}$, and
$$\delta_{\bm{\bm{\mu,\nu}}}\Big(a^l\Big(\frac{e}
{d_{\bm{\bm{\mu,\nu}}}}+m_{\bm{\mu,\nu}}\Big),a^l\Big)\ge 0$$
for all $a\in I$, $e\in\{1,...,d_{\bm{\alpha,\beta}}\}$ and
$l\in\{1,...,{\rm ord}(a)\}$.
Thus the statements II and III are true.
So to finish the proof of the necessity parts of parts (ii) and (iii),
it is enough to show that if $F_{\bm{\alpha,\beta}}(z)$ is N-integral in $K$,
then the statement IV is true for the cases that $u>v$ and $u=v$.
This will be done in what follows.

Since $F_{\bm{\alpha,\beta}}(z)$ is N-integral, by the equivalence of
parts (i) and (iii) of Theorem \ref{thm1.3}, there is a positive
integer $P'$ such that
$\sigma_p(F_{\bm{\alpha,\beta}}(z))\in\mathcal{O}_p[[z]]$
for all prime numbers $p$ with $p>P'$ and all monomorphisms
$\sigma_p\in{\rm Hom}(K,\mathbb{C}_p)$.
Let $a\in H$. For all prime numbers $p$ satisfying (\ref{eq5.8.1}), $p>P'$ and
$ap\equiv 1\mod E$, let $\sigma_p$ be a monomorphism from $K$ to $\mathbb{C}_p$.
Since $\sigma_p(F_{\bm{\alpha,\beta}}(z))\in \mathcal{O}_p[[z]]$, by (\ref{eq5.10.1})
one can deduce for all positive integers $n$ that
\begin{align} \label{eq5.11.4}
0\le &v_p\Big(\frac{(\sigma_p(\alpha_1))_n\cdots(\sigma_p(\alpha_r))_n}
{(\sigma_p(\beta_1))_n\cdots(\sigma_p(\beta_s))_n}\Big) \\
=&\sum_{l=1}^{\infty}\Big(\sum_{i=1}^{r}\Big\lceil\frac{n-T_{p,l}(\sigma_p(\alpha_i))}{p^l}\Big\rceil
-\sum_{j=1}^{s}\Big\lceil\frac{n-T_{p,l}(\sigma_p(\beta_j))}{p^l}\Big\rceil\Big). \notag
\end{align}

Let $\gamma=\gamma_1+\tilde{\gamma}_2\sqrt{\tilde{D}}
\in \Psi$ with $\Psi$ being defined as in (\ref{eq2.0.4}),
where $\gamma_1\in\mathbb{Q}$ and $\tilde\gamma_2\in\mathbb{Z}$.
Let $\mathcal S$ denote the arithmetic progression $\{n>0: an\equiv 1\mod E\}$.
Define
$$\mathcal{P}:=\bigcup_{\gamma\in\Psi} \mathcal{P}_{\gamma},$$
where
$$\mathcal{P}_{\gamma}:=\big\{p\in \mathcal S: \exists\ \sigma_p\in {\rm Hom}(K,\mathbb{C}_p)
\ {\rm s.t.}\ T_p(\sigma_p(\gamma))=T_{p,2}(\sigma_p(\gamma))\big\}.$$
Let $p$ be any prime number in $\mathcal{S}\setminus \mathcal{P}$
and $\sigma_p$ be any monomorphism in ${\rm Hom}(K,\mathbb{C}_p)$.
Then (\ref{Eq2.2'}) applied to $\sigma_p(\gamma)$ gives us that
for all integers $l\ge 2$ and all $\gamma\in\Psi$, one has
$$T_p(\sigma_p(\gamma))<T_{p,2}(\sigma_p(\gamma))\le T_{p,l}(\sigma_p(\gamma))$$
since $T_p(\sigma_p(\gamma))\ne T_{p,2}(\sigma_p(\gamma))$. But
$$T_{p,2}(\sigma_p(\gamma))=T_p(\sigma_p(\gamma))+tp$$
with $t\in \mathbb{Z}$. So $t\ge 1$ and $T_{p,2}(\sigma_p(\gamma))\ge p$.
It then follows that
$$T_{p,l}(\sigma_p(\gamma))\ge p$$
holds for all $l\ge 2$. Hence for all integers $n, l$
with $1\le n\le p$ and $l\ge 2$ and all $\gamma\in \Psi$, one has
$$-1<\frac{n-T_{p,l}(\sigma_p(\gamma)}{p^l}\le 0.$$
Thus for all integers $n, l$ with $1\le n\le p$ and $l\ge 2$ and all
$\gamma\in \Psi$, one has
$$\Big\lceil\frac{n-T_{p,l}(\sigma_p(\gamma))}{p^l}\Big\rceil=0.$$
Hence we have for all integers $l$ with $l\ge 2$ that
\begin{equation} \label{eq5.11.5}
\sum_{i=1}^{r}\Big\lceil\frac{n-T_{p,l}(\sigma_p(\alpha_i))}{p^l}\Big\rceil
-\sum_{j=1}^{s}\Big\lceil \frac{n-T_{p,l}(\sigma_p(\beta_j))}{p^l}\Big\rceil=0.
\end{equation}
Then by (\ref{eq5.11.4}), (\ref{eq5.11.5}) and using Lemma \ref{lem5.5},
it follows that for all prime numbers $p$ with $p>P'$ and
$p\in \mathcal{S}\setminus\mathcal{P}$, all $\sigma_p\in{\rm Hom}(K,\mathbb{C}_p)$
and all integers $n$ with $1\le n\le p$, one has

\begin{align*}
0\le &v_p\Big(\frac{(\sigma_p(\alpha_1))_n\cdots(\sigma_p(\alpha_r))_n}
{(\sigma_p(\beta_1))_n\cdots(\sigma_p(\beta_s))_n}\Big) \\
=&\sum_{i=1}^{r}\Big\lceil\frac{n-T_p(\sigma_p(\alpha_i))}{p}\Big\rceil
-\sum_{j=1}^{s}\Big\lceil\frac{n-T_p(\sigma_p(\beta_j))}{p}\Big\rceil \\
=&\sum_{i=1}^{r}\Big\lceil\frac{n}{p}-\Big\langle a\alpha_{1i}-\frac{\alpha_{1i}}{p}
+\frac{\tilde{\alpha}_{2i}T_{p}(\sigma_p(\sqrt{\tilde{D}}))}{p}\Big\rangle\Big\rceil\\
&-\sum_{j=1}^{s}\Big\lceil\frac{n}{p}-\Big\langle a\beta_{1j}-\frac{\beta_{1j}}{p}
+\frac{\tilde{\beta}_{2j}T_{p}(\sigma_p(\sqrt{\tilde{D}}))}{p}\Big\rangle\Big\rceil.
\end{align*}
So from Lemma \ref{lem5.9} one deduces that
\begin{align*}
&\#\Big\{i:a\alpha_{1i}+\frac{\tilde{\alpha}_{2i}T_{p}(\sigma_p(\sqrt{\tilde{D}}))}{p}
\preccurlyeq a\beta_{1k}+\frac{\tilde{\beta}_{2k}T_{p}(\sigma_p(\sqrt{\tilde{D}}))}{p}\Big\} \\
\ge&\#\Big\{j: a\beta_{1j}+\frac{\tilde{\beta}_{2j}T_{p}(\sigma_p(\sqrt{\tilde{D}}))}{p}
\preccurlyeq a\beta_{1k}+\frac{\tilde{\beta}_{2k}T_{p}(\sigma_p(\sqrt{\tilde{D}}))}{p}\Big\}
\end{align*}
holds for all integers $k$ with $1\le k\le s$.
That is, for all prime numbers $p$ with $p>P'$ and $p\in \mathcal{S}\setminus\mathcal{P}$,
all $\sigma_p\in{\rm Hom}(K,\mathbb{C}_p)$ and all integers $k$ with $1\le k\le s$, one has
\begin{equation} \label{eq5.11.6}
\Delta_{\bm{\alpha,\beta}}\Big(a\beta_{1k}+\frac{\tilde{\beta}_{2k}T_p(\sigma_p(\sqrt{\tilde{D}}))}{p},
a,\frac{T_p(\sigma_p(\sqrt{\tilde{D}}))}{p}\Big)\ge 0.
\end{equation}

Now let $\epsilon$ be any real number in the set $(0,1)\setminus S_{\bm{\alpha,\beta}}$.
Let $\delta_1$ and $\delta_2$ be defined as in (\ref{eq5.7.1}) and (\ref{eq5.7.2}),
respectively, and let $\delta_{\epsilon}:=\min(\delta_1,\delta_2)/M_2$.
Since $p\nmid d(\tilde{D})$ for all prime numbers $p\in\mathcal{S}$ and
$d(\tilde{D})z^2-d(\tilde{D})\tilde{D}$ is a primitive polynomial with integer coefficient,
by Lemma \ref{lem5.2}, there is a positive real number $c$ depended only on $\sqrt{\tilde{D}}$
and $\mathcal{S}$ such that
$$\Big\{\frac{T_p(\sigma_p(\sqrt{\tilde D}))}{p}\in (\epsilon-\delta_{\epsilon},\epsilon
+\delta_{\epsilon}): p\in\mathcal{S}, p<x, \sigma_p\in {\rm Hom}(K, \mathbb{C}_p)\Big\}
\sim 2c\delta_{\epsilon}\pi(x).$$
Hence
\begin{align} \label{eq5.11.7}
&\#\Big\{\frac{T_p(\sigma_p(\sqrt{\tilde D}))}{p}\in (\epsilon-\delta_{\epsilon},
\epsilon+\delta_{\epsilon}): p\in\mathcal{S}\setminus \mathcal{P},
p<x, \sigma_p\in {\rm Hom}(K, \mathbb{C}_p)\Big\} \\
\ge &\#\Big\{\frac{T_p(\sigma_p(\sqrt{\tilde D}))}{p}\in (\epsilon-\delta_{\epsilon},
\epsilon+\delta_{\epsilon}): p\in\mathcal{S},p<x,
\sigma_p\in {\rm Hom}(K, \mathbb{C}_p)\Big\} \notag \\
&-\#\Big\{\frac{T_p(\sigma_p(\sqrt{\tilde D}))}{p}: p\in \mathcal{P},
p<x, \sigma_p\in {\rm Hom}(K, \mathbb{C}_p)\Big\} \notag \\
= & 2c\delta_{\epsilon}\pi(x)-2\cdot\#\{p\in\mathcal{P}: p<x\}+o(\pi(x)). \notag
\end{align}

We claim that for any given $\gamma\in\Psi$, one has
\begin{equation} \label{eq5.11.8}
\#\{p\in \mathcal{P}_{\gamma}: p<x \}=o(\pi(x)).
\end{equation}
Before proving this claim, we show that the truth of
(\ref{eq5.11.8}) implies the truth of Statement IV.
Actually, if (\ref{eq5.11.8}) is true, then for all large
enough positive real numbers $x$, one has
$$\#\{p\in\mathcal{P}: p<x\}
\le \sum_{\gamma\in\Psi}\#\{p\in\mathcal{P}_{\gamma}: p<x\}=o(\pi(x)).$$
So it follows from (\ref{eq5.11.7}) that for all large enough
positive real numbers $x$, we have
\begin{align*}
&\#\Big\{\frac{T_p(\sigma_p(\sqrt{\tilde D}))}{p}\in (\epsilon-\delta_{\epsilon},
\epsilon+\delta_{\epsilon}): p\in\mathcal{S}\setminus \mathcal{P},
p<x, \sigma_p\in {\rm Hom}(K, \mathbb{C}_p)\Big\} \\
\ge & 2c\delta_{\epsilon}\pi(x)-o_1(\pi(x))+o_2(\pi (x))\\
= & 2c\delta_{\epsilon}\pi(x)+o(\pi(x))>0.
\end{align*}
Hence there exists a prime number $p$ with $p>P'$,
$p\in\mathcal{S}\setminus \mathcal{P}$
and $\sigma_p\in {\rm Hom}(K, \mathbb{C}_p)$ such that
$$\frac{T_p(\sigma_p(\sqrt{\tilde D}))}{p}\in
(\epsilon-\delta_{\epsilon},\epsilon+\delta_{\epsilon}).$$
Then it follows from Lemma \ref{lem5.7} and (\ref{eq5.11.6}) that
$$\Delta_{\bm{\alpha,\beta}}(a\beta_{1k}+\tilde{\beta}_{2k}\epsilon, a,\epsilon)\ge 0$$
holds for all integers $k$ with $1\le k\le s$.
Therefore from Lemma \ref{lem2.5}, one deduces that
$\Delta_{\bm{\alpha,\beta}}(x,a,\epsilon)\ge 0$ holds for all $x\in\mathbb{R}$.
Namely, the statement IV is true under the assumption that the claim (\ref{eq5.11.8})
is true. Now it remains to show that the claim (\ref{eq5.11.8}) is true,
which will be done in the following. We divide the proof into the
following two cases.

{\sc Case 1}. $\gamma=\gamma_1+\tilde\gamma_2\sqrt{\tilde D}$ is rational.
Then $\tilde\gamma_2=0$ and so $\gamma=\gamma_1$. Since $p$ is a prime
number such that (\ref{eq5.8.1}) holds, by (\ref{eq5.9.2}) we have
$$p>d(\gamma_1)(|\gamma_1|+3)\ge d(\gamma_1)(|\lfloor 1
-\gamma_1\rfloor|+\langle \gamma_1\rangle+1).$$
Since $ap\equiv 1\mod E$ and $d(\gamma_1)\mid E$,
it then follows from Lemma \ref{lem2.3} that
$$\mathfrak{D}_p^l(\gamma_1)=\langle a^l\gamma_1\rangle\ge \frac{1}{d(\gamma_1)}$$
holds for all positive integers $l$.
Thus for all integers $l$ with $l\ge 2$, we have
$$T_{p,l}(\gamma_1)=p^l\mathfrak{D}_p^l(\gamma_1)-\gamma_1
\ge \frac{p^2}{d(\gamma_1)}-\gamma_1>p(|\gamma_1|+3)-\gamma_1>p>T_p(\gamma_1).$$
So for any prime number $p$ such that (\ref{eq5.8.1}) holds,
one has $p\not\in \mathcal{P}_{\gamma}$. Hence the set
$\mathcal{P}_{\gamma}$ is finite. Therefore (\ref{eq5.11.8}) holds
if $\gamma$ is a rational number. So the claim is true in this case.

{\sc Case 2}. $\gamma=\gamma_1+\tilde\gamma_2\sqrt{\tilde D}$ is
irrational. We let
$$f_{\gamma}(z):=(x+\gamma_1+\tilde\gamma_2\sqrt{\tilde D})
(x+\gamma_1-\tilde\gamma_2\sqrt{\tilde D})
=z^2+2\gamma_1z+\gamma_1^2-\tilde{D}\tilde\gamma_2^2.$$
Then $f_{\gamma}(z)$ is a monic irreducible polynomial over $\mathbb{Q}$.
For any prime number $p$ such that $p\in\mathcal{S}$ and (\ref{eq5.8.1}) holds,
since $ap\equiv 1\mod E$ and $d(\gamma_1)\mid E$, we have $v_p(\gamma_1)\ge 0$.
By (\ref{eq5.9.4}), we have $v_p(\gamma_1^2-\tilde{D}\tilde\gamma_2^2)=0$.
Hence $f_{\gamma}(z)\in\mathbb{Z}_p[z]$. Since
$$f_{\gamma}(-\gamma_1-\tilde\gamma_2\sigma_p(\sqrt{\tilde D}))
=f_{\gamma}(-\gamma_1+\tilde\gamma_2\sigma_p(\sqrt{\tilde D}))=0,$$
it follows that $-\gamma_1-\tilde\gamma_2\sigma_p(\sqrt{\tilde D})$ and
$-\gamma_1+\tilde\gamma_2\sigma_p(\sqrt{\tilde D})$ are both $p$-adic integers.
But
$$v_p(\gamma_1+\tilde\gamma_2\sigma_p(\sqrt{\tilde D}))
+v_p(\gamma_1-\tilde\gamma_2\sigma_p(\sqrt{\tilde D}))
=v_p(\gamma_1^2-\tilde{D}\tilde\gamma_2^2)=0.$$
So we have
\begin{align} \label{eq5.11.10}
v_p(\gamma_1+\tilde\gamma_2\sigma_p(\sqrt{\tilde D}))
=v_p(\gamma_1-\tilde\gamma_2\sigma_p(\sqrt{\tilde D}))=0.
\end{align}

Define
$$\mathcal{F}_{\gamma}:=\Big\{\frac{T_p(\sigma_p(\gamma))}{p}: p\in\mathcal{S} \
{\rm and} \  \sigma_p\in{\rm Hom}(K,\mathbb{C}_p)\ {\rm such \ that} \
T_{p,2}(\sigma_p(\gamma))=T_p(\sigma_p(\gamma)) \Big\}.$$
On the one hand, for every prime number $p$ with $p\in \mathcal{P}_{\gamma}$,
there is one $\sigma_p\in{\rm Hom}(K,\mathbb{C}_p)$ such that
$T_{p,2}(\sigma_p(\gamma))=T_p(\sigma_p(\gamma))$.
It follows that for every prime number $p$ with $p\in \mathcal{P}_{\gamma}$,
one can find at least one $\sigma_p\in{\rm Hom}(K,\mathbb{C}_p)$ such that
$$\frac{T_p(\sigma_p(\gamma))}{p}\in \mathcal{F}_{\gamma}.$$
But (\ref{eq5.11.10}) tells us that
$$
T_p(\sigma_p(\gamma))\equiv -\sigma_p(\gamma)\equiv
(\gamma_1\pm \tilde \gamma_2(\sqrt{\tilde D}))\not\equiv 0 \mod p.
$$
Hence $\frac{T_p(\sigma_p(\gamma))}{p}\ne 0$ and so we have
\begin{equation} \label{eq5.11.11}
\#\{p\in \mathcal{P}_{\gamma}: p<x\}\le \#\Big\{\frac{T_p(\sigma_p(\gamma))}{p}
\in \mathcal{F}_{\gamma}: p<x\Big\}.
\end{equation}

On the other hand, since
$$T_{p,l}(\sigma_p(\gamma))+\sigma_p(\gamma)\equiv 0\mod p^l$$
and $f_{\gamma}(-\sigma_p(\gamma))=0$, it follows that
$$f_{\gamma}(T_{p,l}(\sigma_p(\gamma)))\equiv 0\mod p^l$$
holds over $\mathbb{Z}_p$ for all positive integers $l$
and all $\sigma_p\in{\rm Hom}(K,\mathbb{C}_p)$.
So for any $\frac{T_p(\sigma_p(\gamma))}{p}\in\mathcal{F}_{\gamma}$,
i.e., for any $p\in \mathcal{S}$ and $\sigma_p\in{\rm Hom}(K,\mathbb{C}_p)$ such that
$T_p(\sigma_p(\gamma))=T_{p,2}(\sigma_p(\gamma))$, we have over $\mathbb{Z}_p$ that
$$f_{\gamma}(T_p(\sigma_p(\gamma)))=f_{\gamma}(T_{p,2}(\sigma_p(\gamma)))\equiv 0\mod p^2.$$
Since $$d(\gamma_1)\mid d_{\bm{\alpha_1,\beta_1}}|E,
d(\tilde{D})\mid d_{\bm{\alpha_2,\beta_2}}|E$$
and $p$ is coprime to $E$,
there is a positive integer $c$ which is coprime to $p$ such that
$cf_{\gamma}(z)$ is a primitive polynomial with integer coefficients. Then
$$cf_{\gamma}(T_p(\sigma_p(\gamma)))=cf_{\gamma}(T_{p,2}(\sigma_p(\gamma)))\equiv 0\mod p^2. $$
Thus for any $T_p(\sigma_p(\gamma))/p\in\mathcal{F}_{\gamma}$, $T_p(\sigma_p(\gamma))/p$
is belonging to the following set:
$$\mathcal{F}_{\gamma}':=\Big\{\frac{v}{p}: p\in\mathcal{S},
0\le v<p, cf_{\gamma}(v)\equiv 0\mod p^2\Big\}.$$
Hence $\mathcal{F}_{\gamma}\subset \mathcal{F}_{\gamma}'$ and
\begin{align}
\#\Big\{\frac{T_p(\sigma_p(\gamma))}{p}\in \mathcal{F}_{\gamma}: p<x\Big\}
\le \#\Big\{\frac{v}{p}\in\mathcal{F}': p<x \Big\}.
\end{align}
But Lemma \ref{lem4.4} (ii) applied to $\mathcal{F}_{\gamma}'$ gives us that
\begin{align} \label{eq5.11.12}
\#\Big\{\frac{v}{p}\in\mathcal{F}': p<x \Big\}=o(\pi(x)).
\end{align}
Then by (\ref{eq5.11.11}) to (\ref{eq5.11.12}), one derives that
(\ref{eq5.11.8}) holds if $\gamma$ is irrational. So (\ref{eq5.11.8})
is true in this case. This finishes the proof of the claim, and hence
the proof of the necessities of parts (ii) and (iii).

The proof of Theorem \ref{thm1.4} is complete.  \hfill$\Box$\\


In the conclusion of this paper, we supply the following example
to illustrate our main result Theorem \ref{thm1.4}.\\
\\
{\bf Example 5.11.} Let $D$ be any square-free integer other than 1 and let
$$\bm{\alpha}=\big(\sqrt{D}, -\sqrt{D}, \frac{1}{2}+\sqrt{D}, \frac{1}{2}-\sqrt{D}\big)
\ {\rm and}\ \bm{\beta}=(2\sqrt{D},-2\sqrt{D}).$$
We denote by $F_{\bm{\alpha,\beta}}(z)$ the hypergeometric series
associated to $\bm{\alpha}$ and $\bm{\beta}$, that is,
\begin{equation*}
F_{\bm{\alpha},\bm{\beta}}(z):
=\sum_{n=0}^{\infty}\frac{(\sqrt{D})_n (-\sqrt{D})_n
(\frac{1}{2}+\sqrt{D})_n (\frac{1}{2}-\sqrt{D})_n}
{(2\sqrt{D})_n (-2\sqrt{D})_n}z^n.
\end{equation*}
Let
$$\alpha_1(a,\epsilon):=\epsilon, \alpha_2(a,\epsilon):=-\epsilon,
\alpha_3(a,\epsilon):=\frac{a}{2}+\epsilon,
\alpha_4(a,\epsilon):=\frac{a}{2}-\epsilon$$
and $\beta_1(a,\epsilon)=2\epsilon,\beta_2(a,\epsilon)=-2\epsilon.$
Then $u=v=0$, $E=4D$ and
\begin{align} \label{eq5.12.1}
\Delta_{\bm{\alpha,\beta}}(x,a,\epsilon)
=\#\{1\le i\le 4: \alpha_i(a,\epsilon)\preccurlyeq x\}
-\#\{1\le j\le 2: \beta_j(a,\epsilon)\preccurlyeq x\}.
\end{align}
Since $u=v=0$, $\delta_{\bm{\mu,\nu}}(x,a)=0$ for all $x\in \mathbb{R}$
and $a\in(\mathbb{Z}/E\mathbb{Z})^{\times}$. So the statements
II and III are satisfied. We claim that
\begin{align} \label{eq5.12.2}
\Delta_{\bm{\alpha,\beta}}(\beta_j(a,\epsilon),a,\epsilon)\ge 0
\end{align}
holds for all $x\in\mathbb{R}$, $a\in(\mathbb{Z}/E\mathbb{Z})^{\times}$
and $j=1,2$. Then by Lemma 2.5, we have
$\Delta_{\bm{\alpha,\beta}}(x,a,\epsilon) \ge 0$
for all $x\in\mathbb{R}$, $a\in(\mathbb{Z}/E\mathbb{Z})^{\times}$
and $\epsilon\in (0,1)$. Hence the statement IV is true, and so
by Theorem \ref{thm1.4} (iii), we know that $F_{\bm{\alpha,\beta}}(z)$
is N-integral. In what follows, we show the claim. We divide
the proof of the claim into the following three cases.

{\sc Case 1.} $0< \epsilon<\frac{1}{2}$. Then
$\langle\epsilon\rangle< \langle 2\epsilon\rangle$
and $\langle \frac{a}{2}-\epsilon\rangle< \langle-2\epsilon\rangle$
since the hypothesis that $a$ is coprime to $E=4D$ implies that
$a$ is an odd positive integer.
Then $\alpha_1(a,\epsilon)\preccurlyeq \beta_1(a,\epsilon)$
and $\alpha_4(a,\epsilon)\preccurlyeq \beta_2(a,\epsilon)$.

If $\beta_1(a,\epsilon)\preccurlyeq \beta_2(a,\epsilon)$, then
$$\alpha_1(a,\epsilon)\preccurlyeq \beta_1(a,\epsilon)\preccurlyeq\beta_2(a,\epsilon)
\ {\rm and} \ \alpha_4(a,\epsilon)\preccurlyeq \beta_2(a,\epsilon).$$
It follows that (\ref{eq5.12.2}) holds
for both $j=1$ and $2$. So the claim is true in this case.

If $\beta_2(a,\epsilon)\preccurlyeq \beta_1(a,\epsilon)$, then
$$\alpha_4(a,\epsilon)\preccurlyeq \beta_2(a,\epsilon)\preccurlyeq \beta_1(a,\epsilon)
\ {\rm and} \ \alpha_1(a,\epsilon)\preccurlyeq \beta_1(a,\epsilon).$$
Thus (\ref{eq5.12.2}) is true for both $j=1$ and $2$. That is,
the claim is true in this case. So the claim is proved when
$\epsilon\in (0,\frac{1}{2})$.

{\sc Case 2.} $\frac{1}{2}<\epsilon<1$. Then one can easily check that
$\langle-\epsilon\rangle< \langle -2\epsilon\rangle$
and $\langle \frac{a}{2}+\epsilon \rangle< \langle 2\epsilon\rangle$.
It then follows that $\alpha_2(a,\epsilon)\preccurlyeq \beta_2(a,\epsilon)$
and $\alpha_3(a,\epsilon)\preccurlyeq\beta_1(a,\epsilon)$.

If $\beta_1(a,\epsilon)\preccurlyeq \beta_2(a,\epsilon)$, then
$$\alpha_3(a,\epsilon)\preccurlyeq\beta_1(a,\epsilon)\preccurlyeq \beta_2(a,\epsilon)
\ {\rm and} \ \alpha_2(a,\epsilon)\preccurlyeq\beta_2(a,\epsilon).$$
This tells us that (\ref{eq5.12.2}) holds for both $j=1$ and $2$.
Hence the claim is true in this case.

If $\beta_2(a,\epsilon)\preccurlyeq \beta_1(a,\epsilon)$, then
$$\alpha_2(a,\epsilon)\preccurlyeq\beta_2(a,\epsilon)\preccurlyeq \beta_1(a,\epsilon)
\ {\rm and} \ \alpha_3(a,\epsilon)\preccurlyeq\beta_1(a,\epsilon).$$
Therefore (\ref{eq5.12.2}) is true for both $j=1$ and $2$, i.e.,
the claim is true in this case.
Thus the claim is proved when $\epsilon\in (\frac{1}{2},1)$.

{\sc Case 3.} $\epsilon=\frac{1}{2}$. Then
$\langle\epsilon\rangle=\langle -\epsilon\rangle
<\langle 2\epsilon\rangle= \langle -2\epsilon\rangle$
and $\epsilon >-\epsilon$ and $2\epsilon >-2\epsilon$.
It then follows that
$\epsilon\preccurlyeq -\epsilon\preccurlyeq 2\epsilon\preccurlyeq -2\epsilon$.
Namely,
$$\alpha_1(a,\epsilon)\preccurlyeq\alpha_2(a,\epsilon)
\preccurlyeq \beta_1(a,\epsilon)\preccurlyeq \beta_2(a,\epsilon).$$
Thus (\ref{eq5.12.2}) holds for $j=1$ and $2$.
In other words, the claim is true if $\epsilon =\frac{1}{2}$.

Hence the claim is proved and so $F_{\bm{\alpha,\beta}}(z)$
is N-integral. \hfill$\Box$


\end{document}